\documentclass{amsart}
\usepackage{graphicx}
\usepackage{amssymb}
\usepackage{epstopdf}
\usepackage{nicefrac}
\usepackage{enumerate}
\usepackage{verbatim}
\usepackage{hyperref}
\usepackage{color}
\usepackage{pst-all}
\usepackage{tikz}
\usepackage{enumitem}

\newtheorem{thm}{THEOREM}[section]
\newtheorem{conj}[thm]{CONJECTURE}
\newtheorem{cor}[thm]{COROLLARY}
\newtheorem{definition}[thm]{DEFINITION}
\newtheorem{defn}[thm]{DEFINITION}
\newtheorem{ex}[thm]{EXAMPLE}

\newtheorem{lemma}[thm]{LEMMA}
\newtheorem{prob}[thm]{PROBLEM}
\newtheorem{prop}[thm]{PROPOSITION}

\newtheorem{remark}[thm]{REMARK}

\newcommand{\ds}{\displaystyle}

\newcommand{\mC}{{\mathbb C}}

\newcommand{\mN}{{\mathbb N}}
\newcommand{\mQ}{{\mathbb Q}}

\newcommand{\mZ}{{\mathbb Z}}

\newcommand{\ZZ}{\mathbb Z}

\newcommand{\cA}{{\mathcal A}}

\newcommand{\cG}{{\mathcal G}}
\newcommand{\cH}{{\mathcal H}}

\newcommand{\cK}{{\mathcal K}}

\newcommand{\cM}{{\mathcal M}}

\newcommand{\fG}{\mathfrak{G}}

\def\Aut{{\rm Aut}}
\def\Homeo{{\rm Homeo}}
 
 \def\gcd{{\rm gcd}}
\input xy
\xyoption {all}

\usepackage{etoolbox}
 
\makeatletter
\providerobustcmd*{\bigcupdot}{%
   \mathop{%
     \mathpalette\bigop@dot\bigcup
   }%
}
\newrobustcmd*{\bigop@dot}[2]{%
   \setbox0=\hbox{$\m@th#1#2$}%
   \vbox{%
     \lineskiplimit=\maxdimen
     \lineskip=-0.7\dimexpr\ht0+\dp0\relax
     \ialign{%
       \hfil##\hfil\cr
       $\m@th\cdot$\cr
       \box0\cr
     }%
   }%
}
\makeatother

\begin{document}

\title{Settled elements in profinite groups}

\begin{abstract}
{ Given a polynomial of degree $d$ over a number field, the image of the associated arboreal representation of the absolute Galois group of the field is a profinite group acting on the $d$-ary tree. Boston and Jones conjectured that for a quadratic polynomial, the image of such a representation contains a dense set of settled elements. Here an element is settled if it exhibits a certain pattern of growth of cycles at finite levels of the tree. In this paper, we prove the conjecture of Boston and Jones generically in the case when the quadratic polynomial has a strictly pre-periodic post-critical orbit of length $2$, and provide new evidence that the conjecture holds for quadratic polynomials with strictly pre-periodic post-critical orbits of length at least $3$. To prove our results, we introduce a new dynamical method, which uses the notions of a maximal torus and its Weyl group. These notions are analogous to the notions of maximal tori and Weyl groups in the theory of compact Lie groups, where they are fundamental. For profinite groups, maximal tori and Weyl groups contain the information about settled elements, and this is the foundation of our method. 
}
\end{abstract}

\author{Mar\'{\i}a Isabel Cortez}
\email{maria.cortez@mat.uc.cl}
\address{Facultad de Matem\'aticas, Pontificia Universidad Cat\'olica de Chile, Edificio Rolando Chuaqui, Vicuña Mackenna 4860, Stgo-Chile}

\author{Olga Lukina}
 \email{olga.lukina@univie.ac.at}
\address{Faculty of Mathematics, University of Vienna, Oskar-Morgensternplatz 1, 1090 Vienna, Austria}

 \date{}


\thanks{Version date: June 15, 2021. Revision: April 11, 2022.}

\thanks{The research of M-I. Cortez was supported by proyecto Fondecyt Regular No. 1190538; the research of O. Lukina is supported by the FWF Project P31950-N35.}

\date{}


\maketitle

\section{Introduction}

An arboreal representation of the absolute Galois group  of a number field is defined by a polynomial, and it is a profinite group acting on a rooted tree. Boston and Jones \cite{BosJon07,BosJon09,BosJon12} hypothesized that the image of a Frobenius element under such a representation has the structure of a \emph{settled element} (see Definition \ref{defn-settled}), and formulated a conjecture (see Conjecture \ref{conj-BJ} below) that settled elements are dense in the image of an arboreal representation given by a quadratic polynomial over $\mathbb{Q}$. We consider this conjecture in a wider framework of polynomials over algebraic number fields, and prove that it holds generically for Chebyshev polynomials in Theorems \ref{thm-main2} and \ref{thm-main25}. We provide evidence towards the statement that the conjecture holds generically for representations associated to post-critically finite quadratic polynomials with strictly pre-periodic post-critical orbits of arbitrary length in Theorems \ref{thm-main3} - \ref{thm-main4}. 

The density of settled elements in the image of an arboreal representation associated to a quadratic polynomial over $\mathbb{Q}$ was studied by Boston and Jones \cite{BosJon07,BosJon09,BosJon12}, Goksel, Xia and Boston \cite{GXB2015}, Goksel \cite{Gok2019,Gok2019-1,GokThesis} using factorizations of polynomials modulo primes, and Markov processes. To prove our results, we introduce a new technique, which uses the dynamical and group-theoretical properties of arboreal representations rather than Number Theory. Our main tools are maximal abelian subgroups, called the \emph{maximal tori} (see Definition \ref{defn-maxtorus}), and their \emph{Weyl groups} (see Definition \ref{defn-Weyl}). These notions originate in the theory of compact Lie groups, where they have a fundamental role.  Our application of the Weyl group framework  to the study of arboreal representations, or  more generally groups acting on trees,  is one of the novel concepts in  this work. In the process of proving our results,  we also develop some properties of maximal tori and their Weyl groups in the setting of groups acting on Cantor sets.

The rest of the introduction expands on each of these themes above. In Section \ref{subsec-arbsettled} we   outline the construction of an arboreal representation, given a polynomial, and discuss the problem of settled elements in arboreal representations. In Section \ref{subsec-toriWeyl} we introduce maximal tori and Weyl groups, and discuss their properties. We present our main results and discuss open problems in Section \ref{subsec-mainresults}.

\bigskip

\subsection{Settled elements and arboreal representations}\label{subsec-arbsettled} An arboreal representation is a map into the automorphism group of a spherically homogeneous rooted tree, which we briefly define now.

A \emph{spherically homogeneous rooted tree} is an infinite rooted tree $T$ with the vertex set $\bigcupdot_{n \geq 0} V_n$ such that $V_0$ contains a single vertex called the \emph{root}, and for each $n \geq 1$, each vertex $v \in V_{n-1}$ at level $n-1$ is joined to the same number $m_n\geq 1$ of vertices $w \in V_n$ at level $n$, see Section \ref{rooted} for details. We denote by $\partial T$ the set of all infinite paths of edges without backtracking in $T$, which is a Cantor set. We denote by $\Aut(T)$ the set of automorphisms of $T$, that is, all maps which preserve the structure of the tree, see Section \ref{subsec-group} for details. 

We call a spherically homogeneous rooted tree a $d$-\emph{ary tree} if $m_n = d$ for all $n \geq 1$, where $d \geq 2$ is a positive integer. For $k>n$ we say that $w \in V_{k}$ \emph{lies above} $v \in V_n$ if there is a path of edges which joins $v$ and $w$, and which contains precisely one vertex at each level $V_n,V_{n+1},\ldots, V_k$. 

The following three definitions can be given for a spherically homogeneous rooted tree but we formulate them for a $d$-ary tree $T$ for simplicity. We remark that for $\sigma \in \Aut(T)$ the restriction $\sigma|V_n$, for $n \geq 1$, is a permutation of a set of $d^n$ elements, and so we can talk about the cycle decomposition of $\sigma|V_n$.

\begin{definition}\label{defn-cycles}\cite{BosJon09}
Let $T$ be a $d$-ary tree, where $d \geq 2$ is a positive integer, and let $\sigma\in \Aut(T)$. Let $n\geq 1$ and let $v\in V_n$ be a vertex. 

We say that $v$ is in a \emph{cycle} of length $k\geq 1$ of $\sigma$ if $\sigma^k(v)=v$ and $\sigma^{j}(v)\neq v$ for every $1\leq j<k$. 

We say that $v$ is in a \emph{stable cycle} of length $k$ if $v$ is in a cycle of length $k$ and for every $m>n$ all the vertices in $V_m$ which lie above $\{v, \sigma(v), ..., \sigma^{k-1}(v)\}$ are in the same cycle of length $d^{m-n}k$ in $V_m$.
\end{definition}

In more detail, for a cycle $\tau$ of length $k$ in $\sigma|V_n$ there are $kd$ vertices in $V_{n+1}$  
which lie above the vertices in $\tau$. These vertices may be in at most $d$ disjoint cycles for the action of $\sigma|V_{n+1}$. If $\tau$ is a stable cycle, then for $m \geq 1$ the $kd^m$ vertices in $V_{n+m}$ which lie above the vertices in $\tau$ are in a single cycle.

\begin{definition}\label{defn-settled}\cite{BosJon09}
Let $T$ be a $d$-ary tree, where $d \geq 2$ is a positive integer. We say that $\sigma\in \Aut(T)$ is \emph{settled} if 
$$
\lim_{n\to \infty}\frac{|\{v\in V_n: v \mbox{ is in a stable cycle}\}|}{|V_n|}=1.
$$
We say that $\sigma\in \Aut(T)$ is \emph{strongly settled} if there exists $n\geq 1$ such that every vertex in $V_n$ is in a stable cycle.
\end{definition}

\begin{definition}\label{defn-denselysettled}
Let $T$ be a $d$-ary tree, where $d \geq 2$ is a positive integer. We say that a profinite subgroup $\cG \subset \Aut(T)$ is \emph{densely settled} if the set of settled elements of $\cG$ is dense in $\cG$. 
\end{definition}

Examples of settled and strongly settled elements are given in Section \ref{sec-stablesettled}. The simplest example of a densely settled group is $\Aut(T)$ itself, see Section \ref{subsec-AutTsettled}. The question whether an arbitrary profinite subgroup of $\Aut(T)$ is densely settled is considerably more complicated.  

\bigskip
We now define arboreal representations and discuss Conjecture \ref{conj-BJ} which is due to Boston and Jones.

Let $K$ be an algebraic number field, and denote by $\overline K$ a separable closure of $K$. Given a polynomial $f(x)$ of degree $d \geq 2$ with coefficients in $K$, an \emph{arboreal representation} of $\overline K$ associated to $f(x)$ is a homomorphism $\rho_{f,\alpha}: {\rm Gal}(\overline K/K) \to \Aut(T)$, where $T$ is a $d$-ary tree, and $\alpha \in K$, see Section \ref{sec-arboreal}. A similar construction over the field of rational functions $K(t)$ where $t$ is transcendental, gives a subgroup $\fG_{\rm arith}(f) \subset \Aut(T)$, called the \emph{profinite arithmetic iterated monodromy group} associated to $f(x)$. The arboreal representation $\rho_{f,\alpha}$ for a given $\alpha \in K$ is related to $\fG_{\rm arith}(f)$ via a \emph{specialization}, and $\fG_{\rm arith}(f)$ is thought of as the image of $\rho_{f,\alpha}$ for a generic choice of $\alpha$. The \emph{profinite geometric iterated monodromy group} $\fG_{\rm geom}(f)$ is a normal subgroup of $\fG_{\rm arith}(f)$, which is known to be isomorphic to the closure of the discrete iterated monodromy group of $f(x)$ in Geometric Group Theory, see \cite[Chapter 5]{Nekr} for a discussion of discrete iterated monodromy groups. 

Arboreal representations of Galois groups have been extensively studied in the recent years, but there still remain many open questions, see \cite{Jones2013,Betal2019} for recent surveys. One of the main questions considered in the literature is, what are the conditions under which the image of $\rho_{f,\alpha}$ has finite index in $\Aut(T)$? This problem is motivated by questions about the density of prime divisors in the forward orbits of points under iterations of a polynomial $f(x)$, a problem first studied by Odoni \cite{Odoni1985}. The methods used to study arboreal representations in the literature are mostly number-theoretical, although techniques from Geometric Group Theory were used by Jones \cite{Jones2015}, Pink \cite{Pink13} and the second author \cite{Lukina2021}. Jones \cite{Jones2008,Jones2013,Jones2015} also applied the theory of stochastic processes to arboreal representations. Some topological dynamical properties of arboreal representations and profinite iterated monodromy groups were studied by the second author in \cite{Lukina2019,Lukina2021}. 

Boston and Jones formulated the following conjecture in \cite{BosJon07,BosJon09,BosJon12}.

\begin{conj}\cite{BosJon07,BosJon09,BosJon12}\label{conj-BJ}
Let $f(x) \in \mQ[x]$ be a quadratic polynomial, let $\alpha \in \mQ$, let $T$ be the binary tree and let $\rho_{f,\alpha}: {\rm Gal}(\overline \mQ/\mQ) \to \Aut(T)$ be the associated arboreal representation. Then the set of settled elements is dense in the image of $\rho_{f,\alpha}$.
\end{conj}

As explained in \cite{BosJon07,BosJon09,BosJon12}, the conjecture is motivated by the study of the images of the Frobenius elements in the Galois group under $\rho_{f,\alpha}$. We refer to \cite{BosJon07,BosJon09,BosJon12} for details and the intuitive ideas behind the conjecture, and also for a discussion of its significance for the development of a theory of arboreal representations similar to that of linear representations of absolute Galois groups of fields. Of course, the densely settled property can be studied for any profinite subgroup of $\Aut(T)$, where $T$ is a spherically homogeneous tree (see also the discussion and remarks in \cite[p.30]{BosJon07}, \cite[p.223]{BosJon09}). This motivates the following question.

\begin{prob}\label{prob-main}
Let $K$ be an algebraic number field, let $f(x)$ be a polynomial of degree $d \geq 2$ with coefficients in $K$, let $\alpha \in K$. For which arboreal representations $\rho_{f,\alpha}$ is the set of settled elements dense in the image of $\rho_{f,\alpha}$?
\end{prob}

As it was remarked already in \cite{BosJon07}, the images of arboreal representations which have finite index in $\Aut(T)$ are densely settled since $\Aut(T)$ is densely settled. We show in Remark \ref{remark-settled} that if the image of an arboreal representation is isometrically isomorphic to the infinite wreath product of finite groups satisfying an additional condition, then it is densely settled (for instance, some of the profinite iterated monodromy groups of the normalized Belyi polynomials in \cite{BEK2018} belong to this class). Among the most difficult and interesting cases is that of the arboreal representations associated to post-critically finite (PCF) polynomials. 

\begin{definition}\label{defn-pcf}
Let $f(x)$ be a quadratic polynomial, and let $c$ be the critical point of $f(x)$. The \emph{post-critical orbit} is the orbit $P_c = \{f^n(c) \mid n \geq 1\}$ of the critical point $c$.  The quadratic polynomial $f(x)$ is \emph{post-critically finite (PCF)} if $P_c$ is a finite set. 
\end{definition}

For quadratic polynomials, there are two possibilities for the dynamics on $P_c$: either the post-critical orbit is periodic and consists of a single cycle of length $\# P_c$, or it is strictly pre-periodic, with a periodic cycle of length strictly smaller than $\#P_c$. For instance, if $f(x) \in \mZ[x]$ is PCF, then either $f(x)$ has a periodic post-critical orbit of length $1$ or $2$, or $f(x)$ has a strictly pre-periodic post-critical orbit of length $2$ with a cycle of length $1$ \cite{BosJon09}. For the cases when there is a cycle of length $1$, it was shown in \cite{BosJon09} that the corresponding profinite arithmetic iterated monodromy group is densely settled. 
When the post-critical orbit is periodic of length $2$, then
$\fG_{\rm arith} (f)$ contains a subgroup conjugate to the Basilica group \cite{BosJon09,Pink13}. The question whether $\fG_{\rm arith} (f)$ is densely settled in this case is still open. 

If $K$ is an algebraic number field, then there are further conjugacy classes of quadratic polynomials, and one expects more diverse behavior. Pink \cite{Pink13} carried out an extensive study of the properties of the profinite arithmetic and geometric iterated monodromy groups for quadratic $f(x)$ over any number field $K$. Paper \cite{Pink13} gives an explicit representation of $\fG_{\rm geom}(f)$ using its topological generators, and studies the properties of the normalizers of these groups (which contain $\fG_{\rm arith}(f)$) in $\Aut(T)$. We use some of Pink's results in \cite{Pink13} in our study of profinite iterated monodromy groups in this paper.

\subsection{Maximal tori and Weyl groups}\label{subsec-toriWeyl} Let $T$ be a spherically homogeneous rooted tree, as defined in Section \ref{subsec-arbsettled} and, in more detail, in Section \ref{rooted}. Recall that $\Aut(T)$ is the group of automorphisms of $T$, that is, of maps of $T$ which take paths to paths and restrict to permutations at finite vertex levels $V_n$ of $T$, for $n \geq 1$.

We now introduce maximal tori and minimal elements, which we use in the proofs of our results.

\begin{definition}\label{defn-maxtorus}
Let $\cG$ be a compact subgroup of $\Aut(T)$. 
A \emph{maximal torus} is a maximal closed abelian subgroup $\cA$ of $\cG$ which acts freely and transitively on $\partial T$.
\end{definition}

Throughout the paper, for an element $a \in \Aut(T)$ we denote by $\langle a \rangle$ the cyclic subgroup of $\Aut(T)$ generated by $a$, and we denote by $\overline{\langle a \rangle}$ the closure of $\langle a \rangle$ in ${\rm Homeo}(\partial T)$. We study such procyclic groups in detail in Section \ref{procyclic_groups}.

\begin{definition}
We say that $a \in \Aut(T)$ is a \emph{minimal element}, if $\langle a \rangle$ acts transitively on each level $V_n$, $n \geq 1$. 

For a set $S \subset \Aut(T)$ we denote by ${\rm Min}(S)$ the subset of $S$ which consists of minimal elements.
\end{definition}

We show in Lemma \ref{minimal-elements} that if $a$ is minimal, then the closure $\overline{\langle a\rangle}$ in ${\rm Homeo}(\partial T)$ is a maximal torus. 

For a group $S$, we denote by $N(S)$ the normalizer of $S$ in $\Aut(T)$. We refer to Definitions \ref{defn-settled} and \ref{defn-denselysettled} for definitions of a settled element and a densely settled subgroup of $\Aut(T)$. In Sections \ref{sec_normalizer} and \ref{DenselyNormalizer}, we give a dynamical proof of the following statement.

\begin{thm}\label{thm-main1}
Let $T$ be a $d$-ary tree, for $d$ prime, and let $a \in \Aut(T)$ be a minimal element. Then the normalizer $N(\overline{\langle a\rangle})$ of the maximal torus $\overline{\langle a\rangle}$  in $\Aut(T)$
is densely settled.  
\end{thm}

Problem \ref{prob-main} asks if the subgroups of $\Aut(T)$ which arise as images of arboreal representations are densely settled. Theorem \ref{thm-main1} states that the normalizers of maximal tori in $\Aut(T)$ which are generated by minimal elements, contain many settled elements. Our approach to Problem \ref{prob-main} is to find out if the elements in the image of an arboreal representation can be approximated by settled elements contained in the normalizers of such maximal tori.

For this task, both the maximal tori and their normalizers are important. Minimal elements and their powers are strongly settled (see Example \ref{ex-minimal} and Lemma \ref{settled-iteration}), and so they have no fixed points in $V_n$ for $n$ sufficiently large. On the other hand, the normalizers of maximal tori generated by minimal elements contain settled elements with fixed points, see Section \ref{subsec-proofsNa}. In Section \ref{sec-img} we study the profinite iterated monodromy groups associated to quadratic PCF polynomials with strictly pre-periodic post-critical orbits, which may act on $V_n$ in a complicated way. To approximate the action of the elements in these groups on $V_n$, $n \geq 1$, we need settled elements with fixed points. In order to keep track of the settled elements in normalizers of maximal tori we introduce the \emph{Weyl groups} below. 

\begin{definition}\label{defn-Weyl}
Let $\cG \subset \Aut(T)$ be a compact group, and let $\cA \subset \cG$ be a maximal torus. Then:
\begin{enumerate}
\item The \emph{absolute Weyl group} of the maximal torus $\cA$ is the quotient $W(\cA) = N(\cA)/\cA $. 
\item The \emph{relative Weyl group}, or simply the \emph{Weyl group} of $\cA$ in $\cG$ is the quotient 
 $$W(\cA,\cG) = (N(\cA) \cap \cG) /\cA.$$
\end{enumerate}
\end{definition}

We now use the Weyl groups to formulate and prove our main results.  

\subsection{Main results}\label{subsec-mainresults}

We consider the case when $f(x)$ is a quadratic PCF polynomial (see Definition \ref{defn-pcf}) over an algebraic number field $K$ with strictly pre-periodic orbit of length $r \geq 2$, with periodic cycle of length $r-s$, for $s \geq 1$, $r,s \in \mN$. That is, $r$ and $s$ are the smallest positive integers such that $r>s \geq 1$ and $f^{r+1}(c) = f^{s+1}(c)$, where $c$ is the critical point of $f(x)$.

Recall from Section \ref{subsec-arbsettled} that, associated to $f(x)$, there are profinite geometric and arithmetic iterated monodromy groups, denoted by $\fG_{\rm geom}(f)$ and $\fG_{\rm arith}(f)$ respectively, and that we have the subgroup inclusions $\fG_{\rm geom}(f) \leq \fG_{\rm arith}(f) \leq N(\fG_{\rm geom}(f))$. Our results below show how the normalizers of the maximal tori in $\fG_{\rm geom}(f)$ and $N(\fG_{\rm geom}(f))$ are positioned inside these three groups. As a consequence, we make conclusions about the densely settled property of these three groups.

Recall from \cite{Pink13} that for $r = 2$ and $s=1$, the polynomial $f(x)$ is conjugate to the quadratic Chebyshev polynomial, and its profinite geometric iterated monodromy group $\fG_{\rm geom}(f)$ is conjugate in $\Aut(T)$ to the closure $\overline G$ of the action of the dihedral group $G = \langle m,n \mid n^2 = 1, nmn=m^{-1}\rangle$.

\begin{thm}\label{thm-main2}
Let $f(x)$ be a quadratic PCF polynomial over an algebraic number field $K$ with strictly pre-periodic post-critical orbit of length $r=2$, and consider $\fG_{\rm geom}(f)$ and $\fG_{\rm arith}(f)$. 

Then for any $a \in {\rm Min}(\fG_{\rm geom}(f))$ and the relative Weyl groups of the maximal torus $\overline{\langle a \rangle}$ in $\fG_{\rm geom}(f)$, $N(\fG_{\rm geom}(f))$ and $\fG_{\rm arith}(f)$ we have the following:
\begin{enumerate}
\item The relative Weyl group $W(\overline{\langle a \rangle},\fG_{\rm geom}(f))$ is finite of order $2$. 
\item We have the equality of the relative and the absolute Weyl groups 
 $$W(\overline{\langle a \rangle}, N(\fG_{\rm geom}(f))) = W(\overline{\langle a \rangle}).$$
\item The relative Weyl group $W(\overline{\langle a \rangle}, \fG_{\rm arith}(f))$ has finite index in the absolute Weyl group $W(\overline{\langle a \rangle})$. 
\end{enumerate}
\end{thm}

We interpret Theorem \ref{thm-main2} and similar results further in the paper as follows. The absolute Weyl group $W(\cA)$ of a maximal torus $\cA$ is infinite. If  $\cA < \cG$, where $\cG$ is a subgroup of $\Aut(T)$, then we may have three types of the relationship between the absolute Weyl group $W(\cA)$ and the relative Weyl group $W(\cA,\cG)$: either they are equal, $W(\cA,\cG) = W(\cA)$, or $W(\cA,\cG)$ has finite index in $W(\cA)$, or $W(\cA,\cG)$ has infinite index in $W(\cA)$. If $W(\cA,\cG) = W(\cA)$, then all settled elements which are in the normalizer $N(\cA)$ are also in $\cG$, and this \emph{provides good evidence} that $\cG$ contains many settled elements, and so it is likely to be densely settled. If $W(\cA,\cG)$ has finite index in $W(\cA)$, and so $W(\cA,\cG)$ is infinite, then this \emph{provides substantial evidence} that $\cG$ is densely settled. If $W(\cA,\cG)$ has infinite index in $W(\cA)$, then this \emph{provides little evidence} that $\cG$ is densely settled.

For $f(x)$ quadratic and $r = 2$ and $s = 1$ we the following result.

\begin{thm}\label{thm-main25}
Let $f(x)$ be a quadratic PCF polynomial over an algebraic number field $K$ with strictly pre-periodic post-critical orbit of length $r=2$. Then:
\begin{enumerate}
\item \label{st-one}The profinite geometric iterated monodromy group $\fG_{\rm geom}(f)$ is not densely settled. 
\item \label{st-two}The normalizer $N(\fG_{\rm geom}(f))$ is densely settled.
\item \label{st-three} The profinite arithmetic iterated monodromy group $\fG_{\rm arith}(f)$ is densely settled.
\end{enumerate}
\end{thm}

If $r \geq 3$ then $\fG_{\rm geom}(f)$ is conjugate to a branch group, see \cite[Definition 2.2 on p. 73]{Gri11} or \cite[Section 1.8]{Nekr} for a definition and the properties of branch groups.  The dynamics of a branch group acting on the boundary of a rooted tree can be very complicated. Such groups contain non-trivial elements which fix every point in an infinite collection of disjoint clopen subsets in $\partial T$. For profinite iterated monodromy groups associated to $f(x)$ for $r \geq 3$, we obtain results about maximal tori and the Weyl groups, which go part ways to the conjecture by Boston and Jones. We start by showing that there are many maximal tori generated by minimal elements in the normalizer $N(\fG_{\rm geom}(f))$.

\begin{thm}\label{thm-main3}
Let $f(x)$ be a quadratic PCF polynomial with strictly pre-periodic post-critical orbit of length $r\geq 3$ over an algebraic number field $K$.
Then for every coset in $N(\fG_{\rm geom}(f))/\fG_{\rm geom}(f)$ there is a representative $w$ such that:
\begin{enumerate}
\item For every $a \in {\rm Min}(\fG_{\rm geom}(f))$, the element $aw$ is in ${\rm Min}(w\fG_{\rm geom}(f))$. 
\item For a maximal torus $\overline{\langle aw \rangle}$, the intersection $\overline{\langle aw \rangle} \cap \fG_{\rm geom}(f)$ is an index $2$ subgroup of $\overline{\langle aw \rangle}$, and $\overline{\langle aw \rangle} \subset \fG_{\rm geom}(f) \cup  w \fG_{\rm geom}(f)$. 
\item The normalizer  $N(\fG_{\rm geom}(f))$ contains an uncountable number of distinct maximal tori, which may intersect along proper subgroups of finite index.
\item \label{thm3-4} Minimal elements in $\fG_{\rm geom}(f)$, and minimal elements in each coset $w \fG_{\rm geom}(f)$ in $N(\fG_{\rm geom}(f))$ are in bijective correspondence. 
\end{enumerate}
\end{thm}
 
Next we consider the properties of the relative Weyl groups in profinite arithmetic iterated monodromy groups. In order to do that we first look at the relative Weyl groups in the normalizer $N(\fG_{\rm geom}(f))$. There are two cases to distinguish: first, when a minimal $a$ is in $\fG_{\rm geom}(f)$, or, second,  when a minimal $g$ is in a coset $w\fG_{\rm geom}(f) \ne \fG_{\rm geom}(f)$ for some $w \in N(\fG_{\rm geom}(f))$. In the first case the results of \cite[Section 3]{Pink13} imply that the absolute Weyl group $W(\overline{\langle a \rangle})$ is contained in $N(\fG_{\rm geom}(f))$ and so we have the equality $W(\overline{\langle a \rangle}) = W(\overline{\langle a \rangle}, N(\fG_{\rm geom}(f)))$ of the absolute and the relative Weyl groups. For the second case we prove the following criterion of the equality of the absolute Weyl group $W(\overline{\langle g \rangle})$ and the relative Weyl group $W(\overline{\langle g \rangle}, N(\fG_{\rm geom}(f))$.

\begin{thm}\label{thm-main35}
Let $f(x)$ be a quadratic PCF polynomial with strictly pre-periodic post-critical orbit of length $r\geq 3$ over an algebraic number field $K$. Then there is a well-defined action 
  $$N(\fG_{\rm geom} (f)) \times {\rm Min}(w \fG_{\rm geom}(f)) \to {\rm Min}(w \fG_{\rm geom}(f)): (z,g) \mapsto z g z^{-1}.$$
If this action is transitive on ${\rm Min}(w \fG_{\rm geom}(f))$ then for any $aw \in {\rm Min}(w \fG_{\rm geom}(f))$ and the corresponding maximal torus $\overline{\langle a w \rangle}$ we have the equality of the Weyl groups
  $$W(\overline{\langle a w\rangle}, N(\fG_{\rm geom}(f))) = W(\overline{\langle a w \rangle}).$$
\end{thm}

We note that Theorems \ref{thm-main3} and \ref{thm-main35} exhibit the differences between the properties of maximal tori in compact connected Lie groups and in $\Aut(T)$, which are discussed in more detail in Theorem \ref{thm-differentthanLie}. To name a few such differences, the Weyl group of a maximal torus in a compact connected Lie group is always finite, while the relative Weyl group of a maximal torus in a subgroup of $\Aut(T)$ may be infinite. Another difference, in a compact Lie group all maximal tori are conjugate within the group, and every element in the group is contained in a maximal torus. Theorem \ref{thm-differentthanLie} shows that in the profinite geometric iterated monodromy groups associated to quadratic polynomials, either there are maximal tori which are not conjugate in the normalizer of this group or there are elements which are not contained in a maximal torus (possibly both). 

We collect the consequences of Theorems \ref{thm-main3} and \ref{thm-main35} for the densely settled property of profinite arithmetic iterated monodromy groups associated to quadratic polynomials in the theorem below.

\begin{thm}\label{thm-main4}
Let $f(x)$ be a quadratic PCF polynomial with strictly pre-periodic post-critical orbit of length $r\geq 3$ over a number field $K$, and let $\fG_{\rm geom}(f)$ and $\fG_{\rm arith}(f)$ denote the profinite geometric and arithmetic iterated monodromy groups associated to $f(x)$. Then the following holds:
\begin{enumerate}
\item For any $w \in  \fG_{\rm arith}(f)$ the sets ${\rm Min}(\fG_{\rm geom}(f))$ and ${\rm Min}( w\fG_{\rm geom}(f))$ are in bijective correspondence. Thus for every $w \in  \fG_{\rm arith}(f)$ the coset $w \fG_{\rm geom}(f)$ contains strongly settled elements. 
\item For any $a \in {\rm Min}(\fG_{\rm geom}(f))$, the index of the relative Weyl group $W( \overline{\langle a \rangle },\fG_{\rm geom}(f))$ in the absolute Weyl group $W( \overline{\langle a \rangle })$ is finite. Thus we have substantial evidence that the conjecture of Boston and Jones holds for $\fG_{\rm geom}(f)$. 
\item For any $a \in {\rm Min}(\fG_{\rm geom}(f))$, the index of the relative Weyl group $W( \overline{\langle a \rangle },\fG_{\rm arith}(f))$ in the absolute Weyl group $W( \overline{\langle a \rangle })$ is finite. Thus we have substantial evidence that the conjecture of Boston and Jones holds for $\fG_{\rm arith}(f)$. 
\end{enumerate}
\end{thm}

So far we have discussed the settled elements in $\fG_{\rm arith}(f)$ which arise as elements in normalizers of maximal tori topologically generated by minimal elements. We note that not necessarily all settled elements arise this way; other types of settled elements can be constructed using the method of Lemma \ref{density-1}. Thus a question, precisely which elements of $\Aut(T)$ can be approximated by settled elements in maximal tori and their Weyl groups, is natural. 

Branch groups, and so profinite iterated monodromy groups associated to quadratic polynomials, contain many elements which act trivially on infinite collections of clopen sets. Thus a successful solution to the following problem will constitute a big step towards a full proof of the conjecture by Boston and Jones in the case $r \geq 3$.

\begin{prob}\label{prob-clopensubset}
Let $g \in \Aut(T)$ be such that $g$ is the identity on a proper subset $U$ of $\partial T$, where $U$ is the finite or infinite union of clopen sets. Does there exists a sequence $\{g_n\}$, $n \geq 1$, which converges to $g$, and such that each $g_n$ is a settled element in the normalizer of a maximal torus?
\end{prob}

The proofs of Theorem \ref{thm-main3}, \ref{thm-main35} and \ref{thm-main4} rely on the descriptions of profinite iterated monodromy groups in \cite[Section 3]{Pink13}, and depend on the length and the structure of the post-critical orbit of a polynomial $f(x)$. The proofs do not directly extend to the case when $f(x)$ has a periodic post-critical orbit since the properties of profinite iterated monodromy groups in that case are very different. The following question remains open.

\begin{prob}\label{prob-basilica}
Let $f(x)$ be a quadratic PCF polynomial with periodic post-critical orbit of length $r\geq 2$ over an algebraic number field $K$. Study the maximal tori and their Weyl groups in the associated groups $\fG_{\rm geom}(f)$, $N(\fG_{\rm geom}(f))$ and $\fG_{\rm arith}(f)$.
\end{prob}

Finally, we expect that the properties of maximal tori and their Weyl groups may provide insights into the properties of groups acting on trees, and help classify them. Therefore, the following problem is natural.

\begin{prob}\label{prob-Weyl}
Study the maximal tori and the Weyl groups in profinite groups acting on trees. 
\end{prob}

We do not expect all statements about maximal tori and Weyl groups which are true for compact Lie groups, to carry over verbatim to our setting. Indeed, we have already described some differences in Theorem \ref{thm-differentthanLie}. However, we can always ask how far the similarities go. We remark that looking for notions and concepts for actions on trees, which are similar to those used in smooth dynamics has already proved fruitful, as in  \cite{HL2021}. 

The rest of the paper is organized as follows:  Sections  \ref{sec-arboreal}  and \ref{sec-background}  are devoted to recalling general concepts about arboreal representations and topological dynamics respectively.  In Section \ref{procyclic_groups} we study procyclic subgroups of $\Aut(T)$ from a dynamical point of view.  In Section \ref{sec_normalizer} we study the normalizer $N(\overline{\langle a \rangle})$ of a maximal torus $\overline{\langle a \rangle}$.  In Section \ref{sec-stablesettled} we give a dynamical characterization of settled elements of $\Aut(T)$, in terms of their minimal components, and we study the conditions under which the property of being densely settled is preserved. Finally, in Sections \ref{DenselyNormalizer}  and  \ref{sec-img} we use the previous results to show Theorem \ref{thm-main1} and    Theorems \ref{thm-main2}, \ref{thm-main3}, \ref{thm-main35}, \ref{thm-main4}, respectively.  

{\bf Acknowledgments:} The authors thank Steven Hurder for pointing out the similarities between abelian groups of tree automorphisms and their normalizers in our work, and the notions of the maximal torus and the Weyl group in theory of compact Lie groups. The authors thank Nicole Looper for useful comments on the draft of the paper, and the anonymous referees for the careful reading of the manuscript and providing valuable feedback.

\section{Representations of Galois groups}\label{sec-arboreal}

We briefly recall some background on arboreal representations and profinite iterated monodromy groups, see \cite{Jones2013} and \cite{Pink13} for more details. 

Let $K$ be a number field, that is, $K$ is a finite algebraic extension of the rational numbers $\mQ$. 
Let $f(x)$ be a polynomial of degree $d \geq 2$ with coefficients in $K$. Let $t$ be a transcendental element, then $K(t)$ is the field of rational functions with coefficients in $K$. 

We first define the profinite arithmetic iterated monodromy group. Consider the solutions of the equation $f^n(x) = t $ over $K(t)$. The polynomial $f^n(x)-t$ is separable and irreducible over $K(t)$ for all $n \geq 1$ \cite[Lemma 2.1]{AHM2005}, which means that all $d^n$ roots of $f^n(x) - t = 0$ are distinct, and $f^n(x) - t$ does not factor over $K(t)$. It follows that the Galois group $H_n$ of the extension $K_n$ obtained by adjoining to $K(t)$ the roots of $f^n(x)-t$ acts transitively on the roots. 

The tree $T$ has the vertex set $V = \bigcupdot_{n \geq 0} V_n$, where $V_0=\{t\}$, and $V_n$ are the sets of solutions of $f^n(x) = t$. We join $ \beta \in V_{n+1}$ and $\alpha \in V_n$ by an edge if and only if $f(\beta) = \alpha$. For each $n \geq 1$, the Galois group $H_n$ acts transitively on the roots of $f^n(x) = t$ by field automorphisms, and so induces a permutation of vertices in $V_n$. Since the field extensions satisfy $K_{n} \subset K_{n+1}$, we have a group homomorphism $\lambda^{n+1}_n:H_{n+1} \to H_n$. Taking the inverse limit
   $$\fG_{\rm arith}(f) = \lim_{\longleftarrow}\{\lambda^{n+1}_n: H_{n+1} \to H_n\},$$
we obtain a profinite group called the \emph{arithmetic iterated monodromy group} associated to the polynomial $f(x)$. For $n \geq 1$, the action of $H_n$ on $T$ preserves the connectedness of paths in $T$, and so $\fG_{\rm arith}(f)$ is identified with a subgroup of the automorphism group $\Aut(T)$ of the tree $T$. 

Denote by $\overline{K}$ a separable closure of $K$. An \emph{arboreal representation} of the absolute Galois group ${\rm Gal}(\overline K/K)$ is obtained by a construction similar to the one described above by choosing $\alpha \in K$ and setting $V_n = \{f^{-n}(\alpha)\}$, with $Y_n$ the Galois group of the extension $K(f^{-n}(\alpha))$. The image of the representation is the profinite group ${\ds Y_\infty = \lim_{\longleftarrow}\{Y_{n+1} \to Y_n\}}$. If $\mQ \subset K$, then the polynomial $f^n(x) - \alpha$ is separable. The extension $K(f^{-n}(\alpha))/K$ is normal by construction, and so it is Galois. The tree thus obtained is a $d$-ary tree, and so its automorphism group is isomorphic to $\Aut(T)$, for $T$ defined in the previous paragraph. Therefore, one may think of the groups $Y_n$ as obtained via the \emph{specialization} $t = \alpha$. As explained in \cite{Odoni1985}, Galois groups of polynomials do not increase under such specializations, and certain groups are preserved. In particular, $Y_\infty$ is a subgroup of $\fG_{\rm arith}(f)$, and one can think of $\fG_{\rm arith}(f)$ as $Y_\infty$ for a generic choice of $\alpha$, in a loose sense \cite{Odoni1985,Jones2013}.

A similar construction can be made with iterations of the same polynomial replaced by compositions of distinct polynomials $f_1 \circ \cdots \circ f_n$ such that these compositions are separable and irreducible for $n \geq 1$, see \cite{Ferr2018}.
 
Denote by $\cK$ the extension of $K(t)$ obtained by adding the roots of $f^{n}(x) - t$ for all $n \geq 1$. Let $L = \overline{K} \cap \cK$ be the maximal constant field extension of $K$ in $\cK$, that is, $L$ contains all elements of $\cK$ algebraic over $K$. Then the Galois group $\fG_{\rm geom}(f)$ of the extension $\cK/L(t)$ is a normal subgroup of $\fG_{\rm arith}(f)$, and there is an exact sequence \cite{Pink13,Jones2013}
 \begin{align}\label{exact-1} \xymatrix{ 1 \ar[r] & \fG_{\rm geom}(f) \ar[r] &  \fG_{\rm arith}(f) \ar[r] & \fG(L/K) \ar[r] & 1}. \end{align}
The profinite group $\fG_{\rm geom}(f)$ is called the \emph{geometric iterated monodromy group}, associated to the polynomial $f(x)$. The geometric monodromy group $\fG_{\rm geom}(f)$ does not change under extensions of $L$, so one can calculate $\fG_{\rm geom}(f)$ over $\mC(t)$. The profinite arithmetic iterated monodromy group $\fG_{\rm arith}(f)$ is contained in the normalizer $N = N(\fG_{\rm geom}(f))$ of $\fG_{\rm geom}(f)$.

A relation with iterated monodromy groups in geometric group theory \cite[Chapter 5]{Nekr} is given by the following construction.
Let $\mathbb{P}^1(\mathbb{C})$ be the projective line over $\mathbb{C}$ (the Riemann sphere), and extend the polynomial map $f: \mathbb{C} \to \mathbb{C}$ to $\mathbb{P}^1(\mathbb{C})$ by setting $f(\infty) = \infty$. Then $\infty$ is a critical point of $f(x)$. Let $C$ be the set of all critical points of $f(x)$, and let $P_{C} = \bigcup_{n \geq 1} f^n(C)$ be the set of the forward orbits of the points in $C$, called the \emph{post-critical set}. Suppose $P_{C}$ is finite, then the polynomial $f(x)$ is called \emph{post-critically finite}. 

\begin{remark}
{\rm
The definition of a post-critically finite quadratic polynomial in Definition \ref{defn-pcf} is slightly different to the one we just gave, but the two definitions are equivalent. Indeed, if $f(x)$ is quadratic, then the set $C$ of critical points of $f(x)$ on the Riemann sphere  consists of two points, the point $\infty$ and the critical point $c \in \mC$. Then $P_C = \{\infty\} \cup P_c$, where $P_c$ is the orbit of $c$ under the iterations of $f$. Thus $P_C$ is finite if and only if $P_c$ is finite.}
\end{remark}

If the polynomial $f(x)$ is post-critically finite, then it defines a partial $d$-to-$1$ covering $f: \cM_1 \to \cM$, where $\cM = \mathbb{P}^1(\mathbb{C}) \backslash P_C$ and $\cM_1 = f^{-1}(\cM)$ are spheres with finitely many punctures. An element $s \in \cM$ has $d$ preimages under $f$, and $d^n$ preimages under the $n$-th iterate $f^n$. Define a tree $\widetilde{T}$ so that the $n$-th level vertex set is $\{ f^{-n}(s)\}$, for $n \geq 1$, and there is an edge between $\alpha$ and $\beta$ if and only if $f(\beta) = \alpha$. The fundamental group $\pi_1(\cM, s)$ acts on the vertex sets  of $\widetilde{T}$ via path-lifting. Let ${\rm Ker}$ be the subgroup of $\pi_1(\cM,s)$ consisting of elements which act trivially on \emph{every} vertex level set. The quotient group ${\rm IMG}(f) = \pi_1(\cM,s)/{\rm Ker}$, called the \emph{discrete iterated monodromy group} associated to the partial self-covering $f: \cM_1 \to \cM$, acts on the space of paths of the tree $\widetilde{T}$. If $f(x)$ is post-critically finite, by \cite[Proposition 6.4.2]{Nekr} attributed by Nekrashevych to R. Pink, $\fG_{\rm geom}(f)$ is isomorphic to the closure of the action of ${\rm IMG}(f)$ in $\Aut(\widetilde{T})$.

The actions of discrete iterated monodromy groups ${\rm IMG}(f)$ are well-studied, with many results known and many techniques developed. 
We refer an interested reader to the book by Nekrashevych \cite{Nekr} for details.

\section{Definitions and background in dynamics}\label{sec-background}

\subsection{Dynamical systems}\label{subsec-dynsys} In this article a dynamical system is the action by homeomorphisms of a group $G$ on a compact metric space $X$. The action of an element $g\in G$ on $x\in X$ is denoted by $g x$. The {\it orbit} of $x \in X$ with respect to a subgroup $H$ of $G$ is denoted by $H x$  and it is defined as $Hx = \{y \in X \mid y = h x, \, h \in H\}$. If $H=\langle h \rangle $ for some $h \in G$ then we say that $H x$ is the $h$-orbit of $x$.  If $G=\langle g\rangle$ then the dynamical system given by the action of $G$ on $X$ is denoted by $(X, g)$.

The dynamical system given by the action of $G$ on $X$ is {\it minimal} is $G x$ is dense in $X$ for every $x\in X$. A {\it minimal component} of the dynamical system is a closed subset $Y$ of $X$ such that $Y$ is invariant under the action of $G$ and such that the dynamical system given by the restriction of the action of $G$ to $Y$ is minimal.  In the case that $G=\langle g\rangle$ we say that $g$ is minimal on $Y$, and if further $Y = X$ then we just say that $g$ is minimal (see Lemma \ref{minimal-transitivity} for an alternative characterization of minimality for actions on trees).  See \cite{Aus88} for more details about topological dynamics.

\subsection{Rooted trees}\label{rooted}
For every $n\geq 1$, let  $m_n>1$  be an integer and let $\Sigma_n=\{0,\cdots, m_n-1\}$. The space $\Sigma=\prod_{n\geq 1}\Sigma_n$ is a Cantor set if we endow every $\Sigma_n$ with the discrete topology and $\Sigma$ with the product topology. Then $\Sigma$ is a metric space  with the following metric: let $(x_n)_{n\geq 1}$ and $(y_n)_{n\geq 1}$ be two elements in $\Sigma$, then
\begin{align}\label{eq-metrictree}
d((x_n)_{n\geq 1}, (y_n)_{n\geq 1})=\frac{1}{2^{k-1}} \mbox{ if and only if } k=\min\{n\geq 1: x_n\neq y_n\}.
\end{align}
There is a natural way to identify the elements of $\Sigma$ with the set of infinite paths of a rooted tree.  The  {\it spherically homogeneous tree} associated to $\Sigma$  is the infinite graph $T_{\Sigma}=T$ defined as follows:  
\begin{itemize}
\item The set of vertices $V$ of the tree is equal to the disjoint union $\bigcupdot_{n\geq 0}V_n$, where $V_0$ contains only one element called the \emph{root} of the tree (we identify this element with the empty word) and $V_n=\prod_{i=1}^n\Sigma_i$ is the set of words of length $n$, for every $n\geq 1$.   
 \item  The set of edges $E$ of the tree is equal to the disjoint union $\bigcupdot_{n\geq 1}E_n$, where $E_n$ contains exactly one edge $e$ from $v\in V_{n-1}$ to $vl\in V_{n}$, for every $l\in \Sigma_{n}$ and $n\geq 1$. In this case we say that the source $s(e)$ of $e$ is equal to $v$ and the target $t(e)$ of $e$ is equal to $vl$.
\end{itemize} 
 The sequence $m_T = (m_1,m_2,\ldots)$ is called the \emph{spherical index} of $T$.
 
 \begin{remark}\label{remark-sphhom}
 {\rm
 We remark on the origin of the term the `spherically homogeneous tree' as described in \cite[p. XIX]{BORT1996}. A tree $T$ can be given a \emph{length structure} by assigning length $1$ to each edge. Such a length structure induces a length metric on $T$ (this is a different metric than the one in \eqref{eq-metrictree}), where the distance between two vertices $v$ and $w$ is equal to the infimum of lengths of all paths joining these vertices. Then in a spherically homogeneous rooted tree $T$, each vertex level $V_n$ is a sphere of radius $n$ centered at the root. The maps $p_{n+1}:V_{n+1} \to V_{n}$ that map each vertex $v \in V_{n+1}$ to the vertex $w \in V_{n}$ joined to $v$ by an edge, are surjective fibrations over the spheres $V_n$ whose fibres have constant cardinality, for $n \geq 0$. Thus the tree is spherically homogeneous.
 }
 \end{remark}

The set of infinite paths of $T$ is given by
$$
\partial T=\{ (e_n)_{n\geq 1}\in \prod_{n\geq 1}E_n: t(e_n)=s(e_{n+1}), \forall n\geq 1\}  \cong \{(x_n)_{n \geq 1} \in \Sigma \}. 
$$  
Here the element $(x_n)_{n\geq 1}\in \Sigma$ is identified  with $(e_n)_{n\geq 1}\in \partial T$ given by $t(e_1)=x_1$, $s(e_n)=x_1\cdots x_{n-1}$ and $t(e_n)=x_1\cdots x_n$, for every $n>1$. This identification induces a bijection between $\Sigma$ and $\partial T$,  which is an isometry with respect to the induced metric on $\partial T$.  Thus paths in $\partial T$ can be thought of as infinite sequences $x_1 x_2 \cdots$, where $x_n \in \Sigma_n$ for $n \geq 1$.

When $m_n=2$ for every $n\geq 1$ then we say that $T$ is the {\it binary rooted tree}. If $m_n = d$, for some $d > 2$ and all $n \geq 1$, then we say that $T$ is the {\it $d$-ary rooted tree}.

\subsection{Group of automorphisms}\label{subsec-group} An {\it automorphism } of $T$ is a transformation $\tau: V\to V$ such that $\tau(V_n)=V_n$ and if $x\in V_{n}$ is equal to $yt$ for $y\in V_{n-1}$ and $t\in \Sigma_n$, then there exists $t'\in \Sigma_n$ such that $\tau(x)=\tau(y)t'$, for every $n\geq 1$. In other words, an automorphism of $T$  permutes every $V_n$, $n \geq 1$,  preserving the structure of $T$, that is, two vertices in $V$ are joined by an edge if and only if their images under the automorphism of $T$ are joined by an edge. The collection of all the automorphisms of $T$ is a group denoted by $\Aut(T)$. Observe that an automorphism of $T$ defines an isometry on $\partial T$ with respect to the induced metric on $\partial T$.

\subsubsection{Automorphisms of the binary tree.}\label{subsec-binaryautom}  For the binary tree $T$ and a word $y = y_1 \cdots y_i$ denote by $T_y$ the connected subtree of $T$ with root $V_0^{(y)} = V_0$, the vertex level sets $V_n^{(y)} = \{y_1 \cdots y_n\} \subset V_n$ for $1 \leq n \leq i$, and $V_n^{(y)} = \{yx_{i+1} \cdots x_n \mid x_j \in \Sigma_j, i < j \leq n\}$ for $n > i$. Then the boundary $\partial T_y$ consists of all paths in $\partial T$ which start with the finite word $y$. There is a natural isomorphism $\kappa_y:T_y \to T$, which maps $yx_{i+1} \cdots x_n \in V_n^{(y)}$, for $n \geq i$, to $x_{i+1} \cdots x_n \in V_{n-i}$, and every infinite path $y x_{i+1} x_{i+2} \cdots \in \partial T_y$ to an infinite path $x_{i+1} x_{i+2} \cdots \in \partial T$. It follows that $\Aut(T_y)$ is naturally identified with $\Aut(T)$. 

Using these identifications, we can write $u = (u_0,u_1)$ with $u_0,u_1 \in \Aut(T)$ for $u \in \Aut(T)$, if the restriction $u|V_1$ is the trivial permutation. This means that $u$ acts as $u_0$ on the subtree that lies above the vertex $0$, and as $u_1$ on the subtree that lies above the vertex $1$. 
More generally, let $\eta$ be the non-trivial permutation of $V_1 = \{0,1\}$, and denote also by $\eta$ the automorphism of $T$ given by $i x \mapsto \eta(i) x$ for $i \in \{0,1\}$ and $ix \in V_n$, $n \geq 2$. Then for any $u \in \Aut(T)$ we can write $u = (u_0,u_1)\eta^m$, where $u_0,u_1 \in \Aut(T)$ and $m \in \{0,1\}$. These notation obeys the following simple relations:
 \begin{align}\label{eq-useful} (u_0,u_1) \eta = \eta(u_1,u_0), \quad (u_0,u_1)^{-1} = (u_0^{-1}, u_1^{-1}), \quad (u,v)(z,w) = (uz,vw).\end{align}
The notation $u = (u,b)$, for $u,b \in \Aut(T)$ means that $u$ is defined recursively, namely, $u|V_1$ is trivial, $u|T_0 = u$ and $u|T_1 = b$, that is, $u$ acts as  $b$ on the subtree $T_1$, and as $u$ on the subtree $T_0$. By the recursive formula we have $u|T_{00} = u$ and $u|T_{01} = b$, and, more generally, $u|T_{0^{n}} = u$ and $u|T_{0^{n-1}1} = b$ for $n \geq 1$, where $0^n$ denotes the concatenation of $n$ symbols $0$. Thus $b$ and the formula $u = (u,b)$ completely determine $u$. This way to write automorphism of binary trees is standard in geometric group theory, see for instance \cite{BruSidki97,Gri11,Nekr,Pink13}.

\subsubsection{Topology on the group of automorphisms.}\label{sec-topautom} The group $\Aut(T)$ coincides with the group of isometries of $\partial T$ (or equivalently $\Sigma$)  with respect to the metric on $\partial T$ induced from \eqref{eq-metrictree}.    

We can endow $\Aut(T)$ with the uniform  topology inherited from the group of homeomorphisms of $\partial T$. Since $\Aut(T)$ is a group of isometries, the uniform topology and  the topology of pointwise convergence coincide in $\Aut(T)$. We denote by $D$ the uniform distance in $\Aut(T)$, that is, 
\begin{align}\label{eq-AutTmetric}D(u, w)=\sup_{x\in \partial T}d(u x,w x), \textrm{ for } u, w\in \Aut(T). \end{align}

 We show that the uniform topology on $\Aut(T)$ coincides with the profinite group topology.
For every $n\geq 1$ and $u\in \Aut(T)$, we denote by $u_n = u| V_n$ the restriction of $u$ to $V_n$. The collection 
\begin{align}\label{eq-normalnbhd}N_n=\{u\in \Aut(T): u_n \textrm{ is trivial}\} \end{align}
is a finite index normal subgroup of $\Aut(T)$, which  is open since $\Aut(T)$ is compact \cite[Lemma 2.1.2]{RibZal10}. We note that $\{N_n \mid n \geq 1\}$ forms a system of open neighborhoods of the identity in $\Aut(T)$, making $\Aut(T)$ a profinite group \cite[Theorem 2.1.3]{RibZal10}. On the other hand, $N_n$ coincides with the set of all $u\in \Aut(T)$ which are at distance at most $\frac{1}{2^n}$ from the identity automorphism in $\Aut(T)$ with respect to the metric \eqref{eq-AutTmetric}. Thus the profinite topology and the uniform topology of $\Aut(T)$ coincide. 
    
    We note that the action of any subgroup $\Gamma$ of $\Aut(T)$  on $\partial T$ is automatically effective. That is, for every non-identity $u \in \Gamma$ there is $x \in \partial T$ such that $ u x \ne x$, or, in other words, the homomorphism $\Gamma \to \Homeo(\partial T)$ has trivial kernel. A consequence of this is that if $u, w \in \Aut(T)$ and for all $x \in \partial T$ we have $u x = w x$, then $u = w$.

 For a subset $\Gamma \subset \Aut(T)$, we denote by $\overline{\Gamma}$ the closure of $\Gamma$ with respect to the uniform topology or the topology of pointwise convergence.  Observe that  if $\Gamma$ is a subgroup of $\Aut(T)$ then $\overline{\Gamma}$ corresponds to the {\it enveloping semigroup}  of the dynamical system given by the action of $\Gamma$ on  $\partial T$ (see \cite{Aus88} for details). The following results can be found in \cite{Aus88}, we include the proof here for the  completeness.

 \begin{lemma}\label{basic-dynamics}
 Let $\Gamma$ be a subgroup of $\Aut(T)$ such that the action of $\Gamma$ on $\partial T$ is minimal.  
 \begin{enumerate}
 \item $\overline{\Gamma}x=\partial T$, for every $x\in \partial T$. 
 
 \item If in addition $\Gamma$ is abelian, then the action of $\overline{\Gamma}$ on $\partial T$ is free.
 \end{enumerate}
  \end{lemma}
\begin{proof}
For (1) note that for every $y\in \partial T$ there exists a subsequence $(\gamma_i x)_{i\geq 0}$ of the $\Gamma$-orbit of $x$ that converges to $y$. By compactness of $\overline{\Gamma}$ there exists a subsequence of $(\gamma_i)_{i\geq 0}$ that converges to some $\gamma$ in $\overline{\Gamma}$. We get $\gamma x=y$.

For (2) suppose there exist $x\in \partial T$, $u$ and $w$ in $\overline{\Gamma}$  such that $u x=w x$. If $\Gamma$ is abelian then $\overline{\Gamma}$ is also abelian. Thus for every $v\in \Gamma$ we have $u v x =v u x= v w x=w v x$. Since the action of $\Gamma$ on $\partial T$ is minimal we have that $u$ and $w$ coincide on a dense subset  $\Gamma x \subset \partial T$, which implies that $u=w$.
\end{proof}

\subsection{Permutations}\label{subsec-permutation} Let $n\geq 1$ and let $\sigma$ be a permutation on a set $S$ of cardinality $n$. We denote by $o(\sigma)$ the order of the permutation $\sigma$.   It is known that if $\sigma$ has $s$ cycles whose lengths are $\ell_1,\cdots, \ell_s$, then $\ell_1+\cdots+\ell_s=n$ and  $o(\sigma)$ is equal to the least common multiple of $\ell_1,\cdots, \ell_s$.  We say that $\sigma$ acts {\it transitively} on $S$ if $\sigma$  has only one cycle of length $n$. In this case we have  $o(\sigma)=n$, however the converse is not necessarily true as it is shown in this example: suppose that $n=30$, $s=9$ and the lengths of the cycles of $\sigma$ are given by $2,2,2, 5,5,5, 3,3,3$. The permutation $\sigma$ does not act transitively on $S$,  but since the least common multiple of $2,5,3$ is $30$, its order is $o(\sigma)=30$.
 
\begin{lemma}\label{permutation-primes}
Let $n\geq 1$, let $p>1$ be a prime number and let $S$ be a set of cardinality $p^n$.   The permutation $\sigma$ acts transitively on $S$ if and only if $o(\sigma)=p^n$.
\end{lemma}
\begin{proof}
Let $\sigma$ be a permutation on $S$ verifying $o(\sigma)=p^n$. Let $\ell_1,\cdots, \ell_s$ be the lengths of the cycles of $\sigma$. Since $p^n$ is the least common multiple of $\ell_1,\cdots, \ell_s$, every $\ell_i$ is equal to a power of $p$.   It follows that $n$ is equal to the maximum of these powers of $p$,  and so there is a cycle of order $p^n$. But this implies that $s=1$ and then $\sigma$ acts transitively on $S$.  
\end{proof}

\section{Procyclic subgroups of $\Aut(T)$ and maximal tori} \label{procyclic_groups}
In this section,  $T$ is a spherically homogeneous tree with arbitrary spherical index $m_T = (m_1,m_2,\ldots)$, see Section \ref{rooted} for notation.

\subsection{Odometers.}\label{odometer-definition}

Let $\overline{p}=(p_n)_{n\geq 1}$ be an increasing sequence of positive integers such that $p_n$ divides $p_{n+1}$, for every $n\geq 1$. We denote 
  $$\ZZ_{\overline{p}} = \lim_{\longleftarrow}\{\ZZ/p_n\ZZ, \phi_n, n\geq 1\},$$ 
 where  $\phi_n:\ZZ/p_{n+1}\ZZ \to \ZZ/p_n\ZZ$ is the natural projection, for every $n\geq 1$. If there exists a natural number $q$ such that  $p_n=q^n$ for each $n$, then we write $\ZZ_q$ instead of $\ZZ_{\overline{p}}$.   
 
 We consider $\ZZ_{\overline{p}}$ as a topological group, with  the topology induced by the  following metric  \cite{Down05}: let $x=(x_k)_{k\geq 1}$ and $y=(y_k)_{k\geq 1}$ be two elements in $\ZZ_{\overline{p}}$, then
  \begin{align}\label{eq-groupZp}
  d(x,y)=\frac{1}{2^{n-1}} \mbox{ where } n=\min\{k\geq 1: x_k\neq y_k\}. 
  \end{align}

It is not difficult to see that the following collection of clopen sets forms a basis for the topology:
$$
[g]_n=\{(x_k)_{k\geq 1}\in \ZZ_{\overline{p}}: x_n=g\}, \mbox{ where  } g\in \ZZ/p_n\ZZ \mbox{ and }  n\geq 1. 
$$

Observe that this topology coincides with the profinite topology on $\ZZ_{\overline{p}}$, defined similarly to the profinite topology on $\Aut(T)$ in Section \ref{sec-topautom}.

We denote by ${\bf 1}$ the element in $\ZZ_{\overline{p}}$ corresponding to   $(1+p_n\ZZ)_{n\geq 1}$, then  there is a homeomorphism   
$$
c_{\overline{p}}(x)=x+{\bf 1}, \mbox{ for every } x\in \ZZ_{\overline{p}},
$$
which is a minimal isometry. The dynamical system $(\ZZ_{\overline{p}},c_{\overline{p}})$ is known as the {\it odometer} or the {\it adding machine}. 

For a sequence $(m_1,m_2,\ldots)$ of positive integers, it is not difficult to show that when $p_n=m_1\cdots m_n$, for every $n\geq 1$, then  the dynamical system $(\ZZ_{\overline{p}},c_{\overline{p}})$ is conjugate by means of an isometry to $(\Sigma, \alpha_{\Sigma})$, where $\Sigma = \prod_{n \geq 1} \Sigma_n$ with $\Sigma_n = \{0,\ldots,m_n-1\}$, and $\alpha_{\Sigma}$ is  defined as follows: let $x=(x_n)_{n\geq 1} \in \Sigma$. If $x_n=m_n-1$ for every $n\geq 1$ then $\alpha_{\Sigma}(x)=0^{\infty}$,  the infinite sequence of $0$'s. Otherwise, let $k=\min\{n\geq 1: x_n<m_{n}-1\}$ and $y_1\cdots y_k \in \prod_{n=1}^k\Sigma_n$ be the successor of $x_1\cdots x_k$  with respect to the lexicographic order. We define $\alpha_{\Sigma}(x)=y$, where $y=y_1\cdots y_k x_{k+1}x_{k+2}\cdots$.  

By Section \ref{rooted} $\Sigma$ can be identified with the boundary of a spherically homogeneous tree $T$ with spherical index $m_T = (m_1,m_2,\ldots)$, and then the transformation $\alpha_{\Sigma}$ defines an element of $\Aut(T)$. We call $\alpha_{\Sigma}$ the {\it odometer} of $T$. See \cite{Down05} for more details about this class of dynamical systems.

\subsection{Procyclic subgroups of $\Aut(T)$} \label{subsec-procyclicAut}
For every $a\in \Aut(T)$ we denote by $\langle a \rangle$ the cyclic group generated by $a$. We denote by $\overline{\langle a \rangle}$ the closure of the group $\langle a \rangle$ in $\Aut(T)$. 
It is known that $\overline{\langle a \rangle}$  is a  {\it procyclic} group, i.e., it is isomorphic to the inverse limit of finite cyclic groups (see for example \cite{RibZal10}). 

\subsubsection{Maximal tori}
By Section \ref{subsec-dynsys} $a \in \Aut(T)$ is {\it minimal}  if for every $x\in \partial T$ its $a$-orbit   $\langle a \rangle x$ is dense in $\partial T$. Further,  since $a$ is an isometry,  $a$ is minimal if and only if there exists $x\in \partial T$ such that its $a$-orbit is dense in $\partial T$ (see for example \cite{Aus88}). In the setting of actions on trees, $a$ is minimal if and only if for every $n\geq 1$ and $w\in V_n$ the $a$-orbit of any $x$ intersect the cylinder set $[w]=\{(x_i)_{i\geq 1}: x_1\cdots x_n=w\}$.  Recall the notation $a_n = a|V_n$ for the restriction of $a$ to $V_n$, $n \geq 1$. The argument above gives the following: 
 
 \begin{lemma}\label{minimal-transitivity}
 Let $a\in \Aut(T)$. Then $a$ is minimal if and only if $a_n$ acts transitively on $V_n$, for every $n\geq 1$.
 \end{lemma}
 
Lemma \ref{minimal-transitivity} shows that the definitions of minimal elements in Section \ref{subsec-dynsys} and in the Introduction just after Definition \ref{defn-maxtorus} are equivalent. Recall from Definition \ref{defn-maxtorus} that a maximal torus in $\Aut(T)$ is a maximal abelian subgroup of $\Aut(T)$ which acts freely and transitively on $\partial T$. We now can identify some maximal tori in $\Aut(T)$.
 
\begin{lemma}\label{minimal-elements}
If $a$ is minimal then $\overline {\langle a \rangle}$ is a maximal torus in $\Aut(T)$.   
\end{lemma} 

\proof Since $a$ is minimal and $\langle a \rangle$ is abelian, by Lemma \ref{basic-dynamics} $\overline {\langle a \rangle}$ acts transitively and freely on $\partial T$, and we only have to show that $\overline {\langle a \rangle}$ is maximal. Let $B \supset \overline {\langle a \rangle}$ be a closed subgroup of $\Aut(T)$ which acts freely on $\partial T$. Let $u \in  B$, and let $x \in \partial T$. Since $\overline {\langle a \rangle}$ acts transitively on $\partial T$ then there exists $w \in \overline {\langle a \rangle}$ such that $w(x) = u(x)$, which implies that $w^{-1} \circ u(x) = x$, and so $w = u$ since the action of $B$ is free. But this implies that $B = \overline {\langle a \rangle}$, and $\overline {\langle a \rangle}$ is maximal.
\endproof

Two elements $u$ and  $v$ in $\Aut(T)$ are {\it conjugate}   if there exists $w\in \Aut(T)$ such that $u=w v w^{-1}$. Conjugation induces an equivalence relation on $\Aut(T)$, whose equivalence classes are called {\it conjugacy classes}.  It is well known that if $\sigma$ and $\tau$ are conjugate permutations of a finite set $S$ then their cycle structures are the same.  
This implies, together with Lemma \ref{minimal-transitivity}, the next result.

 \begin{cor}\label{Pink-result}
 The family of minimal elements in $\Aut(T)$ coincides with a conjugacy class of $\Aut(T)$. Thus every minimal element of $\Aut(T)$ is conjugate to $\alpha_{\Sigma}$. \end{cor}
 
 Corollary \ref{Pink-result} states that maximal tori in $\Aut(T)$ that are topologically generated by a single element are conjugate in $\Aut(T)$. This does not imply that that \emph{all} maximal tori are conjugate in $\Aut(T)$, since conceivably there may be maximal tori in $\Aut(T)$ with more than one topological generator.

\subsubsection{General procyclic subgroups}
The next proposition gives us a more precise description of procyclic groups in $\Aut(T)$, without an assumption that $a$ is minimal. 
Recall the notation $a_n = a|V_n$ for a restriction to the vertex level set, and $|V_n|$ for the cardinality of a finite set $V_n$, $n \geq 1$.

\begin{prop}\label{conjugacy-elements}
Let $a\in \Aut(T)$ and $\overline{p}=(p_n)_{n\geq 1}$. The following statements are equivalent:
\begin{enumerate}
\item $p_n=o(a_n)$, for every $n\geq 1$. 
\item There exists an isomorphism  $\psi:  \overline{\langle a \rangle }\to \ZZ_{\overline{p}}$, which is an isometry and such that the image of $a$   is equal to ${\bf 1}\in \ZZ_{\overline{p}}$. 
\item  $ \overline{\langle a \rangle }$ is isometrically isomorphic to $\ZZ_{\overline{p}}$.
\end{enumerate}

In addition, if $a$ is minimal, then $o(a_n)=|V_n|$,  for every $n\geq 1$.
\end{prop}
\begin{proof}
For every $n\geq 1$, let $p_n=o(a_n)$. 
For  $k\in\ZZ$ we define $\psi(a^k)=(k+p_n\ZZ)_{n\geq 1}\in \ZZ_{\overline{p}}$.   
Observe that $d(a^k, a^j)=\frac{1}{2^n}$ if and only if $a_n^{k-j}=id$ and $a_{n+1}^{k-j}\neq id$, which is equivalent to $k-j\in p_n\ZZ$   and $k-j\notin p_{n+1}\ZZ$. This implies that $d(\psi(a^k), \psi(a^\ell))=d(a^k, a^\ell)$ for $k,\ell \in \ZZ$.   Then $\psi$ extends to a unique isometry $\psi: \overline{\langle a \rangle }\to \ZZ_{\overline{p}}$, such that $\psi(a)={\bf 1}$.  By construction, $\psi$ is a homomorphism.  It is surjective since $\psi(\langle a\rangle)$ is dense in $\ZZ_{\overline{p}}$. We have shown that (1) implies (2). Statement (3) is a direct consequence of (2).

 Now suppose that $ \overline{\langle a \rangle }$ is isometrically isomorphic to $\ZZ_{\overline{p}}$. Let $\overline{r}=(r_n)_{n\geq 1}$ be such that $o(a_n)=r_n$, for every $n\geq 1$. According to what has been shown above, we have   $ \overline{\langle a \rangle }$ is isometrically isomorphic to $\ZZ_{\overline{r}}$. Thus $\ZZ_{\overline{r}}$ and $\ZZ_{\overline{p}}$ are isometrically isomorphic.  Let $\phi:\ZZ_{\overline{p}}\to  \ZZ_{\overline{r}}$ be an isometric isomorphism. For every $n\geq 1$, the map $\phi$ induces a bijection between the collection of cylinder sets $\{[\ell]_n: \ell\in\ZZ/p_n\ZZ\}$ and $\{[j]_n: j\in\ZZ/r_n\ZZ\}$. This implies that the cardinalities of $\ZZ/p_n\ZZ$ and $\ZZ/r_n\ZZ$ coincide, which is only possible if $p_n=r_n$. This shows that $o(a_n)=p_n$ for every $n\geq 1$. If $a$ is minimal then $o(a_n) = p_n = |V_n|$.
\end{proof}

In the next proposition, we study the case when $a$ is minimal. Recall that $(\ZZ_{\overline{p}},c_{\overline{p}})$ denotes the odometer defined in Section \ref{odometer-definition}, that is, $\ZZ_{\overline{p}}$ is a profinite abelian group, and $c_{\overline{p}}$ is the addition of one.

 \begin{prop}\label{minimal-means-conjugacy}
Let $\overline{p}=(p_n)_{n\geq 1}$ be such that $p_n = |V_n|$, for every $n\geq 1$. Then $a\in \Aut(T)$ is minimal if and only if the dynamical systems $(\partial T, a)$ and $(\ZZ_{\overline{p}},c_{\overline{p}})$ are conjugate by means of an isometry.
 \end{prop}
 \begin{proof}
 If  $(\partial T, a)$ and $(\ZZ_{\overline{p}},c_{\overline{p}})$ are conjugate, then the minimality of $c_{\overline{p}}$ implies that $a$ is minimal.
 
Suppose now that $a$ is minimal. Corollary \ref{minimal-elements} and Proposition \ref{conjugacy-elements} imply that  $ \overline{\langle a \rangle }$ and  $\ZZ_{\overline{p}}$ are isometrically isomorphic.  Let $\psi$ be the isometry from  $ \overline{\langle a \rangle }$ to $\ZZ_{\overline{p}}$ given in  
Proposition \ref{conjugacy-elements}.

 Let $x\in \partial T$. Since the $a$-orbit  of $x$ is dense in $\partial T$, the function $w: \langle a\rangle x\to \ZZ_{\overline{p}}$ given by $w(a^k x)=\psi(a^k)$, for every $k\in\ZZ$,  extends to a bijective isometry $w$ from $\partial T$ to $\ZZ_{\overline{p}}$. Moreover, this isometry verifies $w\circ a=c_{\overline{p}}\circ w$. The map $w$ is in $\Aut(T)$ and it is a conjugating map between $(\partial T, a)$ and $(\ZZ_{\overline{p}},c_{\overline{p}})$.   
 \end{proof}
 
 We highlight the difference in the statements of Proposition \ref{conjugacy-elements} and \ref{minimal-means-conjugacy} in the example below.

\begin{ex}
{\rm
 Let $T$ be a binary tree, so $|V_n| = 2^n$ for $n \geq 1$. Let $a$ be minimal, and let $b = a^2 = (a,a)$. Then each restriction $b|V_n$ consists of two cycles of length $2^{n-1}$, and $o(b_n) = 2^{n-1}$ by the definition in Section \ref{subsec-permutation}. By Proposition \ref{conjugacy-elements} we have that $ \overline{\langle b\rangle }$ is isometrically isomorphic to the dyadic integers $\ZZ_2$. The dynamical system $(\partial T,b)$ is not conjugate to $(\ZZ_2,c_2)$, while each restriction $(\partial T_0, b_0)$ or $(\partial T_1, b_1)$ is conjugate to $(\ZZ_2,c_2)$.
 }
 \end{ex}

From these results and Lemma \ref{permutation-primes} we obtain a characterization of the minimal elements of the group of automorphisms of $p$-ary rooted trees, when $p$ is a prime number. The key step in this case is the characterization of the permutations acting transitively on $V_n$ by means of their orders. 

\begin{prop}\label{prop-equivodom}
Let $p>1$ be a prime number, and suppose that $T$ is the $p$-ary rooted tree, that is, for the spherical index $m_T=(m_1,m_2,\ldots)$ we have $m_n = p$, $n \geq 1$. Let $a\in \Aut(T)$. Then the following statements are equivalent:
\begin{enumerate}
\item $a$ is minimal. 
\item $a$ is conjugate to $\alpha_{\Sigma}$, the odometer of $T$.  
\item $a_n$ acts transitively on $V_n$, for every $n\geq 1$.
\item $o(a_n)=p^n$, for every $n\geq 1$.
\item  $\overline{\langle a \rangle}$ is isometrically isomorphic to $\ZZ_{p}$.
\item The dynamical systems $(\partial T, a)$ and $(\ZZ_{p}, c_{p})$ are conjugate by means of an isometry.  
\end{enumerate}
\end{prop}

\begin{remark}
{\rm
In the case $p=2$ statement $(4)$ of Proposition \ref{prop-equivodom} can be replaced by a weaker statement that ${\rm sgn}(a_n) = -1$ for $n \geq 1$, where ${\rm sgn}(a_n)$ denotes the parity of the permutation $a_n$, see \cite[Proposition 1.6.2]{Pink13}.
}
\end{remark}

\section{The normalizer and the absolute Weyl group of a maximal torus} \label{sec_normalizer}

Let $T$ be a rooted tree with spherical index $m_T = (m_1,m_2,\ldots)$, see Section \ref{rooted}. For every $n\geq 1$, we set $p_n=m_1\cdots m_n$.
Let $a\in \Aut(T)$ be a minimal automorphism. Recall that we denote by $\langle a \rangle$ the cyclic group generated by $a$, and by $\overline{\langle a \rangle}$ the closure of $\langle a \rangle$ in $\Aut(T)$.  By Lemma \ref{minimal-elements} $\overline {\langle a \rangle}$ is a maximal torus in $\Aut(T)$, and its normalizer is defined by
  $$N(\overline{\langle a \rangle}) = \{w\in \Aut(T): w \overline{\langle a \rangle}  w^{-1}= \overline {\langle a \rangle} \}.$$  
In this section, we study the structure of the normalizer $N(\overline{\langle a \rangle})$, and of the absolute Weyl group $W(\overline {\langle a \rangle}) = N(\overline {\langle a \rangle})/ \overline {\langle a \rangle}$, especially in the case when the spherical index $m_T$ is bounded.
A spherical index $m_T$ is bounded if there exists $d \geq 2$ such that $m_n \leq d$ for all $n \geq 1$.

\subsection{Elements of the Weyl group} For $a \in \Aut (T)$ minimal with the corresponding maximal torus $\overline {\langle a \rangle}$, the absolute Weyl group is given by $W(\overline{\langle a \rangle}) = N(\overline {\langle a \rangle})/\overline {\langle a \rangle}$. Define 
\begin{align}\label{eq-layers}
G_{\zeta}=\{w\in \Aut(T): w a w^{-1}=\zeta\}, \textrm{ for a minimal }\zeta \in \overline {\langle a \rangle}.
\end{align}
Since $\overline {\langle a \rangle}$ is maximal and abelian, such a set $G_\zeta$ is a coset of $\overline {\langle a \rangle}$ in $\Aut(T)$. We prove that the sets $G_\zeta$ are precisely the cosets of $\overline {\langle a \rangle}$ in $N(\overline {\langle a \rangle})$.

 \begin{lemma}\label{lemma-norm2}
Let $w \in Aut(T)$ and let $a \in \Aut (T)$ be a minimal element. Then $w\in N(\overline{\langle a\rangle})$ if and only if  $waw^{-1}\in \overline{\langle a\rangle} $. Therefore: 
\begin{enumerate}
\item  $
N(\overline{\langle a\rangle}) = \bigcupdot \{G_\zeta : \zeta \in \overline{\langle a\rangle} \textrm{ and } \zeta \textrm{ is minimal }\},
 $
 \item The absolute Weyl group $W(\overline{\langle a\rangle})$ of the maximal torus $\overline{\langle a\rangle}$ is the set of cosets \\ $\{G_\zeta : \zeta\in \overline{\langle a\rangle} \textrm{ and } \zeta \textrm{ is minimal }\}$.
\end{enumerate}
\end{lemma}

\begin{proof} If $w \in N(\overline{\langle a\rangle})$ then $w a w^{-1} \in \overline{\langle a\rangle}$ by definition and by Corollary \ref{Pink-result} $waw^{-1}$ is minimal. 

Let $\zeta = waw^{-1} \in \overline{\langle a\rangle}$. We show that $w \overline{\langle a\rangle} w^{-1} = \overline{\langle a\rangle}$.

Note that $wa^kw^{-1} = (waw^{-1})^k = \zeta^k$,
and so we have $w \langle a\rangle w^{-1} = \langle \zeta \rangle \subset\overline{\langle a\rangle}$, which implies $w \overline{\langle a\rangle} w^{-1} \subseteq \overline{\langle a\rangle}$.
On the other hand, since $\zeta$ is minimal, then by Lemma \ref{minimal-elements} its closure $\overline{ \langle \zeta \rangle}$ is a maximal torus, which implies that $\overline{ \langle \zeta \rangle} = \overline{\langle a\rangle}$. The rest of the lemma follows.
\end{proof}

Next we observe that if $\zeta=a$ then $G_a$ contains all elements in $\Aut(T)$ that commute with $a$. Denote by $Z(\overline{\langle a\rangle})$ the centralizer of $\overline{\langle a\rangle}$ in $\Aut (T)$, that is, the set of all elements in $\Aut(T)$ which commute with every element in $\overline{\langle a\rangle}$. 

 \begin{lemma}\label{abelian}
For $a \in \Aut(T)$ minimal we have $Z( \overline{\langle a \rangle }) = G_a=  \overline{\langle a \rangle }$.
\end{lemma} 
\begin{proof}
Since $a$ is a topological generator of $ \overline{\langle a \rangle }$, we have that $g \in \Aut(T)$ commutes with $a$ if and only if $g$ commutes with every element in $ \overline{\langle a \rangle }$, so $Z( \overline{\langle a \rangle }) = G_a$.

Next, since $ \overline{\langle a \rangle }$ is abelian, we have $ \overline{\langle a \rangle } \subseteq G_a$. 
Let $x\in \partial T$ and $\sigma \in G_a$.  Since $ \overline{\langle a \rangle }$ is the enveloping semigroup of $\langle a \rangle$, and the action of $\langle a \rangle$ on $\partial T$ is minimal, we have $ \overline{\langle a \rangle } x=\partial T$ for the orbit of $x$ in $\partial T$. Thus there exists $\gamma\in  \overline{\langle a \rangle }$ such that $\sigma x=\gamma x$.  This implies that for every $n\geq 0$ we have $\sigma a^n x=a^n\sigma x =a^n\gamma x=\gamma a^n x$, since $\sigma$ commutes with $a$. This shows that $\sigma$ and $\gamma$ coincide on the $a$-orbit of $x$, and since this orbit is dense, we conclude $\gamma=\sigma$, and so $\sigma \in  \overline{\langle a \rangle }$.
\end{proof}

\begin{remark}\label{remark-weyl-centralizer}
{\rm
Lemma \ref{abelian} implies that if $a \in \Aut(T)$ is minimal, then we can define the Weyl group of the maximal torus $ \overline{\langle a \rangle }$ by $W( \overline{\langle a \rangle }) = N( \overline{\langle a \rangle })/Z( \overline{\langle a \rangle })$. We note that a similar alternative definition can be made for the Weyl groups of maximal tori in compact Lie groups. 
}
\end{remark}

\subsection{Rational elements of the Weyl group and the normalizer of a maximal torus} For a minimal $a \in \Aut(T)$, in Definition \ref{defn-ratspan} below, we introduce the \emph{rational spanning set} in $N( \overline{\langle a \rangle })$. The set is called rational since its elements can be indexed by two integers. We prove that if the spherical index $m_T$ of $T$ is bounded, then the rational spanning set is dense in $N( \overline{\langle a \rangle })$. We give an example of a tree with unbounded spherical index, where the rational spanning set is degenerate and thus it is not dense in $N( \overline{\langle a \rangle })$.

 \subsubsection{Cosets $G_k$ in the Weyl group}
 We consider the set of positive integers which are coprime to every $p_n$, $n \geq 1$, that is,
 \begin{align}\label{eq-coprime}
 I(T)=\{k\geq 1: \gcd\{k, p_n\}=1 \mbox{ for every } n\geq 1\}.
 \end{align}
For $a^k$ with $k\in I(T)$,  we write $G_k$ instead of $G_{a^k}$, where $G_{a^k}$ is defined by \eqref{eq-layers}. We have the following.
 
 \begin{lemma}\label{integer-minimality}
   Let $k\geq 1$.  Then $a^k$ is minimal if and only if $k\in I(T)$.  
 \end{lemma}
  \begin{proof}
  Let $n\geq 1$. 
  Suppose that $ \gcd\{k, p_n\}=1$. By B\'ezout's identity there exist $m, \ell \in \ZZ$ such that $mk=1+\ell p_n$.  Thus for every $1\leq s\leq p_n$ and every $v\in V_n$ we have
  $$
  (a^k)^{ms}(v)=a^sa^{s \ell p_n}(v)=a^s(v) \in V_n,
  $$
 that is, there exists a power of $a^k$ which maps $v$ to $a^s(v)$ for any $1 \leq s \leq p_n$.  
 Since $a$ is transitive on $V_n$, then $a^k$ is also transitive on $V_n$. Since this is true for every $n\geq 1$, we conclude that $a^k$ is minimal.
  
  Suppose that $ \gcd\{k, p_n\}=\ell>1$. Then there exist $1\leq s<k$ and $1\leq r <p_n$ such that $k=\ell s$ and $p_n=\ell r$. Thus $kr=p_ns$ and for every $v\in V_n$ we have
  $(a^k)^r(v)=(a^{p_n})^s(v)=v$, that is, $(a^k)^r$ fixes every vertex in $V_n$. Since $1\leq r<p_n$ we deduce that $a^k$ is not transitive on $V_n$. This implies that $a^k$ is not minimal. 
  \end{proof}
    
\begin{remark}{\rm  From Lemma \ref{integer-minimality}, we deduce that if $T$ is the $d$-ary tree for some $d\geq 2$,  then $a^k$ is minimal if and only if $ \gcd\{k,d\}=1$. In the case of the binary tree this is equivalent to saying that $a^k$ is minimal if and only if $k$ is odd.}
  \end{remark}

\subsubsection{Rational subsets in cosets $G_k$} Recall from Section \ref{rooted} that the boundary $\partial T$ of a tree $T$ with spherical index $m_T = (m_1,m_2,\ldots)$ is identified with the set $\Sigma = \prod_{n \geq 1} \Sigma_n$, where $|\Sigma_n| = m_n$ for $n \geq 1$. Thus each vertex in $T$ corresponds to a finite word, and each infinite path in $\partial T$ corresponds to an infinite sequence of digits in $\Sigma_n$, $n \geq 1$.
Let $\xi=(\xi_n)_{n\geq 1}\in \partial T$ be the sequence such that $\xi_n=0$ for every $n\geq 1$.  We show that within the cosets $G_k$, for $k \in I(T)$, the elements are determined by their action on $\xi$. As a consequence, they are indexed by elements in $\partial T$.

As earlier in this section, $a$ is a minimal element in $\Aut(T)$, and $G_k$ for $k \in I(T)$ are defined by \eqref{eq-layers}, for $\zeta = a^k$.

\begin{lemma}\label{lemma-initialvalue}
An element $\sigma \in G_k$ is completely  determined by its value at $\xi$.  
\end{lemma}

\proof
Let $\sigma\in G_k$. The relation $\sigma a=a^k\sigma$ implies $\sigma a^m=a^{km}\sigma$, for every $m\geq 0$. Thus
$$
\sigma a^m \xi=a^{mk}\sigma \xi,
$$ 
which implies that $\sigma \xi $ determines completely the values of $\sigma$ on the set $S=\{a^m \xi: m\geq 0\}$. Since $a$ is minimal, then $S$ is dense in $\partial T$, which implies that $\sigma$ is completely determined on $\partial T$.  
\endproof

 \begin{lemma}\label{lemma-findelement}
 For any $t \in \partial T$ there exists $\sigma\in G_k$ such that $\sigma \xi=t$. 
 \end{lemma}
 
 \proof  
 Let $\tau\in G_k$. Then by Lemma \ref{abelian} we have $G_k=\tau \overline{\langle a \rangle }$. This implies that
 $$
 G_k \xi=\tau \overline{\langle a \rangle } \xi=\tau \partial T=\partial T.  
 $$ 
 Thus, there exists $\sigma\in G_k$ such that $\sigma \xi =t$.
 \endproof

 The two lemmas above motivate the following definition.
 
  \begin{definition}\label{defn-onorbits}
 Let $k\geq 1$ be an integer in $I(T)$. For every integer $m\geq 1$  we define $\sigma_{m,k}$ as the element of $G_k$ given by
 $$
  \sigma_{m,k} \xi=a^m \xi.
 $$
 \end{definition}
 
 Indeed, an element satisfying Definition \ref{defn-onorbits} exists by Lemma \ref{lemma-findelement}, and it is unique by Lemma \ref{lemma-initialvalue}. It is in $N(\overline{\langle a \rangle })$ by Lemma \ref{lemma-norm2}(1).
 
 \begin{defn}\label{defn-ratspan}
 A \emph{rational spanning set} in $N( \overline{\langle a \rangle })$ is the set $\{\sigma_{m,k} \mid k \in I(T), m \geq 1\}$, where $\sigma_{m,k}$ are given by Definition \ref{defn-onorbits}. 
 \end{defn}
 
 \begin{remark}\label{remark-k1}
{\rm  Observe that for $k=1$ we have $G_1= \overline{\langle a \rangle }$, and $\sigma_{m,1}=a^m$ for every $m\geq 1$, so $\{\sigma_{m,1} \mid m \geq 1\}$ coincides with the semi-group of positive powers of $a$.}
  \end{remark}
 
 \subsubsection{Dynamics of elements in the cosets $G_k$.}
 In Lemma \ref{convergence} below, we show that a sequence of elements in a coset $G_k$, $k \in I(T)$, of the Weyl group of $ \overline{\langle a \rangle }$, which converges at a single point $\xi \in \partial T$, converges at every point in $\partial T$, thus giving an element in $\Aut(T)$. Thus for elements in $G_k$ convergence at a point implies convergence in $\Aut(T)$.

 \begin{lemma}\label{convergence}
 Let $G_k \in W( \overline{\langle a \rangle })$ be a coset in the Weyl group of the maximal torus $ \overline{\langle a \rangle }$ topologically generated by a minimal element $a \in \Aut(T)$, for $k \in I(T)$. Let $(\sigma_i)_{i\geq 0}$ be a sequence in $G_{k}$ such that $\lim_{i\to \infty}\sigma_i \xi=\sigma \xi$, where $\sigma\in G_k$. Then $(\sigma_i)_{i\geq 0}$ converges to $\sigma$ in $\Aut(T)$.    
 \end{lemma}
 \begin{proof}
 The relation $\sigma_{i}a=a^k\sigma_{i}$ implies $\sigma_{i}a^m=a^{mk}\sigma_{i}$ for every $m\geq 0$. Thus
 $$
 \sigma_{i}a^m \xi= a^{mk}\sigma_{i} \xi,
 $$
 which implies
 $$
 \lim_{i\to \infty} \sigma_{i}a^m \xi=a^{mk}\sigma \xi=\sigma a^m \xi.
 $$
Thus $(\sigma_{i})_{i\geq 0}$ converges pointwise to $\sigma$ on the set $\{a^m \xi: m\geq 0\}$.   The pointwise convergence on $\partial T$ follows from the density of $\{a^m \xi: m\geq 0\}$ in $\partial T$. Indeed, first observe that for every $m\geq 0$ we have
 \begin{equation}\label{convergence-1} 
 d(\sigma_ia^m \xi, \sigma a^m\xi)= d(a^{km}\sigma_i \xi,  a^{km}\sigma \xi)=d(\sigma_i \xi,  \sigma \xi).
 \end{equation}
 Let $t\in \partial T$ and let $(a^{m_i} \xi)_{i\geq 0}$ be a sequence converging to $t$. We write $a^{m_i} \xi =t_i$, so by assumption $t_i \to_i t$. Then we have 
 \begin{eqnarray*}
 d(\sigma_i t, \sigma t) & \leq & d(\sigma_i t, \sigma_i t_i)+d(\sigma_i t_i, \sigma t_i)+d(  \sigma t_i, \sigma t)\\  
 &=& d(t,t_i)+d(\sigma_i \xi, \sigma \xi)+d(  \sigma t_i, \sigma t),\end{eqnarray*}
 where the second line is a consequence of (\ref{convergence-1}) and the fact that $\sigma_i$ is an isometry.
 This implies that $\lim_{i\to \infty}\sigma_i t=\sigma t$. Since the topology of pointwise convergence in $\Aut(T)$ coincides with the uniform topology and the profinite topology, we conclude that $(\sigma_i)_{i\geq 0}$ converges to $\sigma$ in $\Aut(T)$.   
 \end{proof}
 
 In the following lemma we show that the set of rational elements in $G_k$ is dense in $G_k$.
 
 \begin{lemma}\label{density-0}
  Let $k\geq 1$ be an integer in $I(T)$. The set $\{\sigma_{m,k}: m\geq 1\}$ is dense in $G_k$. 
 \end{lemma}
 \begin{proof}
 For $k=1$ it is clear from Remark \ref{remark-k1}.
 
Let $k>1$ and  let $\sigma\in G_k$. By minimality of $a$, there exists a sequence $\{a^{n_i} \xi : n_{i+1}>n_i>0\}$ that converges to $\sigma \xi$.  Since  $\sigma_{n_i,k}\xi=a^{n_i}\xi$ for every $i\geq 1$ (see Definition \ref{defn-onorbits}),  we have
 $$
 \lim_{i\to \infty}\sigma_{n_i,k} \xi=\sigma \xi.   
 $$
 Then the conclusion follows from Lemma \ref{convergence}.
  \end{proof}

 \subsubsection{Rational spanning set in $N(\overline{\langle a \rangle })$}
 In this section we assume that the spherical index $m_T$ of the tree $T$ is bounded. That is, there exists $d>0$ such that $|m_n|\leq d$ for every $n\geq 1$. Note that this is equivalent to saying that there exist integers  $d_1, \cdots, d_M \geq 2$ such that 
  $k=m_n$ for some $n\geq 1$ if and only if $k\in \{d_1,\cdots, d_M\}$.  Observe that in this case
  $$
  I(T)=\{p\geq 1:  \gcd\{p,d_1,\cdots, d_M\}=1\}.
  $$
 The case when $M=1$ corresponds to the case when $T$ is the $d$-ary tree, with $d=d_1$ and $m_n=d$ for every $n\geq 1$.

 \begin{thm}\label{thm-skeleton}
Let $T$ be a tree with bounded spherical index $m_T$ and let $a \in \Aut(T)$ be minimal. Then the rational spanning set 
 $$
 \bigcup_{\substack{k\geq 1\\  \gcd\{k,d_1,\cdots, d_M\}=1}}\{\sigma_{m,k}: m\geq 1\}
 $$
given by Definition \ref{defn-ratspan} is dense in $N( \overline{\langle a \rangle })$.
\end{thm}
 
\begin{proof} We first show that the union of rational cosets in the Weyl group $$\bigcupdot_{\substack{k\geq 1\\  \gcd\{k,d_1,\cdots, d_M\}=1}}G_{k}$$ is dense in $N( \overline{\langle a \rangle })$.

  Let $\zeta$ be a minimal automorphism in $ \overline{\langle a \rangle }$, and let $\sigma\in G_{\zeta}$. Since $\zeta$ is in $ \overline{\langle a \rangle }$ there exists a sequence $(a^{k_i})_{i\geq 0}$ that converges to $\zeta$, with $k_i \in \ZZ$ for $i \geq 0$. In particular this means that, given $i$, for every sufficiently large $j$ we have $\zeta|_{V_i} = a^{k_j}|_{V_i}$.  Assume that $i$ is large enough so that $|V_i|$ is divisible by all the integers in $\{d_1,\cdots, d_M\}$.  Since $\zeta$ is transitive at each level, we have that every $a^{k_j}$ is transitive on $V_i$, which implies that $ \gcd\{k_j, |V_i|\}=1$  and then $ \gcd\{k_j, d_1,\cdots, d_M\}=1$. Thus we can assume $ \gcd\{k_i, d_1,\cdots, d_M\}=1$ for every $i\geq 0$, which by Lemma \ref{integer-minimality} implies that $a^{k_i}$ is minimal.
  
  For  every $i\geq 0$, let $\sigma_i \in G_{k_i}$. Then we have
  $$
  \lim_{i\to \infty}\sigma_i a \sigma_i^{-1}=\lim_{i\to \infty}a^{k_i}=\zeta=\sigma a\sigma^{-1}.  
  $$   
 Consider the sequence $(\sigma^{-1}\sigma_i)_{i \geq 1}$. Since $\Aut(T)$ is compact, this sequence has a convergent subsequence $(\sigma^{-1}\sigma_{i_s})_{s \geq 1}$ to a limit point $\lambda$, that verifies
 $$
\lambda a \lambda^{-1} =\lim_{s\to \infty}\sigma^{-1}\sigma_{i_s} a \sigma_{i_s}^{-1}\sigma= \lim_{s\to \infty}\sigma^{-1}a^{k_{i_s}}\sigma= \sigma^{-1}(\sigma a \sigma^{-1})\sigma= a.
 $$
 Since $\lambda$ commutes with $a$, by Lemma \ref{abelian} we get $\lambda\in  \overline{\langle a \rangle }$. Then $\sigma_{i} \lambda^{-1} \in G_{k_{i}}$, for every $i\geq 0$. 
  
On the other hand, since $(\sigma_{i_j} \lambda^{-1})$ converges to $\sigma$, we get that  $\sigma$ is in the closure of  $\bigcupdot_{\substack{k\geq 1\\ k \in I(T)}}G_{k}$. Then the statement of the theorem follows by Lemma \ref{density-0}.
 \end{proof}

\begin{remark}\label{remark-degenerate}
{\rm  Theorem \ref{thm-skeleton} need not be true if $m_n$'s are not bounded. For example, suppose that $m_n=n+1$ for each $n\geq 1$. We get that $p_n=(n+1)!$ for every $n\geq 1$, and then $I(T)=\{1\}$.  This implies that
$$
\bigcupdot_{\substack{k\in I(T)}}G_{k}=G_1= \overline{\langle a \rangle }.
$$
On the other hand, since $a^{-1}$ is minimal ($a^{-1}$ is transitive on each $V_n$), Corollary \ref{Pink-result} implies there exists $w\in \Aut(T)$ such that $waw^{-1}=a^{-1}$.  From Lemma \ref{lemma-norm2} we get that $w\in N( \overline{\langle a \rangle })$ and, since $w$ does not commute with $a$, by Lemma \ref{abelian} we deduce that $w\notin  \overline{\langle a \rangle }$. This shows that $G_1= \overline{\langle a \rangle }$ is not dense in $N( \overline{\langle a \rangle })$ because $w\in N( \overline{\langle a \rangle })-  \overline{\langle a \rangle }$ can not be approximated by elements in $ \overline{\langle a \rangle }=G_1$. 
   }
\end{remark}

\section{Stable cycles and settled elements in $\Aut(T)$}\label{sec-stablesettled}

In this section we give a topological characterization of settled elements in $\Aut(T)$, for $T$ a $d$-ary rooted tree. However, most of the definitions and results in this section are valid for more general spherically homogeneous trees.

\subsection{Minimal components of $(\partial T, \sigma)$}\label{min-comp} Let $\sigma \in \Aut(T)$, and recall that $(\partial T, \sigma)$ denotes the dynamical system given by the action of the cyclic group $\langle \sigma \rangle$ on $\partial T$.  Recall from Section \ref{rooted} that we say that $u$ is a \emph{word of length} $k$ if $u$ is a concatenation of $k$ symbols in $\{0,\ldots, d-1\}$. We denote by $\{0,\ldots,d-1\}^k$ the set of all words of length $k$. If $v$ is a vertex in $V_n$ and $u$ is a word of length $k$, then $vu$ is a vertex in $V_{n+k}$. For every $n\geq 1$ and a vertex $v\in V_n$ we denote by
$$
[v]=\{x \in \partial T: x \mbox{ contains } v\}
$$
the cylinder set of paths passing through $v$. We denote by $\sigma^j[v]$ the image of the cylinder set $[v]$ under the power $\sigma^j$.

Recall from Definition \ref{defn-cycles} that we say that $v \in V_n$ is in a {\it cycle} of length $k\geq 1$ of $\sigma$ if $\sigma^k(v)=v$ and $\sigma^{j}(v)\neq v$ for every $1\leq j<k$. We say that $v$ is in a {\it stable cycle} of length $k$ if $v$ is in a cycle of length $k$ and if in addition for every $m>n$ and $w\in \{0,\cdots, d-1\}^{m-n}$, the vertex $vw\in V_m$ is in a cycle of length $d^{m-n}k$.  Observe this is equivalent to the following: for every $w,u\in \{0,\cdots, d-1\}^{m-n}$ and $1\leq i, j<k$ the vertices $\sigma^j(v)w$ and $\sigma^{i}(v)u$ are in the same cycle of length $d^{m-n}k$. 

A minimal component of a dynamical system was defined in Section \ref{subsec-dynsys}.

\begin{lemma}\label{simple}
Let $\sigma\in \Aut(T)$ and $n\geq 1$. The vertex $v\in V_n$ is in a stable cycle of length $k\geq 1$ if and only if the union of the cylinder sets  $X=\bigcup_{j=1}^{k}\sigma^j[v]$  is a minimal component of $(\partial T, \sigma)$.
\end{lemma}
\begin{proof}
Let  $k\geq 1$ be such that $v\in V_n$ is in a cycle of length $k$. If the cycle is not stable then there exist $m>n$, words $u,w\in\{0,\cdots, d-1\}^{m-n}$ and integers $1\leq i,j<k$ such that $\sigma^{j}(v)u$ and $\sigma^i(v)w$ are not in the same cycle in $V_m$. This implies that the $\sigma$-orbit of every $x\in [\sigma^j(v)u]$ avoids the cylinder set $[\sigma^{i}(v)w]$, and so it is not minimal in $X$. Thus the minimal components  of the elements in $[\sigma^j(v)u]$ are strictly contained in $X$. On the other hand, if $v\in V_n$ is in a  stable cycle of length $k$, then the orbit of every $x\in X$ is contained in $X$ and intersects  any cylinder set contained in $X$, which implies that $X$ is minimal. 
\end{proof}

\begin{remark}
{\rm  The set $X$ in Lemma \ref{simple} is closed and open (clopen) since it is a finite union of clopen cylinder sets. Observe that for any clopen minimal component $X$ of $(\partial T, \sigma)$  there exist $n\geq 1$ and $v\in V_n$ such that $X=\bigcup_{j=1}^{k}\sigma^j[v]$. Thus Lemma \ref{simple} implies that every clopen minimal component is determined by stable cycles.}
\end{remark}

Recall from Definition \ref{defn-settled} that we say that $\sigma\in \Aut(T)$ is \emph{settled} if 
$$
\lim_{n\to \infty}\frac{|\{v\in V_n: v \mbox{ is in a stable cycle}\}|}{|V_n|}=1.
$$

As a consequence of Lemma \ref{simple} we can give another characterization of settled elements, using the uniform (Bernoulli) measure $\mu$ on $\partial T$, defined for every $n\geq 1$ and $v\in V_n$ by
$$
\mu([v])=\frac{1}{|V_n|}.
$$

\begin{lemma}\label{almosteverywhere}
An automorphism $\sigma\in \Aut(T)$ is settled if and only   there exists a family $\{C_i\}_{i\in I}$ of $\sigma$-minimal clopen sets such that 
$$
\mu\left(\partial T- \bigcup_{i\in I}C_i \right)=0.
$$
\end{lemma}
\begin{proof} 
Let $\{C_i\}_{i\in I}$ be the collection of minimal components which are clopen ($I$ could be empty). For every $n\geq 1$ let $J_n$ be the subset of vertices of $V_n$ which are in a stable cycle ($J_n$ could be empty).  We have
$$\bigcup_{i\in I}C_i=\bigcup_{n\geq 1}\bigcup_{v\in J_n}[v].$$
Since $\bigcup_{v\in J_n}[v]\subseteq \bigcup_{w\in J_{n+1}}[w]$, we get
$$
\mu\left( \bigcup_{i\in I}C_i\right)=\lim_{n\to \infty}\frac{|J_n|}{|V_n|},
$$  
 which is equal to $1$ if and only if $\sigma$ is settled. 
\end{proof}
Since $\sigma$ is an isometry, $\partial T$ is equal to the disjoint union of the minimal components of $(\partial T, \sigma)$. Thus Lemma \ref{almosteverywhere} is equivalent to the following.

\begin{prop}
An element $\sigma\in \Aut(T)$ is settled if and only for $\mu$-almost all $x$ in $\partial T$ the closure of its $\sigma$-orbit is clopen in $\partial T$.
\end{prop}

\begin{ex}\label{ex-settled-point}
{\rm
We give an example of a settled element with a fixed point. Let $T$ be a binary tree, and let $\eta$ be the non-trivial permutation of $V_1$. Recursively define an odometer $a = (a,1)\eta$ and an element $b = (a,b)$, see Section \ref{subsec-binaryautom} for notation. By construction $b$ acts as the odometer $a$ on the subtree of $T$ which lies above the vertex $0$, and on each subtree lying above the vertex $1^{n-1}0$ for $n \geq 1$, where $1^{n-1}$ denotes the concatenation of $n-1$ digits $1$. Since the restriction $a|V_n$ consists of a single cycle for each $n \geq 1$, every vertex in $V_n$ except for the rightmost vertex $1^n$ is in a stable cycle of $b$. The path in $T$, consisting of the rightmost vertices at each level is a fixed point of $b$. Therefore, $b$ is settled with a single fixed point. 

This construction can be modified to obtain an element with a countable number of fixed points by taking $c = (b,c)$ with $b$ as above. Then $c$ acts as $b$ on the subtree above $0$, and every subtree above $1^{n-1}0$. In each such subtree, a single path is fixed, and every vertex which is not in this path is in a stable cycle. The path through the vertices $1^n$, $n \geq 1$, is fixed by $c$. Thus $c$ is settled with a countable number of fixed points.
}
\end{ex}

Recall from Definition \ref{defn-settled} that a settled element $\sigma\in \Aut(T)$ is \emph{strongly settled} if there exists $n\geq 1$ such that every vertex in $V_n$ is in a stable cycle. 

As a direct consequence of Lemma \ref{simple} and the compactness of $\partial T$ we obtain the characterization of the closures of orbits for strongly settled elements.

\begin{prop}\label{strongly-settled-characterization}
Let $\sigma\in \Aut(T)$. The following are equivalent:
\begin{enumerate}
\item $\sigma$ is strongly settled.
\item For every $x\in \partial T$ the closure of the $\sigma$-orbit of $x$ is clopen.
\item $\sigma$ has  finitely many minimal components.
 \end{enumerate}
\end{prop}
 
 \begin{ex}\label{ex-minimal}
 {\rm
 Let $a \in \Aut(T)$ be a minimal element. Then the $a$-orbit of every $x \in \partial T$ is dense in $\partial T$, and so $a$ has a single minimal component. Then $a$ is strongly settled.
 }
 \end{ex}
 
\subsection{Settled elements in $\Aut(T)$}\label{subsec-AutTsettled} We show next that settled elements are abundant in $\Aut(T)$, that is, they form a dense subset of $\Aut(T)$. We assume that $T$ is a $d$-ary tree. The proof is constructive and easily generalizes to the case when $T$ has arbitrary spherical index. Recall that $D(\tau, \sigma)$ is the distance between $\tau, \sigma\in \Aut(T)$ defined in Equation  (\ref{eq-AutTmetric}).
 
\begin{lemma}\label{density-1}
For every $\tau\in \Aut(T)$ and every $n\geq 1$, there exists a strongly settled element $\sigma$ of $\Aut(T)$ such that $D(\tau, \sigma)\leq \frac{1}{2^n}$. 
In other words, the set of (strongly) settled elements of $\Aut(T)$ is dense in $\Aut(T)$.
\end{lemma} 

\begin{proof}
Let $\tau\in \Aut(T)$ and $\tau_n$ be the restriction of $\tau$ to $V_n$. Let $n\geq 1$ and let $C_1,\cdots, C_k$ be the cycles of $\tau_n$ in $V_n$. We choose $\sigma\in \Aut(T)$ as follows:  the restriction $\sigma_n$ of $\sigma$ to $V_n$ is equal to $\tau_n$. For every $m>n$, the restriction $\sigma_m$ of $\sigma$ to $V_m$ must verify the following: let $D_{m,1}, \cdots, D_{m,k_m}$ be the cycles of $\sigma_{m-1}$ in $V_{m-1}$. For $1\leq i \leq k_m$, let $v_{1}, \cdots , v_{\ell}$ be the vertices in $D_{m,i}$ ordered such that $\sigma_{m-1}(v_{j})=v_{j+1}$ and $\sigma_{m-1}(v_{\ell})=v_{1}$, for every $1\leq j <\ell$.   We take $\sigma_{m}$ to be the permutation on $V_m$ whose restriction to $\{vs: v\in D_{m,i}, 0\leq s<d\}$, for any $1 \leq i \leq k_m$, is given by
$$
\sigma_m(v_js)=v_{j+1}s \mbox{ for every } 1\leq j <\ell, \quad \sigma_m(v_\ell s)=v_1(s+1), \mbox{ for every } 0\leq s < d-1
$$
and
 $$
\sigma_m(v_j(d-1))=v_{j+1}(d-1) \mbox{ for every } 1\leq j <\ell, \quad \sigma_m(v_\ell(d-1))=v_10.
$$
Observe that the restriction of $\sigma_m$  to $\{vs: v\in D_{m,i}, 0\leq s<d\}$ is transitive, for every $1\leq i\leq k_m$. Since this is true for every $m>n$, we get that every vertex in $V_n$ lies in a stable cycle of $\sigma$, which shows that $\sigma$ is strongly settled.   

On the other hand the construction of $\sigma$ ensures that $d(\sigma x, \tau x)\leq \frac{1}{2^n}$, for every $x\in \partial T$, which means that $\sigma$ is at distance at most $\frac{1}{2^n}$ from $\tau$.
\end{proof}

\begin{remark}\label{nonsettled-aperiodic}
{\rm We say that $x\in \partial T$ is $\sigma$-aperiodic if the orbit of $x$ under the cyclic group generated by $\sigma$ is not periodic. We observe that if $\sigma$ is settled then   $\mu$-almost every $x\in \partial T$ is $\sigma$-aperiodic. If $\sigma$ is strongly settled then every $x\in \partial T$ is $\sigma$-aperiodic. 
}
\end{remark}

We note that it is possible to construct examples of elements in $\Aut(T)$ which are not settled and completely aperiodic. 

\begin{ex}\label{ex-doublecycle}
{\rm
Let $T$ be the binary tree, and let $h\in \Aut(T)$ be such that the length of every cycle doubles at levels $V_n$ for $n$ odd, and stays the same for $n$ even. Such an element $h$ has no stable cycles, which implies that $h$ is not settled. On the other hand, since the length of cycles in $V_n$ is increasing with $n$, the $h$-orbit of every point in $\partial T$ is aperiodic (infinite). We claim that the closure of each such orbit is not open, since it does not contain any cylinder sets. Indeed, let $x = (x_1,x_2, \ldots) \in \partial T$, then $x \in [x_n]$ for $n \geq 1$. Every cylinder set $[x_n]$ is a union of the sets $[x_n0]$ and $[x_n1]$. Suppose that $n$ is odd and for definitiveness, assume that $x_{n+1} = x_n0$. Then the cycle in $V_{n+1}$ which contains $x_{n+1}$ has the same length as the cycle in $V_n$ containing $x_n$. It follows that the $h$-orbit of $x$ never visits the cylinder set $[x_n1]$, and so its closure intersects but does not contain the cylinder set $[x_n]$. Since $n$ is arbitrary and $[x_n]$ for $n$ even contains cylinder sets $[x_m]$, $m > n$ for $m$ odd, this argument shows that the closure of any $h$-orbit in $\partial T$ is not open.
}
\end{ex}

\begin{remark}\label{remark-settled}
{\rm
The proof of Lemma \ref{density-1} implicitly uses the well-known fact that $\Aut(T)$ is isomorphic to the \emph{wreath product} $[S_d]^\infty$ of an infinite number of copies of the symmetric group on $d$ symbols. To recall the definition of the wreath product, let $G$ and $H$ be finite groups acting on the sets $X$ and $Y$ respectively, and let $f: X \to H$ be a function. An element $(g,f) \in G \ltimes H^X$ of the wreath product acts on the product $X \times Y$ in such a way that the action on $X$ is that of $g$, and the action on $\{x\} \times Y$, $x \in X$, is determined by the values of the function $f$ at points in $X$. In particular, $f(x)$ can be any permutation in $H$, and for different $x_1,x_2 \in X$ the values of $f$ at these points are independent.

Let $T$ be a $d$-ary tree with $\Sigma_n = \Sigma_1$, and let $H$ be a finite group acting transitively on $\Sigma_1 = V_1$. Suppose there is $g \in H$ such that the action of the cyclic group $\langle g \rangle$ is transitive on $\Sigma_1$. Then the infinite wreath product $[H]^\infty$ acts minimally on $\partial T = \prod_{n \geq 1} \Sigma$, and the proof of Lemma \ref{density-1} generalizes straightforwardly to show that $[H]^\infty$ is densely settled. It follows that if the image of an arboreal representation $\rho_{f,\alpha}: {\rm Gal}(\overline K/K) \to \Aut(T)$, for a polynomial $f(x)$ of degree $d$ and some $\alpha \in K$, has finite index in such an infinite wreath product $[H]^\infty$, then it is densely settled. Thus the conjecture of Boston and Jones is true for such representations.

In the current literature on arboreal representations, there are many examples of arboreal representations whose image has finite index in the wreath product of groups, as described above. Examples of arboreal representations which have finite index in $\Aut(T)$ can be found, for instance, in the survey \cite{Jones2013} and other publications. If the image of an arboreal representation has infinite index in $\Aut(T)$, then it is more difficult to determine if it is densely settled. We now show how to find examples of arboreal representations which have infinite index in $\Aut(T)$ and which are conjugate to the infinite wreath product $[H]^\infty$ with $H$ containing a transitive permutation. The images of such arboreal representations are densely settled by the argument above.

If a profinite subgroup of $\Aut(T)$ is topologically finitely generated, then it has infinite index in $\Aut(T)$, see for instance \cite{Jones2013}.
A criterion for when that happens for infinite wreath products may be found in \cite{Bondarenko2010}. For example, this criterion is satisfied when 
 the group $H$ has no non-trivial normal subgroups. This is true, for instance, if $H = A_d$ for $d \geq 5$, where $A_d$ denotes the alternating group on $d$ symbols. The group $A_d$ contains a transitive cyclic subgroup, and so a transitive permutation, when $d > 4$ and $d$ is odd, see \cite[Table 1]{LP2012}. Paper \cite{BEK2018} specifies a class of maps, called the normalized (dynamical) Belyi polynomials, for which the profinite geometric iterated monodromy group is isomorphic to the infinite wreath product $[A_d]^\infty$ where $d$ is the degree of the map. Thus by the argument above, if $d > 4$ and $d$ is odd, then the profinite geometric iterated monodromy group of a normalized Belyi polynomial is densely settled. 
 
In \cite[Section 3]{JKMT2015}, the authors give conditions under which the arithmetic iterated monodromy group associated to a rational function $r(x)$ of degree $d \geq 2$ is isomorphic to the infinite wreath product $[H]^\infty$ for $H$ finite (although $H$ is not specified in these results). The study of permutation groups with a transitive cyclic subgroup is in itself a topic of research in Group Theory, and a list of examples and the description of the properties of these groups can be found, for instance in \cite{LP2012}, see also the references therein. If the wreath product of groups in \cite{LP2012} can be obtained as the Galois groups of field extensions in \cite{JKMT2015}, then we will have another class of arboreal representations, for which the conjecture of Boston and Jones is true.
}
\end{remark}

\subsection{Stability of the settled  property}\label{stab-settled}  In general the property of settledness of elements is not very well-behaved, since the product of two settled elements need not be settled, while the product of two elements which are not settled can be settled, see Example \ref{ex-products}. However, as we show in this section, the powers of a settled element are settled, and the settledness is preserved under conjugacy.

Recall from Section \ref{subsec-dynsys} that for $h\in\Aut(T)$ and a point $x \in \partial T$ the $h$-orbit of $x$, denoted by $\langle h \rangle x$, is the orbit of $x$ under the action of the subgroup $\langle h\rangle$ generated by $h$.

\begin{lemma}\label{settled-iteration}
An automorphism $h \in \Aut(T)$ is (strongly) settled if and only if  $h^k$ is (strongly) settled, for every $k\in\mathbb{Z}-\{0\}$.
\end{lemma}
\begin{proof}
Suppose that $h$ is (strongly) settled. Let $x\in \partial T$ be and let $\ell\in\ZZ-\{0\}$. Observe that
$$
\langle h\rangle x=\bigcup_{r=0}^{\ell-1}\langle h^\ell \rangle h^rx, 
$$
which implies, since the union is finite, that
$$
\overline{\langle h\rangle x}=\bigcup_{r=0}^{\ell-1}\overline{\langle h^\ell \rangle h^rx}.
$$
Recall from Section \ref{subsec-dynsys} that, for $g \in \Aut(T)$, a minimal component of the action of $\langle g \rangle$ is a closed subset $Y$ of $\partial T$ such that $Y$ is invariant under the action, and such that the $g$-orbit of every $y \in Y$ is dense in $Y$. In a degenerate case, $Y$ may be a fixed point of the action. In our setting, since $h^\ell$ is an isometry, for every $y \in \overline{\langle h^\ell\rangle h^rx}$ the points in the $h^\ell$-orbit of $y$ stay at the same distance from the points in the $h^\ell$-orbit $h^rx$, and so if the $h^\ell$-orbit of $h^rx$ accumulates at a point $z \in \partial T$, then the $h^\ell$-orbit of $y$ accumulates at $z$ as well. Thus
the set $\overline{\langle h^\ell\rangle h^rx}$ is a minimal component of $h^\ell$, for every $0\leq r<\ell$. Then there exists a subset $I\subseteq \{0,\cdots, \ell-1\}$ such that 
$$\overline{\langle h\rangle x}=\bigcupdot_{r\in I}\overline{\langle h^\ell\rangle h^rx}. $$
Note that this union is a disjoint union of sets. 

Now suppose that the set $\overline{\langle h\rangle x}$ is clopen. Then every  $h^\ell$-minimal component contained in the finite disjoint union $\overline{\langle h\rangle x}$ is clopen. This implies that if  $x\in \partial T$ is in a clopen  $h$-minimal component then $x$ is in a clopen $h^\ell$-minimal component. Thus if $h$ is (strongly) settled then $h^\ell$ is (strongly) settled.  

Let $k\in \ZZ-\{0\}$. Suppose that $h^k$ is (strongly) settled. Since for every $x\in \partial T$ the $h^k$-orbit of $x$ is included in the $h$-orbit of $x$, then we get that the $h^k$-minimal component containing $x$ is included in the $h$-minimal component containing $x$. Thus if the minimal component of $x$ with respect to $h^k$ is clopen, the minimal component of $x$ with respect to $h$ is also clopen.  Then by Lemma \ref{almosteverywhere} we conclude that $h$ is settled.
\end{proof}

\begin{ex}\label{ex-products}
{\rm From Lemma \ref{settled-iteration} we have that if $h$ is (strongly) settled, then $h^{-1}$ is also (strongly) settled. Since $h^{-1}h =id$ is not settled, we conclude  that the product of two (strongly) settled elements is not necessarily (strongly) settled. 

On the other hand, the product of two non-settled elements can be (strongly) settled. Let $u_1 = \eta$ be the non-trivial permutation of two elements, and define $u_2 = (u_1,u_2)$. Then both $u_1$ and $u_2$ have order $2$, so they are not settled. However, the product $u_1u_2 = \eta(u_1,u_2)$ is minimal, so it is strongly settled by Example \ref{ex-minimal}. We discuss the action of the group generated by $u_1$ and $u_2$ in detail in Section \ref{sec-dihedral}.}
\end{ex}

The property of being (strongly) settled is preserved under conjugacy in $\Aut(T)$ as shown in the next result.
This result is a direct consequence of the fact that the minimal components of a dynamical systems are preserved under conjugacy. 

\begin{prop} \label{settled-preserved-conjugacy}
Let $w, h\in \Aut(T)$. Then $h$ is (strongly) settled if and only if $whw^{-1}$ is (strongly) settled.
\end{prop} 
\begin{proof}
Let $x\in \partial T$. Since $\langle w h w^{-1}\rangle=w\langle h\rangle w^{-1}$, we have $\langle w h w^{-1}\rangle wx= w \langle h \rangle x$. In other words, the $whw^{-1}$-orbit of $wx$ is equal to the image by $w$ of the $h$-orbit of $x$.  This implies that the closure of the $whw^{-1}$-orbit of $wx$ is equal to the image by $w$ of the closure of the $h$-orbit of $x$. Thus if the closure of the $h$-orbit of $x$ is clopen then the closure of the $whw^{-1}$-orbit of $wx$ is also clopen. Then
$$
w \{x\in \partial T: \overline{\langle h\rangle x} \mbox{ is clopen }\} \subseteq \{ x\in \partial T: \overline{\langle whw^{-1} \rangle x } \mbox{ is clopen}\}.
$$
Since the uniform measure $\mu$ is invariant under the action of $\Aut(T)$ we conclude that if $h$ is (strongly) settled then $w h w^{-1}$ is (strongly) settled.  From this we also get that $h$ is (strongly) settled if and only if $w h w^{-1}$ is strongly settled.  
\end{proof}

\begin{cor}\label{cor-conjugate}
Let $\Gamma\subseteq \Aut(T)$ and $w\in \Aut(T)$. Then $\Gamma$ is densely settled if and only if $w\Gamma w^{-1}$ is densely settled. 
\end{cor}


\begin{prop}\label{densely-settled-prop}
The following statements are equivalent:
\begin{enumerate}
\item $h$ is (strongly) settled.
\item $\overline{\langle h\rangle}$ contains a (strongly) settled element.
\item The set of (strongly) settled elements of  $\overline{\langle h\rangle}$ is dense in  $\overline{\langle h\rangle}$.
\end{enumerate}
\end{prop}
\begin{proof}
It is obvious that  (3) implies (2). The implication (1) to (3) is consequence of Lemma \ref{settled-iteration}.

It is not difficult to see that if $\zeta \in \overline{\langle h \rangle}$ then the minimal components of $\zeta$ are contained in the minimal components of $h$. Thus, if the $h$-minimal component  of $x$ is not clopen, then neither is the $\zeta$-minimal component of $x$. From this we deduce that if $h$ is not (strongly) settled then $\overline{\langle h\rangle}$ has no (strongly) settled elements. This shows (2) that implies (1).
\end{proof}

\begin{remark}\label{cor-conjugate-bis}{\rm 
Let $\cH_1,\cH_2$ be profinite groups acting effectively on $\partial T$, that is, we have that the homomorphisms $\cH_i \to {\rm Homeo}(\partial T)$ are injective for $i = 1,2$. Then we can identify $\cH_1,\cH_2$ with subgroups of $\Aut(T)$. Next, suppose the actions are conjugate by an isometry as dynamical systems, that is, there exists a topological isomorphism $\Phi: \cH_1 \to \cH_2$ and an isometry $\phi: \partial T \to \partial T$, such that $\phi (h x) = \Phi(h)(\phi(x))$ for all $h \in \cH_1$ and all $x \in \partial T$. We note that this is equivalent to the subgroups $\cH_1,\cH_2 \subset \Aut(T)$ being conjugate in $\Aut(T)$. Indeed, since $\phi$ is an isometry, it must restrict to a permutation on each level $V_n$, and so it must map cylinder sets in $\partial T$ to cylinder sets in $\partial T$. It follows that $\phi$ preserves the structure of the tree $T$ and defines an automorphism $w_\phi \in \Aut(T)$. It follows that $\Phi: \cH_1 \to \cH_2$ is given by the conjugation by $w_\phi$.

We note that the settled property is not preserved under the topological isomorphism of profinite groups. To see an example, suppose $T$ is the binary tree, $a$ is a minimal element in $\Aut(T)$, and $h$ is the non-settled aperiodic automorphism in $\Aut(T)$  introduced   in Example \ref{ex-doublecycle}. The order of the restriction of $h$ to $V_n$ is equal to $2^k$ if $n=2k$ or $n=2k-1$, for each $n\geq 1$. Thus, according to Proposition \ref{conjugacy-elements},  the group $\overline{H}=\overline{\langle h\rangle}$ is isomorphic to the group given by the inverse limit of the following sequence:
$$ 
         \ZZ/2\ZZ\leftarrow \ZZ/2\ZZ\leftarrow \ZZ/2^2\ZZ\leftarrow\ZZ/2^2\ZZ\leftarrow\ZZ/2^3\ZZ\leftarrow\ZZ/2^3\ZZ\leftarrow \cdots. 
$$
This group is topologically isomorphic to $\ZZ_2$, which is by Proposition   \ref{conjugacy-elements}  isomorphic to $ \overline{\langle a \rangle }$. This shows that $ \overline{\langle a \rangle }$ and $\overline{H}$ are topologically isomorphic.  However, their actions on $\partial T$ cannot be conjugate since the action of $ \overline{\langle a \rangle }$ is minimal, while the action of $\overline H$ is not.  By Proposition \ref{densely-settled-prop}, $\overline{H}$  is not densely settled  whereas $ \overline{\langle a \rangle }$ is densely settled. 

In the rest of the article, in all cases when we say that two groups are (topologically) isomorphic they are (topologically) isomorphic \emph{and} their actions are conjugate by an isometry, as in the first paragraph of this remark. The term `isomorphism' in this context is often used to describe the subgroups of $\Aut(T)$ algebraically, for instance, as wreath products or other known algebraic groups, and by construction the isometry in question is often the identity map in $\Aut(T)$.}
\end{remark}

 
 \section{Maximal tori with densely settled normalizers}\label{DenselyNormalizer}
 

 In this section we assume that $d\geq 2$ is a prime number, and $T$ is the $d$-ary tree, see Section \ref{rooted} for the necessary background on trees. In particular, the spherical index of $T$ is $m_T = (d,d, \ldots)$, and for each vertex level set its cardinality is given by the number $p_n= |V_n| = d^n$. Then the set of numbers coprime with the entries of the spherical index of $T$ defined in \eqref{eq-coprime} becomes
 $$
 I(T)=\{k\geq 1: k \mbox{ is not divisible by } d\}.
 $$
 Recall that for $a \in \Aut(T)$ we denote by $\langle a \rangle$ the cyclic subgroup of $\Aut(T)$ generated by $a$, and we denote by $ \overline{\langle a \rangle }$ the closure of $\langle a \rangle$ in $\Aut(T)$. If $a$ is minimal, then by Lemma \ref{minimal-elements} $ \overline{\langle a \rangle }$ is a maximal torus in $\Aut(T)$.
 In this section we study the properties of the normalizer of such a maximal torus, and give a dynamical proof of Theorem \ref{thm-main1}. 
 
 We restate the theorem now for the convenience of the reader.
 
 \begin{thm}\label{thm-Nads}
Let $d>1$ be a prime number and let $T$ be the $d$-ary tree. For every minimal automorphism $a\in \Aut(T)$ the normalizer $N( \overline{\langle a \rangle })$ of the maximal torus $ \overline{\langle a \rangle }$ is densely settled.
\end{thm}

 The proof proceeds by carefully examining elements in the rational spanning set of $N( \overline{\langle a \rangle })$. Recall that the sets $G_k$, for $k \in I(T)$, defined in Section \ref{sec_normalizer}, are the cosets in $N( \overline{\langle a \rangle })$ given by
   $$G_k = \{w \in N( \overline{\langle a \rangle }) :  w a w^{-1} = a^k\},  \, k \in I(T).$$
 A rational spanning set of $N( \overline{\langle a \rangle })$ defined in Definition \ref{defn-ratspan}, consists of the union of countable subsets in $G_k$, $k \in I(T)$. If the spherical index of $T$ is bounded (for $d$-ary trees it is constant), then the rational spanning set is dense in $N( \overline{\langle a \rangle })$. In Remark \ref{remark-degenerate} we gave an example of a tree with unbounded spherical index, for which the rational spanning set is degenerate, that is, it is not dense in $N( \overline{\langle a \rangle })$. Our method to prove Theorem \ref{thm-main1} does not extend to trees with degenerate rational spanning sets. 
 
 Our method also does not extend to the case when the spherical index of $T$ is bounded but not constant, since in that case we do not have an analog of Lemma \ref{auxiliar-5p} below.
 
\subsection{A few technical preliminaries} 

 Let $a \in \Aut(T)$ be minimal. Then $ \overline{\langle a \rangle }$ acts freely and transitively on $\partial T$, and so there is a bijection between $ \overline{\langle a \rangle }$ and $\partial T$ which depends on a choice of a point in $\partial T$. We make this identification explicit. 

As described in Section \ref{rooted}, for every $n\geq 1$  the vertices in $V_n$ are labelled by the words in $\prod_{j=1}^n\{0,\cdots, d-1\}$. Since $a$ is transitive on $V_n$, the map $\phi_n:V_n\to \ZZ/d^n\ZZ$ given by
 $$
 a^j(0^n)=j+d^n\ZZ, \mbox{ for every } j\geq 0,
 $$
 is a well-defined bijection. The action of $a$ on $\ZZ/d^n\ZZ$ reads as
 $$
 a(j+d^n\ZZ)=j+1 +d^n\ZZ,
 $$
 or, which is equivalent,
 $$
 a(j)=j+1 \mod d^n, \mbox{ for every } 0\leq j<d^n.
 $$

 Let $k\in I(T)$ and $\sigma\in G_k$.    For $n\geq 1$, and $v\in V_n$ (or for $v\in \ZZ/d^n\ZZ$, which is equivalent) the relation $\sigma a=a^k\sigma$ together with the identification described above implies that
 \begin{equation}\label{equation-k3}
 \sigma a^{v}(0)=\sigma(v)=\sigma(0)+kv \mod d^n.
 \end{equation}
 
 Applying Equation (\ref{equation-k3}) $p>1$ times we get
 \begin{equation}\label{equation-k4}
 \sigma^p(v)=k^pv+ \sigma(0)\sum_{i=0}^{p-1}k^i=\left\{ \begin{array}{ll}   
                                                                                       k^pv+\sigma(0)\left(\frac{k^p-1}{k-1} \right) \mod d^n & \mbox{ if  } k>1\\
                                                                                        v+\sigma(0)p & \mbox{ if } k=1
                                                                                         \end{array}\right. 
 \end{equation}

 A proof of the next lemma can be found at \cite{MP}.

\begin{lemma}\label{auxiliar-5p}
Let $k > 1$ and $n\geq 1$  be integers. Consider the following geometric series
$$
 r(n): = \sum_{i=0}^{n-1}k^i= \left(\frac{k^n-1}{k-1} \right).
 $$ 
 Then for any $n \geq 1$ the following holds.
 \begin{enumerate}
 \item If $d$ is an odd prime and $k=ds+1$, for some $s\geq 0$, then  the highest power of $d$ dividing $r(n)$ is equal to the highest power of $d$ dividing $n$.
 
 \item if $d=2$ and $k=4s+1$, for some $s\geq 0$, then the highest power of $2$ dividing $r(n)$ is equal to the highest power of $2$ dividing $n$.
  \end{enumerate}
 \end{lemma}

 We say that an integer  $k>1$ {\it satisfies the conditions of Lemma \ref{auxiliar-5p}} if $k= 4s+1$ whenever $d=2$ and $k= ds +1$  when $d$ is an odd prime, for some $s \in \ZZ$.
 
  \subsection{Cycle structure of elements in $N( \overline{\langle a \rangle })$}\label{subsec-proofsNa}
 
Recall that for integers  $m\geq 0$ and $k\geq 1$, the element $\sigma_{m,k}\in G_k$ of the rational spanning set is the unique automorphism in $G_k$ verifying  $\sigma_{m,k}\xi=a^m \xi$, where $\xi$ is an infinite sequence of $0$'s. In particular, using \eqref{equation-k3} we get in $\ZZ/d^n \ZZ$ that $\sigma_{m,k}(0) = m$. We now determine what points in $\partial T$ are in stable cycles of $\sigma_{m,k}$, depending on $k$ and $m$.

\begin{lemma}\label{density}
Let $k>1$ be an integer that satisfies the conditions of Lemma \ref{auxiliar-5p}. Let $m\geq 1$ be an integer which is not divisible by $d$.  Then $\sigma_{m,k}$ is minimal and so strongly settled.
\end{lemma} 
 
\begin{proof}
 Let $n\geq 1$. Since $m$ is not divisible by $d$, according to \eqref{equation-k4} and Lemma \ref{auxiliar-5p}, for $0\in V_n$ we have that if $p>0$ is the smallest positive integer such that
$$
\sigma_{m,k}^p(0)=m\left(\frac{k^p-1}{k-1} \right)=0 \mod d^n \mbox{ then } p=d^n. 
$$
This implies that the orbit of $0$ in $V_n$ by $\sigma_{m,k}$ is in a cycle of length $d^n$. Since $|V_n| = d^n$, $\sigma_{m,k}$ is transitive in $V_n$, that is, the restriction of $\sigma_{m,k}|V_n$ consists of a single cycle. Since this is true for any $n\geq 1$ we conclude that $\sigma_{m,k}$ is minimal. Therefore, $\sigma_{m,k}$ is strongly settled by Example \ref{ex-minimal}.
\end{proof}

  \begin{remark}
  {\rm  Recall that $\sigma_{m,1}\xi=a^m\xi$ by definition. Thus by Lemma \ref{integer-minimality} we have $\sigma_{m,1}$ is minimal if and only if $m$ is not divisible by $d$.}
  \end{remark}
 
\begin{lemma}\label{cycle-0}
Let $k>1$  be an integer that satisfies the conditions of Lemma \ref{auxiliar-5p}. Let $m\geq 1$ be an integer which is divisible by $d$.   Let  $j\geq 1$ be the biggest integer such that $d^j$ divides $m$. Then  for every $n>j$ the vertex $0$ is in a stable cycle of $\sigma_{m,k}$ of length $d^{n-j}$.
 \end{lemma}
 
 \begin{proof}
Let $p\geq 1$ be such that $\sigma_{m,k}^p(0)=0 \mod d^n$, that is, $0$ is in a cycle of length $p$ in $V_n$. Then \eqref{equation-k4} implies that $d^n$ divides $mr(p) = m\left(\frac{k^{p}-1}{k-1}\right)$. By the choice of $j$ and Lemma \ref{auxiliar-5p}, we get that the smallest $p$ such that $d^{n-j}$ divides $r(p)$ is $d^{n-j}$, and $d^{n-j}$ is the highest power of $d$ which divides $r(d^{n-j})$. This implies that $0$ is in a cycle of length $d^{n-j}$ of $\sigma$ in $V_n$, which lifts to an open path in $V_{n+1}$. Since $n>j$ is arbitrary, we deduce that $0$ is in a stable cycle of $\sigma_{m,k}$ in $V_n$,  for every $n>j$. 
 \end{proof}
 
In Lemma \ref{cycle-0} for a given $\sigma_{m,k}$ with $m \geq 1$ we studied the stability of the cycle containing $0$ for the restriction of $\sigma_{m,k}$ to $V_n \cong \ZZ/d^n\ZZ$. Lemma \ref{cycle-2} below investigates a similar question for an arbitrary vertex in $V_n$.

\begin{lemma}\label{cycle-2}
Let $k>1$  be an integer that satisfies the conditions of Lemma \ref{auxiliar-5p}.  Let $m \geq 1$ be an integer which is divisible by $d$ and let  $j\geq 1$ be the biggest integer such that $d^j$ divides $m$.  Then the automorphism $\sigma_{m,k}$ is  settled, and for every $n \geq 1$ the lengths of cycles of the restriction $\sigma_{m,k}|V_n$ are powers of $d$.

\end{lemma}
\begin{proof}
  For simplicity we will write $\sigma$ instead of $\sigma_{m,k}$. 
  
  Let $n>j$ and $v\in V_n$. Let $q\geq 1$ be a number not divisible by $d$  such that $m=d^j q$. By Lemma \ref{cycle-0}  we can assume that $v\neq 0$. Let $p\geq 1$ be such that $\sigma^p(v)=v \mod d^n$. Then $d^n$ divides  \begin{align}\label{eq-sigmapv}\sigma^p(v) - v = (m+(k-1)v)\left(\frac{k^{p}-1}{k-1}\right) = (m+(k-1)v) r(p),\end{align} 
  and, since $k$ satisfies Lemma \ref{auxiliar-5p},  $d^{n-1}$ divides 
  $$\left(d^{j-1}q+\frac{(k-1)}{d}v\right)r(p) = (d^{j-1}q + sv) r(p),$$
  where $s$ is such that $k=ds+1$ (the notation for $s$ coincides with the notation in Lemma \ref{auxiliar-5p},(1), i.e. when $d$ is an odd prime; when $d=2$ in Lemma \ref{auxiliar-5p},(2), we have $k = 2(2s)+1$, i.e. $s$ in this proof corresponds to $2s$ in Lemma \ref{auxiliar-5p},(2)).
  
   
   Let $i\geq 0$ be the biggest integer such that $d^i$ divides $sv$. In the arguments below, we use the fact that $sv/d^i$ is not divisible by $d$, and for any $\ell < i$ the integer $sv/d^{\ell}$ is divisible by $d$.
   
  {\it Case 1:}  If $i <j-1$ then $\frac{1}{d^i}(d^{j-1}q+sv)$ is not divisible by $d$, which implies that $d^{n-(i+1)}$ must divide $r(p)$. From Lemma \ref{auxiliar-5p} the smallest $p$ for which that happens is $p = d^{n-(i+1)}$. This implies that $v$ is in a cycle of length $d^{n-(i+1)}$ of $\sigma$ in $V_n$. Since this is true for any $n>j$, we conclude that $v$ is in a stable cycle of $V_n$.

   {\it Case 2:} If $i>j-1$ then $\frac{1}{d^{j-1}}(d^{j-1}q+sv)$
    is not divisible by $d$, which implies that $d^{n-j}$ must divide $r(p)$. From Lemma \ref{auxiliar-5p} we get $p=d^{n-j}$, which implies that $v$ is in a cycle of length $d^{n-j}$ of $\sigma$ in $V_n$. Since this is true for any $n>j$, we conclude that $v$ is in a stable cycle of $V_n$.

   {\it Case 3:} If $i=j-1$ then $sv=d^{j-1}r$, for some positive integer $r$ not divisible by $d$. Let $t\geq 0$ be the biggest integer such that $d^t$ divides the number $(q+r)$. If $t<n-j$ then 
    $d^{n-j-t}$ must divide  $r(p)$. Then from Lemma \ref{auxiliar-5p} we get $p=d^{n-j-t}$. Thus $v$ is in a cycle of length $d^{n-j-t}$ of $\sigma$ in $V_n$. 
   Since this is true for any $n>j$, we conclude that $v$ is in a stable cycle of $V_n$.  

Now suppose $t\geq n-j$, and write $q+r = d^t r'$ for an integer $r' > 0$ non divisible by $d$. Then
$$(d^{j-1}q+sv)r(p) = d^{j-1}(q+r)r(p) = d^{j+t-1}r' r(p).$$
Since $d^{j+t - 1} \geq d^{j+n-j-1} = d^{n-1}$, $d^{n-1}$ divides this expression for any $p$, in particular for $p = 1$. This implies that $v$ is a fixed point of $\sigma$ in $V_n$.

We now estimate the number of fixed points of $\sigma$ in $V_n$. According to Equation (\ref{eq-sigmapv}), the vertex $v\in V_n$ is a fixed point of $\sigma$ if $m+(k-1)v$ is divisible by $d^n$, which is equivalent to $d^{j-1}q + sv=d^{n-1}\ell$ for some $\ell\geq 1$. Thus to estimate the number of fixed points of $\sigma$ in $V_n$,  we need to count the number of possible $\ell$ such that $d^{n-1}\ell$ is equal to $d^{j-1}q + sv$, for some $v\in V_n$. For that let $\ell \geq 1$ be an integer such that $d^{j-1}q+sv=d^{n-1}\ell$. Solving for the second term, since $v < d^n$ we obtain the estimate
  $$0<  -d^{j-1}q +\ell d^{n-1} = sv = \frac{(k-1)}{d}v<\frac{(k-1)}{d}d^n,$$  
which implies  that 
  $$\ell<(k-1)+\frac{q}{d^{n-j}}.$$ 
Since $q$ is fixed we can assume that $n$ is sufficiently large so that $\frac{q}{d^{n-j}}<1$ and then $\ell \in \{1, \ldots , k-1\}$. This implies that in $V_n$ for $n$ large there are at most $k-1$ fixed points of $\sigma$, and then, there are at most $k-1$ vertices which are not in a stable cycle. We conclude that $\sigma$ is settled.
  \end{proof}

Thus Lemmas   \ref{density}, \ref{cycle-0} and \ref{cycle-2} yield the following statement.

\begin{prop}\label{densely-settled-0}
Let $k>1$  be an integer that satisfies the conditions of Lemma \ref{auxiliar-5p}. For every $m\geq 1$ the automorphism $\sigma_{m,k}$ is settled, and for every $n \geq 1$ the lengths of cycles of the restriction $\sigma_{m,k}|V_n$ are powers of $d$.
\end{prop}

\subsection{Proof of the densely settled property of $N( \overline{\langle a \rangle })$} Proposition \ref{densely-settled-0} shows that if $k = ds +1$, for some $s \geq 0$, then $\sigma_{m,k}$ is settled. We now show that $\sigma_{m,k}$ is settled for any $k \geq 1$ such that $\gcd(d,k) = 1$. Since by  Theorem \ref{thm-skeleton} the rational spanning set $\{\sigma_{k,m} \mid m \geq 1, \gcd(d,k) = 1\}$ is dense in $N( \overline{\langle a \rangle })$ this finishes the proof of Theorem 
\ref{thm-Nads}.

\begin{prop}\label{densidad-completa} 
For every integer  $k\geq 1$ such that $k$ is not divisible by $d$, and  every $m\geq 1$, the automorphism $\sigma_{m,k}$ is settled. 
\end{prop}

\begin{proof}
If $k = 1$, then $\sigma_{m,1} = a^m$ by Remark \ref{remark-k1}. Since $a$ is minimal, the orbit of every point under $\langle a \rangle$ is dense in $\partial T$. Thus $a$ has a single minimal component, and $a$ is strongly settled by Proposition \ref{strongly-settled-characterization},(3). Then by Lemma \ref{settled-iteration} $\sigma_{m,1}$ is strongly settled. 

For simplicity, we will write $\sigma$ instead of $\sigma_{m,k}$.

Let $k>1$ be an integer which is not divisible by $d$. We set $n=4$ if $d=2$ and $n=d$ if $d$ is an odd prime. Let $p=\varphi(n)$, where $\varphi$ is the Euler function. Then $k^p=ns+1$ for some $s\geq 1$.   

 Observe that $\sigma a =a^k \sigma$ implies
$$
\sigma^p a = a^{k^p} \sigma^p.  
$$
This implies  that $\sigma^p$ is in $G_{k^p}$. In addition,  using that $\sigma a^r=a^{rk}\sigma$ for every $r\geq 1$,  we get
$$
\sigma^p \xi=\sigma^{p-1}(\sigma \xi)=\sigma^{p-1} a^m \xi=\sigma^{p-2}a^{mk}\sigma \xi=\sigma^{p-2}a^{m(k+1)} \xi=a^{m(1+k+\cdots+ k^{p-1})} \xi,  
$$ 
which implies by Lemma \ref{lemma-initialvalue} and Definition \ref{defn-onorbits}, that $\sigma^p=\sigma_{m(1+k+\cdots+ k^{p-1}),k^p}$. From Proposition \ref{densely-settled-0} it follows that $\sigma^p$ is settled. Then by Lemma \ref{settled-iteration}, we conclude that $\sigma$ is settled.
\end{proof}

\subsection{Application:  minimal automorphisms of $N( \overline{\langle a \rangle })$.} 

 In this section we give a characterization of the minimal elements in $N( \overline{\langle a \rangle })$.

\begin{lemma}\label{minimal-final}
Let $k\geq 1$ be an integer which is not divisible by $d$. Let $m\geq 1$.  
\begin{enumerate}
\item If $d=2$ then the following statements are equivalent:
\begin{enumerate}
\item $\sigma_{m,k}$ is transitive on $V_2$.
\item   $m$ is odd and $k=4s+1$, for some $s\geq 0$.
\item $\sigma_{m,k}$ is minimal.
\end{enumerate}
 \item If $d$ is an odd prime then the following statements are equivalent:
 \begin{enumerate}
   \item $\sigma_{m,k}$ is transitive on $V_1$.
  \item   $m$ is not divisible by $d$ and $k=ds+1$, for some $s\geq 0$.
  \item $\sigma_{m,k}$ is minimal.
\end{enumerate}
\end{enumerate}
Thus $\sigma_{m,k}$ is minimal if and only if $m$ is not divisible by $d$ and $k$   satisfies the conditions of Lemma \ref{auxiliar-5p}.
\end{lemma}

\begin{proof}
Both in (1) and (2), Lemma \ref{density} gives the implication (b) to (c). The implication (c) to (a) follows from Lemma \ref{minimal-transitivity}. So it is enough to show the implication (a) to (b), in both cases.

For the case $d=2$, observe that if $\sigma_{m,k}$ is transitive on $V_2$, then it is transitive on $V_1$. For any prime $d$,
observe that if  $m\geq 1$ is divisible by $d$, then formula \eqref{equation-k4} yields $\sigma_{m,k}(0)=m=0 \mod d$.  This implies that $\sigma_{m,k}$ has a fixed point at $0$ in $V_1$. So if $\sigma_{m,k}$ is transitive on $V_1$, then $m$ cannot be divisible by $d$. 

 \medskip

{\it Case $d=2$.} Suppose that $k=4s+3$, for some $s\geq 0$. Then by \eqref{equation-k4} we get
$$
\sigma_{m,k}^2(0)=m\left(\frac{(4s+3)^2-1}{4s+3-1}\right)=m\left(\frac{(4s+3-1)(4s+3+1)}{4s+3-1}\right)=4m(s+1),  
$$ 
which implies that $\sigma^2_{m,k}(0)=0 \mod  2^2 $. Thus $\sigma_{m,k}^2$ has a fixed point at $0$ in $V_2$, and since $|V_2| = 4$, then $\sigma_{m,k}$ is not transitive on $V_2$. Thus if $\sigma_{m,k}$ is transitive on $V_2$ then $k=4s+1$ and $m$ is odd.  

\medskip

{\it Case $d$ an odd prime.} Suppose that $k=ds+r$ for some $2\leq r<d$ and $s\geq 0$. Then by \eqref{equation-k4}
\begin{equation}\label{eq-minimal}
\sigma_{m,k}^{d-1}(0)=m(1+(ds+r)+\cdots + (ds+r)^{d-2}).
\end{equation}
Observe that for every $1\leq \ell \leq d-2$ 
$$
(ds+r)^\ell=\sum_{i=0}^\ell \binom{\ell}{i} (ds)^ir^{\ell-i}=r^\ell+\sum_{i=1}^\ell \binom{\ell}{i}(ds)^ir^{\ell-i}=r^\ell +dN_\ell,
$$
where $N_\ell$ is an integer. Recall that by Fermat’s Little Theorem, for $r\neq 0 \mod d$ we have $r^{d-1}=1 \mod d$. Substituting the formula above and this into \eqref{eq-minimal} we get
$$
\sigma_{m,k}^{d-1}(0)=m(1+r+\cdots r^{d-2}+dN)=m\left(\frac{r^{d-1}-1 }{r-1} + dN\right)=m\left(d\frac{q}{r-1} + dN\right),   
$$
 where $N=\sum_{i=1}^{d-2}N_i$ and $q$ is the integer such that $r^{d-1}-1=dq$. 
  Since $\sigma_{m,k}^{d-1}(0)$ is an integer, $q$ is divisible by $r-1$ and
 $$\sigma_{m,k}^{d-1}(0)=0 \mod d. $$
 Thus $\sigma_{m,k}^{d-1}$ has a fixed point at $0$ in $V_1$, and so $\sigma_{m,k}$ cannot be transitive  on $V_1$.  Thus if $\sigma_{m,k}$ is transitive on $V_1$ then $m$ is not divisible by $d$ and $k=ds+1$, for some $s\geq 0$.
 \end{proof}
 
For $d \geq 2$ prime consider the set  $\mathcal{M}_d$ of all minimal elements in the rational spanning set, namely, by Lemma \ref{minimal-final} 
 $$
 \mathcal{M}_d=\{ \sigma_{m,k}: m >0 \mbox{ is not divisible by } d \mbox{  and } k= 1 \mod d    \}  \mbox{ for } d\geq 3.
 $$ 
 and
 $$
 \mathcal{M}_2=\{ \sigma_{m,k}: m >0 \mbox{ is odd }  \mbox{  and } k= 1 \mod 4    \}. 
 $$
The next proposition characterizes the minimal elements of $N( \overline{\langle a \rangle })$ as those in the closure of $\mathcal{M}_d$.
 
 \begin{prop}\label{prop-applminimal}
 Let $\sigma\in N( \overline{\langle a \rangle })$. The following statements are equivalent:
 \begin{enumerate}
 \item $\sigma$ is minimal.
 
 \item There exists $w\in \Aut(T)$ such that $\sigma=waw^{-1}$.
 
 \item $\sigma\in \overline{\mathcal{M}_d}$.
 \end{enumerate}
 \end{prop}
 \begin{proof}
 The equivalence between (1) and (2) was shown in Corollary \ref{Pink-result}.
 
 To show the implication (3) to (2), suppose that $\sigma\in \overline{\mathcal{M}_d}$. Then there exists a sequence $(\sigma_{m_i, k_i})_{i\geq 0}$ in $\mathcal{M}_d$ that converges to $\sigma$. Since every $\sigma_{m_i,k_i}$ is minimal, there exists $w_i\in \Aut(T)$ such that $\sigma_{m_i,k_i}=w_iaw_i^{-1}$. By compactness of $\Aut(T)$ there exists a subsequence $(w_{i_j})_{j\geq 0}$ of $(w_i)_{i\geq 0}$ that converges to some $w\in \Aut(T)$. This implies that $(\sigma_{m_{i_j}, k_{i_j}})_{j\geq 0}$ converges to $waw^{-1}$, which shows that $\sigma=waw^{-1}$. We conclude that $\sigma$ is minimal.    
 
 Suppose that $\sigma\in N( \overline{\langle a \rangle })$ is minimal. Since the rational spanning set is dense in $N( \overline{\langle a \rangle })$ (see Theorem \ref{thm-skeleton}), there exists a sequence $(\sigma_{m_i,k_i})_{i\geq 0}$ converging to $\sigma$.  We need to show that this sequence can be chosen to be in $\mathcal{M}_d$. Observe that there exists $i_0\geq 0$ such that for every $i\geq i_0$ the restrictions $\sigma|_{V_2}$ and $\sigma_{m_i,k_i}|_{V_2}$ coincide. Since $\sigma$ is transitive on $V_2$ (because it is minimal), $\sigma_{m_i, k_i}$ is transitive in $V_2$ (and then in $V_1$) too. Then Lemma \ref{minimal-final} implies that for every $i\geq i_0$     the automorphism $\sigma_{m_i,k_i}$ is in $\cM_d$, which shows that $\sigma\in \overline{\mathcal{M}_d}$.
 \end{proof}
 
 
 \section{Iterated monodromy groups associated to PCF polynomials} \label{sec-img}
 
 Let $f(x)$ be a quadratic polynomial with strictly pre-periodic post-critical orbit of length $r \geq 2$ with periodic cycle of length $r-s$, for $1 \leq s <r$. That is, $r$ and $s$ are the smallest positive integers such that $r>s \geq 1$ and $f^{r+1}(c) = f^{s+1}(c)$, where $c$ is the critical point of $f(x)$.
 
Consider the associated profinite arithmetic and geometric iterated monodromy groups $\fG_{\rm arith}(f)$ and $\fG_{\rm geom}(f)$, as defined in Section \ref{sec-arboreal}. In this section we study the densely settled property for $\fG_{\rm arith}(f)$ and $\fG_{\rm geom}(f)$ by considering the maximal tori in these groups, and the Weyl groups of these maximal tori. We recall necessary concepts and notation first.

Let $T$ be a binary tree, and recall that for $f(x)$ as above by \cite[Proposition 1.7.15]{Pink13} $\fG_{\rm geom}(f)$  is conjugate in $\Aut(T)$ to the closure of the group generated by the elements 
  \begin{align}\label{eq-gen}u_1 = \eta, u_i = (u_{i-1},1) \textrm{ if }2 \leq i \leq r \textrm{ and }i \ne s+1, u_{s+1} = (u_s,u_r). \end{align}
  The recursive notation used in formula \eqref{eq-gen} was explained in Section \ref{subsec-binaryautom}. We denote the countable group generated by \eqref{eq-gen} by $G$, and its closure in $\Aut(T)$ by $\overline G$. By Corollary \ref{cor-conjugate} $\fG_{\rm geom}(f)$ is densely settled if and only if $\overline G$ is densely settled.
  
As earlier in the paper, given an element $a \in \Aut(T)$ we denote by $\langle a \rangle$ the cyclic group generated by $a$, and by $ \overline{\langle a \rangle }$ the closure of $\langle a \rangle$ in $\Aut(T)$. We denote by $N( \overline{\langle a \rangle })$ the normalizer of $ \overline{\langle a \rangle }$ in $\Aut(T)$.
  
 We claim that $\overline G$ contains maximal tori. Indeed, the product $a_0=u_1u_2\cdots u_r$ is minimal by \cite[Proposition 3.1.3]{Pink13}, and so the closure $ \overline{\langle a_0 \rangle }$ is a maximal torus in $\overline G$. 
 
 Now let $a \in N(\overline G)$ be any minimal element. By Proposition \ref{conjugacy-elements} the maximal torus $ \overline{\langle a \rangle }$ is isometrically isomorphic to the dyadic integers $\ZZ_2$, and by Lemma \ref{lemma-norm2} its absolute Weyl group satisfies
   $$W( \overline{\langle a \rangle }) = \{G_\zeta \mid \zeta \in  \overline{\langle a \rangle } \textrm{ is minimal }\}, \textrm{ where }G_\zeta = \{w \in N( \overline{\langle a \rangle }) \mid w a w^{-1} = \zeta \}.$$
 We recall the following result about the structure of $N(\langle a \rangle)$. Denote by $\ZZ_2^{\times}$ the multiplicative group in $\ZZ_2$, and by $e^\times$ the unit element in $\ZZ_2^{\times}$.
 
 \begin{thm} \cite[Proposition 1.6.3]{Pink13}, \cite[Theorem 4.4.7]{RibZal10}\label{thm-Zstabilizer}
There is a group isomorphism 
  $$N( \overline{\langle a \rangle }) \to \ZZ^\times_2 \ltimes \ZZ_2,$$ 
which maps $ \overline{\langle a \rangle }$ onto $\{e^\times\} \times \ZZ_2$, and induces a group isomorphism 
  $$W( \overline{\langle a \rangle }) \cong \ZZ_2^\times.$$
     \end{thm} 
Every $\zeta \in  \overline{\langle a \rangle }$ can be written down as a sequence $\zeta = (a^{k_i})_{i \geq 1}$, where $k_i$ is odd, and $k_{i+1} \mod 2^i = k_i$. An element in $\ZZ^\times_2$ is represented by a sequence $(k_i)_{i \geq 1}$, where each $k_i$ is coprime to $2$, and so  $k_i$ is odd. Thus every equivalence class $G_\zeta$ is represented by a sequence $(k_i)_{i \geq 1} \in \ZZ_2^\times$. For instance, $e^\times = (1)_{i \geq 1}$.
      
 Our approach to investigating the densely settled property of the profinite arithmetic and geometric iterated monodromy groups is based on the study of the Weyl groups of maximal tori in $\overline G$. Recall from Definition \ref{defn-Weyl} that the Weyl group of a maximal torus $ \overline{\langle a \rangle }$ in $\overline G$ is the quotient 
 $$W( \overline{\langle a \rangle },\overline G) = (N( \overline{\langle a \rangle }) \cap \overline G) / \overline{\langle a \rangle }.$$
Similarly, one can define the Weyl group of $ \overline{\langle a \rangle }$ in the normalizer $N(\overline G)$ by
  $$W( \overline{\langle a \rangle },N(\overline G)) = (N( \overline{\langle a \rangle }) \cap N(\overline G)) / \overline{\langle a \rangle }.$$
By Theorem \ref{thm-main1} $N( \overline{\langle a \rangle })$ is densely settled, and by \cite[Proposition 3.9.5]{Pink13} we have that $N( \overline{\langle a \rangle }) \subset N(\overline{G})$, but a priori we do not know whether $N( \overline{\langle a \rangle })$ is contained in $\overline G$. If $g \in N(\overline G) - \overline G$ is a minimal element which is in the normalizer of $\overline G$ but not in $\overline G$, then we do not know if its normalizer $N( \overline{\langle g \rangle })$ is contained in $N(\overline G)$. 
Thus the indices of the Weyl groups $W( \overline{\langle a \rangle },\overline G)$ and $W( \overline{\langle a \rangle }, N(\overline G))$ in the absolute Weyl group $W( \overline{\langle a \rangle })$ contain information about settled elements in $\overline G$ and $N(\overline G)$. The index of the Weyl group $W(\overline{\langle g \rangle},N(\overline G))$ in the absolute Weyl group $W(\overline{\langle g \rangle})$ for a minimal $g \in N(\overline G) - \overline G$, contains information about settled elements in $N(\overline G)$.

Below, we prove Theorems \ref{thm-main2} - \ref{thm-main4}. We consider four cases for different values of $r$ and $s$, since the groups $\overline G$ and the normalizer $N(\overline G)$ have different descriptions in these cases.

 \subsection{The case $s = 1$ and $r = 2$}\label{sec-dihedral}
We prove Theorem \ref{thm-main2} and \ref{thm-main25}, which concern the case $r = 2$ and $s=1$.

For $r = 2$ and $s=1$, the group $G = \langle u_1,u_2\rangle = \langle \eta, (u_1,u_2) \rangle$ is the infinite dihedral group, and $a = u_1u_2$ is an element of infinite order, see for instance \cite{Lukina2021} for a detailed proof. Moreover, $a$ is transitive on each $V_n$, $n \geq 1$, and so it is minimal. 

Set $b = u_1$. Since $u_1$ and $u_2$ have order $2$, then
  $$b a b = u_1 u_1 u_2 u_1 = (u_1u_2)^{-1} = a^{-1},$$
and one obtains another standard presentation of the dihedral group, namely  
  \begin{align}\label{dihedral}G = \langle a,b \mid bab=a^{-1}, b^2 = 1\rangle,\end{align}
were $1$ denotes the identity element in $G$.

The proof of the first part of Proposition \ref{prop-dihedral-nds} below is based on the fact that the only settled elements in $G$ are the powers of $a$, and all other elements are periodic of period $2$.

\begin{prop}\label{prop-dihedral-nds}
The profinite group $\overline{G}$ is not densely settled, while the normalizer $N = N(\overline{G})$ is densely settled.
\end{prop}

\proof
Every element in $G$ can be presented as a word $b^m a^k$, for some $k \in \ZZ$ and $m \in \{0,1\}$. If $m= 1$, then 
  $$(ba^k)^2 = (ba^k)(ba^k) = b^2 a^{-k}a^k = 1,$$
so each such element has order $2$. Therefore, the only settled elements in $G$ are the powers of $a$.

By Corollary \ref{minimal-elements} the maximal torus $ \overline{\langle a \rangle }$ is isomorphic to the group of the dyadic integers $\ZZ_2$. Since every non-zero element in $\ZZ_2$ is of infinite order, $\overline{G}$ is strictly larger than $ \overline{\langle a \rangle }$. Let $g \in \overline{G} -  \overline{\langle a \rangle }$, and consider a sequence  $(b a^{k_n}) \subset G$ converging to $g$. Then for every $n \geq 1$ there exists $i_n \geq 1$ such that $g|V_n = ba^{k_i}|V_n$, for all $i \geq i_n$. It follows that for any point in $\partial T$, the projection of its orbit under the action of $\langle g \rangle $ onto $V_n$ has length $2$, and so $g$ is periodic with period $2$. Thus any element in $\overline{G} -  \overline{\langle a \rangle }$ is periodic of period $2$.

Since $ \overline{\langle a \rangle }$ is closed, every sequence of elements in $\overline{G}$ converging to $g \in \overline G -  \overline{\langle a \rangle }$ is eventually in $\overline{G} -  \overline{\langle a \rangle }$, which contains only periodic elements which are not settled. Thus $\overline{G}$ is not densely settled.

On the other hand, if $w\in N(\overline{G})$ then $waw^{-1}$ is minimal (Corollary \ref{Pink-result}), which implies that $waw^{-1}$ has infinite order and then $waw^{-1}\in  \overline{\langle a \rangle }$. It follows from Lemma   \ref{lemma-norm2} that $w\in N( \overline{\langle a \rangle })$, which implies that $N(\overline{G})\subseteq N( \overline{\langle a \rangle })$. Then $N( \overline{\langle a \rangle }) = N(\overline{G})$, and Theorem \ref{thm-Nads} implies that $N(\overline{G})$ is densely settled. 
\endproof

\proof \emph{(of Theorems \ref{thm-main2} and \ref{thm-main25})} Let $K$ be a number field, and let $f(x)$ be a quadratic polynomial with post-critical orbit of length $r=2$, which is strictly pre-periodic with $s = 1$. Then the profinite geometric iterated monodromy group $\fG_{\rm geom}(f)$ is conjugate to the group $\overline G$ in Proposition \ref{prop-dihedral-nds} and so it is not densely settled. 
The normalizer $N(\fG_{\rm geom}(f))$ is conjugate to $N(\overline{G})$ and so it is densely settled. This proves statements \eqref{st-one} and \eqref{st-two} in Theorem \ref{thm-main25}.

For any $a \in \fG_{\rm geom}(f)$ minimal we have for the Weyl groups 
  $$W( \overline{\langle a \rangle }) = N(\fG_{\rm geom}(f)) = W( \overline{\langle a \rangle },N(\fG_{\rm geom}(f))), \textrm{ and }W( \overline{\langle a \rangle }, \fG_{\rm geom}(f)) = \fG_{\rm geom}(f) / \overline{\langle a \rangle } = \{-1,1\}.$$
By \cite[Corollary 3.10.6(g)]{Pink13} $\fG_{\rm arith}(f)/\fG_{\rm geom}(f)$ has finite index in 
  $$\ZZ^\times_2 / \{-1,1\} \cong N(\fG_{\rm geom}(f))/\fG_{\rm geom}(f).$$ 
Thus $W( \overline{\langle a \rangle }, \fG_{\rm arith}(f))$ has finite index in $W( \overline{\langle a \rangle })$. This proves Theorem \ref{thm-main2}.

In particular, $\fG_{\rm arith}(f)$ is a finite index closed subgroup of $N(\fG_{\rm geom}(f))$. It follows that $\fG_{\rm arith}(f)$ is also an open subgroup, and so it is densely settled since $N(\fG_{\rm geom}(f))$ is densely settled. This proves statement \eqref{st-three}  in Theorem \ref{thm-main25}.
\endproof

\subsection{Normalizer of $\overline G$ for $r \geq 3$} We recall  from \cite{Pink13} the description of the normalizer of $\overline G$ in $\Aut(T)$ for different values of $r \geq 3$ and $s \geq 1$.

Denote by $C_n$ the additive cyclic group on $n$ symbols, and by $C_n^\times$ the multiplicative group of $C_n$. Recall that $\ZZ_2^\times \cong C_4^\times \times \ZZ_2$, with equivalence classes $[-1]$ and $[5]$ generating the respective factors \cite[Theorem 4.4.7]{RibZal10}. Composing the projection onto the first factor with the isomorphism from $C_4^\times$ to $C_2$, we obtain, as in \cite[Section 3.9]{Pink13}, the homomorphism
  \begin{align}\label{proj-maps} \theta_1: \mZ_2^\times \to C_2 :  & \, k  \mapsto \frac{k-1}{2} \mod 2. 
  \end{align}
Next, notice that the projection of the set of squares $\{k^2 \mid k \in \ZZ_2^\times\}$ to $C_8^\times$ is trivial, and the projection of the same set onto $C_{16}^\times$ is equal to the order $2$ subgroup $\{1,9\}$. Composing squaring with the projection onto $C^\times_{16}$ and the isomorphism of the image subgroup to $C_2$, one obtains the homomorphism
  \begin{align}\label{proj-maps1}
  \theta_2: \mZ_2^\times \to C_2 :  & \, k \mapsto \frac{k^2-1}{8} \mod 2.  
  \end{align}

\begin{lemma} \label{lemma-kernel}
We have the following:
\begin{enumerate}
\item $k = (k_i) \in \ker(\theta_1)$ if and only if $k_i = 4\ell_i+1$ for all $i \geq 2$.
\item $k = (k_i) \in \ker(\theta_2)$ if and only if either $k_i = 8 \ell_i +1$ or $k_i = 8\ell_i +7$, for all $i \geq 4$.
\end{enumerate}
\end{lemma}

We now consider different cases of $r \geq 3$ and $s \geq 1$. We recursively define the following elements:

$\bullet$~ {\bf (a) Case $r \geq 4$ and $s \geq 2$:} Set 
\begin{align}\label{eq-wis} w_1 = (u_s,u_s), && w_{i+1} = (w_i,w_i) \textrm{ for all } i\geq 1. \end{align}
Since $u_s$ has order $2$, every $w_i$ has order $2$. Also note that $w_i|V_k = id$ for $1 \leq k < i+s$, and so for a fixed $k \geq 1$, if $i > k-s$ then the restriction $w_i|V_k$ is trivial.

$\bullet$~{\bf (b) Case $r \geq 3$ and $s =1$:} Set 
\begin{align}\label{eq-wis1} w_1 = (1,(u_s u_r)^2), && w_{i+1} = (w_i,w_i) \textrm{ for all } i\geq 1. \end{align}
Since by \cite[Proposition 3.1.9]{Pink13} $u_s u_r$ has order $4$ then $(u_s u_r)^2$ has order $2$, and so every $w_i$ has order $2$. Also note that $w_i|V_k = id$ for $1 \leq k < i +1$, and so for a fixed $k \geq 1$, if $i > k -1$, then the restriction $w_i|V_k$ is trivial.

$\bullet$~{\bf (c) Case $r = 3$ and $s =2$:} Set 
\begin{align}\label{eq-wis2} w_0 = u_3 (w_0,w_0), && w_1 = (1,(u_s u_r)^2), && w_{i+1} = (w_i,w_i) \textrm{ for all } i\geq 1. \end{align}
Since by \cite[Proposition 3.1.9]{Pink13} $u_s u_r$ has order $4$, then $(u_s u_r)^2$ has order $2$, and so for $i \geq 1$ $w_i$ has order $2$. Also note that $w_i|V_k = id$ for $1 \leq k < i+2$, and so for a fixed $k \geq 1$, if $i \geq k-2$, then $w_i|V_k$ is trivial. Besides, by \cite[Lemma 3.8.18]{Pink13} $w_0^2 \in \overline G$.

In the cases (a) and (b) define the map
  \begin{align}\label{eq-varphiab}\varphi: \prod_{i \geq 1} C_2 \to N (\overline G): (t_1,t_2,\ldots) \mapsto w_1^{t_1} w_2^{t_2} \cdots,\end{align}
which is continuous but not a homomorphism. In the case (c) set for $i \geq 0$
  $$w_i^t = \left\{\begin{array}{ll} 1, & \textrm{ if } t = 0, \\ w_i & \textrm{ if }t = 1, \end{array} \right.$$
and then define the map (which we denote by the same symbol as the map in \eqref{eq-varphiab})
 \begin{align}\label{eq-varphic}\varphi: \prod_{i\geq 0} C_2 \to N(\overline G): (t_0,t_1,t_2,\ldots) \mapsto w_0^{t_0} w_1^{t_1} w_2^{t_2} \cdots,\end{align}
 which is continuous but not a homomorphism. In the cases (b) and (c) the image of $\varphi$ is in fact contained in a specific subgroup of $N(\overline G)$, but this is not important for our arguments.

\begin{thm}\label{thm-isonorm}
\begin{enumerate}
\item \cite[Lemma 3.6.9, Theorem 3.6.12]{Pink13}: In the case (a), where $r \geq 4$ and $s \geq 2$, we have $[w_i,w_j] \in \overline G$ for all $i,j \geq 1$, and so there is an induced isomorphism $\overline{\varphi}: \prod_{i=1}^\infty C_2 \to N(\overline G)/\overline{G}$.
\item \cite[Lemma 3.7.14, Theorem 3.7.18]{Pink13}: In the case (b), where $r \geq 3$ and $s =1 $, we have $[w_i,w_j] \in \overline G$ for all $i,j \geq 1$, and so there is an induced isomorphism $\overline{\varphi}: \prod_{i=1}^\infty C_2 \to N(\overline G)/\overline{G}$.
\item  \cite[Lemma 3.8.23, Theorem 3.8.27]{Pink13}: In the case (c), where $r = 3$ and $s =2 $, we have $[w_i,w_j], w_0^2 \in \overline G$ for all $i,j \geq 0$, and there is an induced isomorphism $\overline{\varphi}: \prod_{i=0}^\infty C_2 \to N(\overline G)/\overline{G}$.
\end{enumerate}
\end{thm}

Thus every coset in $N(\overline G)/ \overline G$ is represented by $w_1^{t_1} w_2^{t_2} \cdots$ where $t_i \in \{0,1\}$ for $i \geq 1$ in the cases (a) and (b). In the case (c) every coset $N (\overline G)/ \overline G$ is represented by $w_0^{t_0} w_1^{t_1} w_2^{t_2} \cdots$ where $t_i \in \{0,1\}$ for $i \geq 0$.

\subsection{Weyl groups of some maximal tori} In this section we study the Weyl groups of maximal tori contained in $\overline G$. 

We recall again our notation. For $a \in \Aut(T)$ we denote by $\langle a \rangle$ the cyclic group generated by $a$, and by $ \overline{\langle a \rangle }$ its closure in $\Aut(T)$. If $a \in \Aut(T)$ is minimal, then $ \overline{\langle a \rangle }$ is a maximal torus by Lemma \ref{minimal-elements}.

 We denote by $N( \overline{\langle a \rangle })$ the normalizer of $ \overline{\langle a \rangle }$ in $\Aut(T)$. We recall that the absolute Weyl group of $ \overline{\langle a \rangle }$ is the group $W( \overline{\langle a \rangle }) = N( \overline{\langle a \rangle })/ \overline{\langle a \rangle }$, and the relative Weyl group of $ \overline{\langle a \rangle }$ in $\overline G$ is the group $W( \overline{\langle a \rangle }, \overline G) = (N( \overline{\langle a \rangle })\cap \overline G)/ \overline{\langle a \rangle }$.

Choose a minimal element $a \in \overline G$ and fix an isomorphism $N( \overline{\langle a \rangle }) \cong \ZZ^\times_2 \ltimes \ZZ_2 $ given by Theorem \ref{thm-Zstabilizer}, then $W( \overline{\langle a \rangle }) \cong \ZZ_2^\times$.

$\bullet$~ {\bf (a) Case $r \geq 4$ and $s \geq 2$}. Composing the induced map $\overline \varphi$ in Theorem \ref{thm-isonorm} with the map
   $$(\theta_1, \theta_1, \ldots): \ZZ^\times_2 \to \prod_{i\geq1} C_2,$$
   where $\theta_1$ is defined by \eqref{proj-maps},
by \cite[Proposition 3.9.14]{Pink13} we have the map
   \begin{align}\label{weyl-24}\overline \varphi \circ (\theta_1, \theta_1, \ldots): \ZZ^\times_2 \to N(\overline{G})/\overline{G}: k  \mapsto \overline{\varphi}(k)\overline G = (w_1^{\theta_1(k)} w_2^{\theta_1(k)} \cdots )\overline G. \end{align}
 By \cite[Proposition 3.9.14]{Pink13} the image of the absolute Weyl group $W( \overline{\langle a \rangle }) \cong \ZZ^\times_2$ coincides with the image of \eqref{weyl-24}, 
 where the coset $G_{\zeta} \in W( \overline{\langle a \rangle })$ defined by \eqref{eq-layers} is identified with $k = (k_i) \in \ZZ^\times_2$ such that $\zeta = (a^{k_i})$, see the comments after Theorem \ref{thm-Zstabilizer}. Thus the image of $W( \overline{\langle a \rangle })$ in $\prod_{i \geq 1} C_2$ consists of two points, the infinite sequence of $0$'s and the infinite sequence of $1$'s, and the normalizer $N( \overline{\langle a \rangle })$ is contained in the union of two cosets of $\overline G$ in $N(\overline G)$. This gives the following.

\begin{lemma}\label{lemma-twoparts}
Let $r \geq 4$ and $s \geq 2$, and let $a\in \overline{G}$ be a minimal element. Then we have:
\begin{enumerate}
\item $N( \overline{\langle a \rangle })\cap \overline{G}=\overline{\{ \sigma_{m,4\ell+1}: m\geq 1,  \ell \geq 0\} }$.
\item \label{lemma-twoparts2} $W( \overline{\langle a \rangle }, \overline G)$ has index $2$ in the absolute Weyl group $W( \overline{\langle a \rangle })$.
\item If $b \in N( \overline{\langle a \rangle })$ is minimal, then $b \in N( \overline{\langle a \rangle }) \cap \overline G$. 
\end{enumerate}
\end{lemma}
\begin{proof}
Since $N( \overline{\langle a \rangle })$ and $\overline G$ are profinite groups, they are closed. Then their intersection is closed, and the first statement follows from Lemma \ref{lemma-kernel}(1) and the fact that the rational spanning set $\{\sigma_{m,k} \mid m \geq 1,k \geq 0\}$ is dense in $N( \overline{\langle a \rangle })$. The second statement follows from the fact that the image of \eqref{weyl-24} consists of two points. The last statement follows by Lemma \ref{minimal-final}(1) and Proposition \ref{prop-applminimal}.
\end{proof}

$\bullet$~ {\bf (b) Case $r \geq 3$ and $s =1$}. Composing the induced map $\overline \varphi$ in Theorem \ref{thm-isonorm} with the map
   $$(\theta_2, \theta_2, \ldots): \ZZ^\times_2 \to \prod_{i \geq 1} C_2,$$
where $\theta_2$ is defined by \eqref{proj-maps1}, by \cite[Proposition 3.9.14]{Pink13} we obtain the map
   \begin{align}\label{eq-phitheta2}\overline \varphi \circ (\theta_2, \theta_2, \ldots): \ZZ^\times_2 \to N(\overline G)/\overline{G}: k  \mapsto \overline{\varphi}(k)\overline G = (w_1^{\theta_2(k)} w_2^{\theta_2(k)} \cdots )\overline G. \end{align}
Again, using the identification of $W( \overline{\langle a \rangle })$ with $\ZZ^\times_2$, we obtain that the image of $W( \overline{\langle a \rangle })$ in $\prod_{i \geq 1} C_2$ consists of two points, the infinite sequence of $0$'s, and the infinite sequence of $1$'s, and $N( \overline{\langle a \rangle })$ is contained in the union of two cosets of $\overline G$ in $N(\overline G)$. More precisely, we have the following.

\begin{lemma}\label{lemma-twopartsr3s1}
Let $r \geq 3$ and $s =1$, and let $a\in \overline{G}$ be a minimal element. Then we have:
\begin{enumerate}
\item $N( \overline{\langle a \rangle })\cap \overline{G}=\overline{\{ \sigma_{m,8 \ell+1}, \sigma_{m,8 \ell+7}: m\geq 1, \ell\geq 0\} }$.
\item \label{lemma-twopartsr3s1-2} $W( \overline{\langle a \rangle }, \overline G)$ has index $2$ in the absolute Weyl group $W( \overline{\langle a \rangle })$.
\item Both cosets of $N( \overline{\langle a \rangle }) \cap \overline G$ in $N( \overline{\langle a \rangle })$ contain minimal elements of $N( \overline{\langle a \rangle })$. 
\end{enumerate}
\end{lemma}

\proof Since $N( \overline{\langle a \rangle })$ and $\overline G$ are profinite groups, they are closed. Then their intersection is closed, and the first statement follows from Lemma \ref{lemma-kernel}(2) and the fact that the rational spanning set $\{\sigma_{m,k} \mid m \geq 1, k \geq 0\}$ is dense in $N( \overline{\langle a \rangle })$. The second statement follows from the fact that the image of \eqref{eq-phitheta2} consists of two points. The last statement follows by Lemma \ref{minimal-final}(1) and Proposition \ref{prop-applminimal}.
\endproof

$\bullet$~ {\bf (c) Case $r = 3$ and $s =2$}. Composing the induced map $\overline \varphi$ in Theorem \ref{thm-isonorm} with the map
   $$(\theta_1, \theta_2, \theta_2, \ldots): \ZZ^\times_2 \to \prod_{i \geq 0} C_2,$$
where $\theta_1$ is defined by \eqref{proj-maps} and $\theta_2$ is defined by \eqref{proj-maps1}, by \cite[Proposition 3.9.14]{Pink13} we obtain the map
   \begin{align}\label{eq-compr3}\overline \varphi \circ (\theta_1,\theta_2, \theta_2, \ldots): \ZZ^\times_2 \to N/\overline{G}: k  \mapsto \overline{\varphi}(k)\overline G = (w_0^{\theta_1(k)} w_1^{\theta_2(k)} w_2^{\theta_2(k)} \cdots )\overline G. \end{align}
In this case the image of $W( \overline{\langle a \rangle })$ in $\prod_{i \geq 0} C_2$ consists of four points, the infinite sequences $0^\infty$, $0 1^\infty$, $1 0^\infty$ and $1^\infty$, where $x^\infty$ denotes the concatenation of an infinite number of symbols $x$. Then $N( \overline{\langle a \rangle })$ is contained in the union of four cosets of $\overline G$ in $N(\overline G)$. 

\begin{lemma}\label{lemma-twopartsr3s2}
Let $r = 3$ and $s =2$, and let $a \in \overline{G}$ be a minimal element. Then we have:
\begin{enumerate}
\item $N( \overline{\langle a \rangle })\cap \overline{G}=\overline{\{ \sigma_{m,8\ell+1}: m\geq 1, \ell\geq 0\} }$.
\item \label{lemma-twopartsr3s2-2} $W( \overline{\langle a \rangle }, \overline G)$ is a normal subgroup of index $4$ of the absolute Weyl group $W( \overline{\langle a \rangle })$.
\item The cosets of $N( \overline{\langle a \rangle }) \cap \overline G$ in $N( \overline{\langle a \rangle })$ represented by $\overline \varphi(0^\infty)$ and $\overline \varphi(01^\infty)$ contain minimal elements of $N( \overline{\langle a \rangle })$. 
\end{enumerate}
\end{lemma}

\begin{proof}
The kernel of \eqref{eq-compr3} is the intersection of normal subgroups $\ker \theta_1 \cap \ker \theta_2$, and so it is a subgroup of index $4$ in $W( \overline{\langle a \rangle })$ which is normal since $W( \overline{\langle a \rangle })$ is abelian. Then arguments similar to the ones in Lemma \ref{lemma-twopartsr3s1} finish the proof.
\end{proof}

Lemmas \ref{lemma-twoparts}, \ref{lemma-twopartsr3s1} and \ref{lemma-twopartsr3s2} show whether, for a given $a \in \overline G$, the maximal tori contained in the normalizer $N( \overline{\langle a \rangle })$ are also entirely contained in $\overline G$. In particular, Lemma \ref{lemma-twoparts} states that for $r \geq 4$ and $s \geq 2$ all maximal tori contained in $N( \overline{\langle a \rangle })$ are entirely contained in $\overline G$. Lemmas \ref{lemma-twopartsr3s1} and \ref{lemma-twopartsr3s2} state that when $r \geq 3$ and $s =1$, or when $r = 3$ and $s = 2$, then some minimal elements in $N( \overline{\langle a \rangle })$ are not in $\overline G$. We will see in Section \ref{subsec-toriincosets} that such elements topologically generate maximal tori, which are not contained in $\overline G$ but nevertheless have a non-empty intersection with $\overline G$. The normalizers of such maximal tori are studied in Section \ref{subsec-normingroup}.

\subsection{Maximal tori in the normalizer $N(\overline{G})$}\label{subsec-toriincosets} As earlier in the paper, for $a \in \Aut(T)$ we denote by $\langle a \rangle$ the cyclic group generated by $a$, and by $ \overline{\langle a \rangle }$ its closure in $\Aut(T)$. If $a \in \Aut(T)$ is minimal, then $ \overline{\langle a \rangle }$ is a maximal torus by Lemma \ref{minimal-elements}.

\begin{thm}\label{thm-moreodometers}
Let $\overline{G}$ be topologically generated by the elements $u_1,\ldots,u_r$ given by \eqref{eq-gen} and let $r \geq 3$. 
Then for every coset in $N(\overline G)/\overline G$ there exists a representative $w$ such that the following holds:
\begin{enumerate}
\item For every minimal $a \in  \overline{G}$, the element  $aw \in w \overline{G}$ is minimal. 
\item For the maximal torus $ \overline{\langle aw \rangle }$ topologically generated by $aw$, the intersection $ \overline{\langle aw \rangle } \cap \overline{G}$ is an index $2$ subgroup of $ \overline{\langle aw \rangle }$, and $ \overline{\langle aw \rangle } \subset \overline{G} \cup  w \overline{G}$. 
\item The normalizer $N(\overline G)$ contains an uncountable number of distinct maximal tori.
\item Minimal elements in $\overline G$ and in each coset $w \overline G \in N(\overline G)/\overline G$ are in bijective correspondence. 
\end{enumerate}
\end{thm}

\proof Following \cite[Section 1.5]{Pink13}, for $n \geq 1$ define a homomorphism
 \begin{align*} {\rm sgn}_n: Aut(T) \to \{\pm 1\}\end{align*} 
which assigns to each $g \in Aut(T)$ the parity of the permutation $g|V_n$. Recall that by \cite[Proposition 1.6.2]{Pink13} $a$ is minimal if and only if ${\rm sgn}_n(a) = -1$ for all $n \geq 1$. 
Note that if $g = (h_1,h_2)$, then ${\rm sgn}_n(g) = {\rm sgn}_n(h_1){\rm sgn}_n(h_2)$ since $h_1$ and $h_2$ induce permutations of disjoint sets of vertices on $V_n$ and ${\rm sgn}_n$ is a homomorphism. We start by proving statement (1).

$\bullet$ ~ {\bf (a) Case $s \geq 2$ and $r \geq 4$.} Let $t \in \prod_{i \geq 1} C_2$, and let $w = \varphi(t) \in N(\overline G)$, where $\varphi$ is given by \eqref{eq-varphiab}. By \eqref{eq-wis} we have
  $$w = w_1^{t_1} w_2^{t_2} \cdots = (u_s^{t_1} w_1^{t_2} \cdots, u_s^{t_1} w_1^{t_2} \cdots),$$
and in particular  
  $${\rm sgn}_n(w|V_n) = {\rm sgn}_n^2(u_s^{t_1} w_1^{t_2} \cdots) = 1.$$
Let $a \in \overline G$ be minimal, then for $aw$ we have for $n \geq 1$
  $${\rm sgn}_n(aw|V_n) = {\rm sgn}_n(a|V_n) \, {\rm sgn}_n(w|V_n)=  {\rm sgn}_n(a|V_n) = -1,$$
and so $aw \in w \overline G$ is a minimal element. Throughout the paper, we use the same notation $w \overline{G}$ for the left and the right cosets of $\overline G$ in $N(\overline G)$, since they are equal. 

$\bullet$ ~{\bf (b) Case $s =1$ and $r \geq 3$.} 
First note that for $n \geq 1$ we have
 $${\rm sgn}_n(w_1) = {\rm sgn}_n((1, (u_su_r)^2)) = ({\rm sgn}_n (u_s u_r))^2 = 1. $$
For $i \geq 2$ and $n \geq 1$ we have for $w_i$ defined by \eqref{eq-wis1}
 $${\rm sgn}_n(w_{i+1}) = {\rm sgn}_n((w_i,w_i)) = ({\rm sgn}_n(w_{i}))^2 = 1,$$
so clearly for any finite product ${\rm sgn}_n(w_{i_1} \cdots w_{i_k}) = 1$ for $n \geq 1$. Next, let $\varphi(t) = w = w_1^{t_1} w_2^{t_2} \cdots$ be an infinite product. Since for $i > n -1$ the restriction $w_i|V_n$ is a trivial permutation, then for $n \geq 1$
  $$w|V_n = w_{1}^{t_1} \cdots w_{n-1}^{t_{n-1}}|V_n,$$
and ${\rm sgn}_n w = 1$ for all $ n \geq 1$. Now let $a$ be a minimal element in $\overline G$. Then 
  $${\rm sgn}_n(aw) = {\rm sgn}_n(a) {\rm sgn}_n(w) = {\rm sgn}_n(a) = -1$$
for $n \geq 1$ and so $aw$ is minimal.

$\bullet$ ~{\bf (c) Case $s =2$ and $r = 3$.} 
The choice of a suitable $w$ is more complicated in this case. Since $w_i$ for $i \geq 1$ are defined by the same formulas as in Case (b), by the same argument as in (b) one obtains that ${\rm sgn}_n(w_i) = 1$ for all $i \geq 1$ and $n \geq 1$, and so ${\rm sgn}_n(w_1^{t_1} w_2^{t_2}\cdots)=1$ for any $n \geq 1$ and any combination of $t_i $, $i\geq 1$. However, it is easy to see that ${\rm sgn}_n(w_0) = -1$ at least for some $n \geq 0$. 

If for $t \in \prod_{i\geq 0} C_2$ we have $t_0 = 0$, then we set $w = \varphi(t)$ with $\varphi$ defined by \eqref{eq-varphic} and the argument showing that $aw$ is minimal if $a$ is minimal proceeds as in Case (b). So assume that $t_0 =1$ and so $\varphi(t) = w_0 w_1^{t_1} w_2^{t_2} \cdots$. Then for $n \geq 1$ we have
  $${\rm sgn}_n(\varphi(t)) = {\rm sgn}_n(w_0){\rm sgn}_n(w_1^{t_1} w_2^{t_2} \cdots)={\rm sgn}_n(w_0)= {\rm sgn}_n(u_3(w_0,w_0)) = {\rm sgn}_n(u_3).$$
Now set $w = u_3 \varphi(t)$. Since $w$ and $\varphi(t)$ differ by $u_3 \in \overline G$, they represent the same coset. For $n \geq 1$
  $${\rm sgn}_n(w) = {\rm sgn}_n(u_3) {\rm sgn}_n(\varphi(t)) = ({\rm sgn}_n(u_3))^2 = 1.$$
Now let $a$ be a minimal element in $\overline G$, then $aw \in \varphi(t)\overline G$, and $aw$ is minimal since for $n \geq 1$
  $${\rm sgn}_n(aw) = {\rm sgn}_n(a) {\rm sgn}_n(w) = {\rm sgn}_n(a) = -1.$$
This finishes the proof of statement (1).

We now prove statement (2). In our arguments we repeatedly use the fact that the left and the right cosets of $\overline G$ in $N(\overline G)$ are equal, so given $g \in \overline G$ and $w \in N(\overline G)$, there exists $g' \in \overline G$ such that $gw = wg'$.

The proof proceeds similarly for the cases (a) and (b). The case (c) has some differences and we prove it separately. In the proofs, $w$ is a representative of a coset in $N(\overline G)$ chosen as in the proof of statement (1) above.

{\bf Cases (a) and (b).} Consider the cyclic subgroup $\langle aw \rangle$. Let $m \geq 1$ be an integer. Note that for $i > n -s$ the restriction $w_i|V_n$ is a trivial permutation (recall that $s = 1$ in the case (b)). Then 
  $$(aw)^m |V_n = (a w_1^{t_1} \cdots w_{n-s}^{t_{n-s}})^m|V_n = (w_1^{t_1} \cdots w_{n-s}^{t_{n-s}} g_n)^m|V_n$$
for some $g_n \in \overline{G}$. Since for any $w_i,w_j$ the commutator $[w_i,w_j] = w_iw_jw_i^{-1}w_j^{-1} \in \overline{G}$ by \cite[Lemma 3.6.9, Lemma 3.7.14]{Pink13}, and the left and the right cosets of $\overline{G}$ in $N(\overline G)$ are equal, then
 $$(aw)^m|V_n = (w_1^{t_1} \cdots w_{n-s}^{t_{n-s}} g_n)^m|V_n = w_1^{mt_1} \cdots w_{n-s}^{mt_{n-s}} g_n ' |V_n$$
for some $g_n' \in \overline{G}$. Since $w_i$ has order $2$ for any $i \geq 1$, then for $m' \geq 1$
\begin{align*}(aw)^{2m'}|V_n = g_n'  \in \overline{G}, && (aw)^{2m'-1}|V_n =  w_1^{t_1} \cdots w_{n-s}^{t_{n-s}} g_n' \in w_1^{t_1} \cdots w_{n-s}^{t_{n-s}}\overline{G}. \end{align*}
We obtain that 
  $$(aw)^{2m'} = \lim_{n \to \infty} g_n',$$
and since $\overline{G}$ is a closed subgroup of $Aut(T)$, then $(aw)^{2m'} \in \overline{G}$. On the other hand,
 $$(a w)^{2m'-1} = \lim_{n \to \infty}w_1^{t_1} \cdots w_{n-s}^{t_{n-s}} g_n',$$
where the sequence on the right converges to a point in the coset $w \overline{G}$. Here we use that the cosets of $\overline G$ are closed since $\overline G$ is closed and group multiplication in $\Aut(T)$ is a homeomorphism. 

We have shown that $\langle aw \rangle \subset \overline{G} \cup w \overline{G}$, and the intersection $\langle aw \rangle \cap \overline G$ is an index two subgroup of $\langle aw \rangle$, which implies statement (2). 
Next note that the maximal tori intersecting distinct cosets in $N(\overline G)/\overline{G}$ are distinct. 
Since $\prod_{i \geq 1} C_2$ is uncountable and this space is in bijection with $N(\overline G)/\overline{G}$, there is an uncountable number of distinct maximal tori in $N(\overline G)$, and statement (3) follows as well. Finally note that $aw = a'w$ for $a,a' \in \overline G$ implies that $a = a'$, which proves statement (4).

{\bf Case (c).} Let $t \in \prod_{i = 0}^\infty C_2$. If $t_0 = 0$, then the proof proceeds precisely as in the case (b), so suppose that $t_0 = 0$. Then $ \varphi(t) = w_0 w_1^{t_1} w_2^{t_2} \cdots$. Let $w = u_3 \varphi(t)$, and consider the cyclic subgroup $\langle aw \rangle$. Since for $i \geq 1$ $w_i|V_n$ is trivial if $i > n-2$, and $u_3 \in \overline G$, we compute that
  $$(aw)^m |V_n = (a u_3 w_0 w_1^{t_1} \cdots w_{n-2}^{t_{n-2}})^m|V_n = w_0^m w_1^{mt_1} \cdots w_{n-2}^{mt_{n-2}} g_n'|V_n$$
for some $g_n' \in \overline G$. Since $w^2_0 \in \overline G$, we use the equality of left and right cosets of $\overline G$ to obtain
  $$(aw)^m |V_n = w_0^{t_0} w_1^{mt_1} \cdots w_{n-2}^{mt_{n-2}} g_n''|V_n$$
for $t_0 = m \mod 2$ and some $g_n'' \in \overline G$. Then for $m' \geq 1$
\begin{align*}(aw)^{2m'}|V_n = g_n''  \in \overline{G}, && (aw)^{2m'-1}|V_n =  w_0 w_1^{t_1} \cdots w_{n-s}^{t_{n-2}} g_n'' \in w_0 w_1^{t_1} \cdots w_{n-s}^{t_{n-2}}\overline{G}. \end{align*}
The rest of the proof proceeds as in cases (a) and (b).
\endproof

\proof \emph{(of Theorem \ref{thm-main3})}
Let $f(x)$ a quadratic polynomial with strictly pre-periodic post-critical orbit of length $r\geq 3$. By \cite[Proposition 1.7.15]{Pink13} the profinite geometric iterated monodromy group $\fG_{\rm geom}(f)$ associated to $f(x)$ is conjugate to the group generated by \eqref{eq-gen}, and it follows that its normalizer is conjugate to $N(\overline G)$. Moreover, the conjugation maps the cosets of $N(\overline G)$ to the cosets of $\fG_{\rm geom}(f)$ in its normalizer. Then the proof follows from Theorem \ref{thm-moreodometers}.
\endproof

\subsection{Weyl groups of maximal tori in $N(\overline G)$}\label{subsec-normingroup} Theorem \ref{thm-moreodometers} shows that the normalizer $N(\overline G)$ contains more maximal tori than given by Lemmas \ref{lemma-twoparts}, \ref{lemma-twopartsr3s1} and \ref{lemma-twopartsr3s2}, especially in the case $r \geq 4$ and $s \geq 2$. If a maximal torus $ \overline{\langle a \rangle }$ is contained in $\overline G$, we know that its normalizer $N( \overline{\langle a \rangle })$ is contained in $N(\overline G)$. In this section we formulate a criterion which ensures that the normalizer of a maximal torus, topologically generated by a minimal element and not entirely contained in $\overline G$, is in $N(\overline G)$.

As before, for $w\overline G \in N(\overline G)/\overline G$ denote by ${\rm Min}(w \overline G)$ the set of all minimal elements in $w \overline G$. 

\begin{prop}\label{prop-normalizer}
Let $\overline{G}$ be topologically generated by the elements $u_1,\ldots,u_r$ given by \eqref{eq-gen}, and let $r \geq 3$. 
Then there is a well-defined action 
  \begin{align}\label{eq-normaction}N(\overline G) \times {\rm Min}(w \overline G) \to {\rm Min}(w \overline G): (z,aw) \mapsto z (aw) z^{-1}.\end{align}
That is, if $z \in N(\overline G)$ satisfies $z a_1 z^{-1} = a_2$ for $a_1,a_2 \in {\rm Min}(\overline G)$, then we have in $w \overline G$
  \begin{align}\label{eq-zcommutator}z (a_1w) z^{-1} = (a_2 \, [z,w])\, w \in {\rm Min}(w \overline G). \end{align}
\end{prop}

\proof By \cite[Theorem 3.9.4]{Pink13} all minimal elements in $\overline G$ are conjugate under $N(\overline G)$, so the action \eqref{eq-normaction} is well-defined and transitive on ${\rm Min}(\overline G)$. We now show that it is well-defined on other cosets in $N(\overline G)$. 

Note that minimal elements in $\overline G$ and $w \overline G$ are in bijective correspondence by Theorem \ref{thm-main3}, so for any  $g \in {\rm Min}(w \overline G)$ and any representative $w$ satisfying Theorem \ref{thm-main3} we can find $a \in {\rm Min}(\overline G)$ such that $g = aw$. Since $aw$ is minimal, its conjugate $z(aw)z^{-1}$ is minimal, and we only must prove that this element is in the same coset as $aw$.

Recall that for the elements $w_i$, defined by \eqref{eq-wis}, \eqref{eq-wis1} or \eqref{eq-wis2}, their commutators are in $\overline G$ by \cite[Lemma 3.6.9, Lemma 3.7.14, Lemma 3.8.23]{Pink13}. We note that a similar statement holds for any two elements in $N(\overline G)$. Indeed, for any $w, \widehat w$ given by \eqref{eq-wis}, \eqref{eq-wis1} or \eqref{eq-wis2}, consider the elements $hw$ and $g \widehat w$ with $h,g \in \overline G$. Then, using the fact that the right cosets are equal to the left cosets we have for some $g' \in \overline G$
  \begin{align*} [hw,g \widehat w] & = (hw)(g \widehat w)(hw)^{-1} (g \widehat w)^{-1} = g' [w, \widehat w] \in \overline G. \end{align*}

Now suppose $z a_1z^{-1} = a_2$ for $a_1,a_2 \in {\rm Min}(\overline G)$ and $z \in N(\overline G)$. Multiplying $za_1 = a_2z$ on the left by $w$ and using the commutator relation, we obtain
$$z a_1 w = a_2 z w = (a_2 [z,w]) w z.$$
Further multiplying on the right by $z^{-1}$ gives \eqref{eq-zcommutator}. Since $a_2 [z,w] \in \overline G$, then $ (a_2 [z,w]) w \in w \overline G$.
\endproof

We now give a criterion for when the Weyl group of a maximal torus generated by an element in ${\rm Min}(N(\overline G))$ is contained in $N(\overline G)$.

\begin{prop}\label{prop-transitiveweyl}
Let $\overline{G}$ be topologically generated by the elements $u_1,\ldots,u_r$ given by \eqref{eq-gen}, and let $r \geq 3$. Let $aw \in {\rm Min}(w \overline G) \subset N(\overline G)$ and let $ \overline{\langle aw \rangle }$ be the corresponding maximal torus. If the action $N(\overline G) \times {\rm Min}(w \overline G) \to {\rm Min}(w \overline G)$ defined by \eqref{eq-normaction} is transitive, then we have the equality of the Weyl groups 
$$ W( \overline{\langle aw \rangle }, N(\overline G))  = W( \overline{\langle aw \rangle }).$$
\end{prop}

\proof The inclusion $W( \overline{\langle aw \rangle }, N(\overline G)) \subset W( \overline{\langle aw \rangle })$ holds by definition. We only have to show the reverse inclusion.

 Let $aw \in {\rm Min}(w \overline G)$, and let $g \in N( \overline{\langle aw \rangle })$. Then $g(aw)g^{-1} = ((aw)^k) \in  \overline{\langle aw \rangle }$ for some $k \in \ZZ_2^\times$. By Theorem \ref{thm-moreodometers} we have $((aw)^k) = mw$ for some $m \in {\rm Min}(\overline G)$. If the action  \eqref{eq-normaction} is transitive, then there exists $h \in N(\overline G) $ such that 
   $$h (aw) h^{-1} = mw = g(aw) g^{-1}.$$
This implies that $g^{-1}h$ commutes with $aw$. Then by Lemma \ref{abelian} $g^{-1}h \in  \overline{\langle a w\rangle }$. Thus $g$ is a product of elements in $N(\overline G)$ and must be in $N(\overline G)$. It follows that $W( \overline{\langle a w\rangle }) = W( \overline{\langle aw \rangle }, N(\overline G))$. 
\endproof

We finally collect the results proved in this section together to obtain a proof of Theorem \ref{thm-main35}.

\proof \emph{(of Theorem \ref{thm-main35})}
Let $f(x)$ be a quadratic polynomial with strictly pre-periodic post-critical orbit of length $r\geq 3$. By \cite[Proposition 1.7.15]{Pink13} the profinite iterated monodromy group $\fG_{\rm geom}(f)$ associated to $f(x)$ is conjugate to the group generated by \eqref{eq-gen}, and it follows that the normalizer $N(\fG_{\rm geom}(f))$ of $\fG_{\rm geom}(f)$ in $\Aut(T)$ is conjugate to $N(\overline G)$. Then the proof follows from Propositions \ref{prop-normalizer} and \ref{prop-transitiveweyl}.
\endproof

We note that Theorems \ref{thm-main3} and \ref{thm-main35} highlight the differences between the properties of maximal tori and Weyl groups in subgroups of $\Aut(T)$ and in compact connected Lie groups. For instance, it is known that the Weyl group of a maximal torus in a compact Lie group is always finite, all maximal tori are conjugate in the group, and every element in the group is contained in a maximal torus. In the theorem below, we show that similar statements need not hold for subgroups of $\Aut(T)$.

\begin{thm}\label{thm-differentthanLie}
Let $f(x)$ be a quadratic polynomial with strictly pre-periodic post-critical orbit of length $r\geq 3$. Let $\fG_{\rm geom}(f)$ be the associated profinite geometric iterated monodromy group, and let $N(\fG_{\rm geom}(f))$ be its normalizer in $\Aut(T)$. Then the following holds.
\begin{enumerate}
\item For a maximal torus $ \overline{\langle a \rangle } \subset \fG_{\rm geom}(f)$, the Weyl group $W( \overline{\langle a \rangle }, \fG_{\rm geom}(f))$ is infinite.
\item For $N(\fG_{\rm geom}(f))$ and $\fG_{\rm geom}(f)$, at least one of the following alternatives is true: either there are maximal tori in $\fG_{\rm geom}(f)$ which are not conjugate in $N(\fG_{\rm geom})$, or there are elements in $\fG_{\rm geom}(f)$ which are not contained in a maximal torus.
\item There are maximal tori in $N(\fG_{\rm geom}(f))$ which are not conjugate in $N(\fG_{\rm geom}(f))$. 
\end{enumerate}
\end{thm}

\proof For a maximal torus $ \overline{\langle a \rangle } \subset \fG_{\rm geom}(f)$, its absolute Weyl group $W( \overline{\langle a \rangle })$ is isomorphic to $\ZZ_2^\times$, which is an infinite group, see Theorem \ref{thm-Zstabilizer}. By Lemmas \ref{lemma-twoparts}, \ref{lemma-twopartsr3s1} and \ref{lemma-twopartsr3s2} for $r \geq 3$ $W( \overline{\langle a \rangle },\fG_{\rm geom}(f))$ has finite index in $W( \overline{\langle a \rangle })$ and so it must be infinite. This shows (1).

To see (2), note that if all maximal tori in $\fG_{\rm geom}(f)$ are conjugate, then they are necessarily generated by minimal elements. Since 
 $\fG_{\rm geom}(f)$ contains elements of order $2$ (for instance, elements conjugate to the generators $u_1,\ldots, u_r$ in \eqref{eq-gen}), it then cannot be the union of maximal tori.  On the other hand, if every element of $\fG_{\rm geom}(f)$ is contained in a maximal torus, then there must be maximal tori which are not the closures of cyclic groups generated by minimal elements. Such tori cannot be conjugate to the torus $ \overline{\langle a \rangle }$ topologically generated by a minimal element $ a $, since all elements in $ \overline{\langle a \rangle }$ are of infinite order. 
 
To see (3)  note that Proposition \ref{prop-normalizer} implies that a conjugate of any minimal $g \in N(\fG_{\rm geom}(f))$ by an element of $N(\fG_{\rm geom}(f))$ is contained in the same coset of $\fG_{\rm geom}(f)$ in $N(\fG_{\rm geom}(f))$ as $g$. By Theorem \ref{thm-moreodometers} there is an uncountable number of cosets in $N(\fG_{\rm geom}(f))/\fG_{\rm geom}(f)$ and each contains minimal elements. Minimal elements in different cosets are conjugate in $\Aut(T)$ but the conjugating element cannot be in $N(\fG_{\rm geom}(f))$.
\endproof

\subsection{The profinite arithmetic iterated monodromy group}\label{subsec-arith}

We prove Theorem \ref{thm-main4} which provides evidence that the conjecture by Boston and Jones (see Conjecture \ref{conj-BJ} and Problem \ref{prob-main}) holds for the profinite arithmetic and geometric iterated monodromy groups associated to quadratic PCF polynomials with strictly pre-periodic post-critical orbit of length at least $3$.

\proof \emph{(of Theorem \ref{thm-main4})}. Denote by $\overline K$ a separable closure of $K$, and by $\fG(\overline K/K)$ the absolute Galois group of $K$. 

By \cite[Theorem 3.10.2]{Pink13} there exists $\omega \in \Aut(T)$ such that $\fG_{\rm geom}(f) = \omega \overline G \omega^{-1}$ for $\overline G$ topologically generated by \eqref{eq-gen}, and the elements $\omega u_k \omega^{-1}$, where $u_k$ are given in \eqref{eq-gen}, provide a generating set of $\fG_{\rm geom}(f)$ up to a conjugation by (possibly different for different $u_k$) elements $\fG_{\rm geom}(f)$. Then by relabelling the vertices in the tree $T$ by the automorphism $\omega$ we may assume that $\fG_{\rm geom}(f) = \overline G$. Then there is a subgroup inclusion $\fG_{\rm arith}(f) \leq N(\fG_{\rm geom}(f))$.

By \cite[Section 1.7]{Pink13} there is a surjective homomorphism
  $$\overline \rho: \fG(\overline K/K) \to \fG_{\rm arith}(f)/\fG_{\rm geom}(f) \subset N(\fG_{\rm arith}(f))/\fG_{\rm geom}(f).$$
By \cite[Lemma 3.10.4]{Pink13} $\fG_{\rm geom}(f)$ contains a minimal element $b_\infty$, which can be obtained, for instance, as the product of the generators of $\fG_{\rm geom}(f)$, and for any $\tau \in \fG(\overline K/K)$ we have $\overline \rho(\tau) = w \fG_{\rm geom}(f)$ for some $w \in N(\fG_{\rm geom}(f))$ satisfying $w b_\infty w^{-1} = b^{c(\tau)}_\infty$, where $c: \fG(\overline K/K) \to \ZZ_2^\times$ is the cyclotomic character. By \cite[Theorem 3.10.5]{Pink13} the profinite group $\fG_{\rm arith}(f)$ is contained in the image of the map \eqref{weyl-24} if $r \geq 4$ and $s \geq 2$, the map \eqref{eq-phitheta2} if $r \geq 3$ and $s = 1$, and the map \eqref{eq-compr3} if $r = 3$ and $s=2$. 

 Then the statement that for any $w \in  \fG_{\rm arith}(f)$ the sets of minimal elements ${\rm Min}(\fG_{\rm geom}(f))$ and ${\rm Min}( w\fG_{\rm geom}(f))$ are in bijective correspondence follows from the subgroup inclusion $\fG_{\rm arith}(f) \leq N(\fG_{\rm geom}(f))$ and Theorem \ref{thm-main3},\eqref{thm3-4}. 
 
 The statement that for any $a \in {\rm Min}(\fG_{\rm geom}(f))$ the relative Weyl group $W(\overline{\langle a \rangle},(\fG_{\rm geom}(f)))$ has finite index in the absolute Weyl group $W(\overline{\langle a \rangle})$ is given by Lemma \ref{lemma-twoparts},\eqref{lemma-twoparts2}, for $r \geq 4$ and $r \geq 2$, Lemma \ref{lemma-twopartsr3s1}, \ref{lemma-twopartsr3s1-2}, for $r \geq 3$ and $s = 1$, and Lemma \ref{lemma-twopartsr3s2},\ref{lemma-twopartsr3s2-2}, for $r=3$ and $s=2$. Thus we have substantial evidence that the conjecture by Boston and Jones holds for $\fG_{\rm geom}(f)$.

Finally, by \cite[Corollary 3.10.6]{Pink13} $\fG_{\rm geom}(f)$ has finite index in $\fG_{\rm arith}(f)$. The inclusion of groups 
  $$W(\langle a \rangle,\fG_{\rm geom}(f)) \leq W(\langle a \rangle,\fG_{\rm arith}(f)) \leq W(\overline{\langle a \rangle})$$
implies that the relative Weyl group $W(\langle a \rangle,\fG_{\rm arith}(f))$ has finite index in the absolute Weyl group $W(\overline{\langle a \rangle})$.  Depending on the field $K$, there may also be the equalities of the relative Weyl groups $W(\langle a \rangle,\fG_{\rm geom}(f)) = W(\langle a \rangle,\fG_{\rm arith}(f))$, or of the absolute and relative Weyl groups $W(\langle a \rangle) = W(\langle a \rangle,\fG_{\rm arith}(f))$, or neither. Thus we have substantial to good evidence that the conjecture by Boston and Jones holds for $\fG_{\rm arith}(f)$.
\endproof


\end{document}